\documentclass[english]{article}

\usepackage[margin=1in]{geometry}
\usepackage{hyperref}
\usepackage{amsthm,amsfonts,amssymb,amsmath}
\usepackage{graphicx}
\usepackage{tikz}
\usetikzlibrary{arrows}

\newtheorem{theorem}{Theorem}[section]
\newtheorem{proposition}[theorem]{Proposition}
\newtheorem{lemma}[theorem]{Lemma}

\newtheorem{claim}[theorem]{Claim}
\newtheorem{corollary}[theorem]{Corollary}

\theoremstyle{definition}
\newtheorem{definition}[theorem]{Definition}

\newtheorem{procedure}[theorem]{Procedure}

\numberwithin{equation}{section}

\usepackage{babel}
\usepackage[T1]{fontenc}
\usepackage[utf8]{inputenc}
\allowdisplaybreaks

\newcommand{\eps}{\varepsilon}
\newcommand{\su}{\subseteq}
\newcommand{\sm}{\setminus}
\newcommand{\Cay}{\operatorname{Cayley}}
\newcommand{\ind}{\operatorname{ind}}
\newcommand{\emb}{\operatorname{emb}}
\newcommand{\aut}{\operatorname{aut}}
\newcommand{\into}{\hookrightarrow}
\newcommand{\tH}{\widetilde{H}}
\newcommand{\tk}{\tilde{k}}
\newcommand{\kap}{\kappa}
\renewcommand{\l}{\ell}
\renewcommand{\P}{\operatorname{\mathbb{P}}}
\renewcommand{\theta}{\vartheta}
\renewcommand{\phi}{\varphi}
\newcommand{\E}{\operatorname{E}}
\renewcommand{\log}{\ln}

\newcommand{\liniendicke}{0.5pt}

\begin{document}
\title{On the inducibility problem for random Cayley graphs of abelian groups with a few deleted vertices}
\author{Jacob Fox\thanks{Department of Mathematics, Stanford University, Stanford, CA 94305. Email: {\tt jacobfox@stanford.edu}. Research supported by a Packard Fellowship and by NSF award DMS-185563.} \and Lisa Sauermann\thanks{Department of Mathematics, Stanford University, Stanford, CA 94305. Email: {\tt lsauerma@stanford.edu}.} \and Fan Wei\thanks{Department of Mathematics, Stanford University, Stanford, CA 94305. Email: {\tt fanwei@stanford.edu}.}}

\maketitle

\begin{abstract}
\noindent
Given a $k$-vertex graph $H$ and an integer $n$, what are the $n$-vertex graphs with the maximum number of induced copies of $H$? This question is closely related to the inducibility problem introduced by Pippenger and Golumbic in 1975, which asks for the maximum possible fraction of $k$-vertex subsets of an $n$-vertex graph that induce a copy of $H$. Huang, Lee and the first author proved that for a random $k$-vertex graph $H$, almost surely the $n$-vertex graphs maximizing the number of induced copies of $H$ are the balanced iterated blow-ups of $H$. In this paper, we consider the case where the graph $H$ is obtained by deleting a small number of vertices from a random Cayley graph $\widetilde{H}$ of an abelian group. We prove that in this case, almost surely all $n$-vertex graphs maximizing the number of induced copies of $H$ are balanced iterated blow-ups of $\widetilde{H}$.
\end{abstract}

\section{Introduction}

\subsection{Background}

In 1975 Pippenger and Golumbic \cite{pippenger-golumbic} asked the following question: Given a $k$-vertex graph $H$, what is the maximum number of $k$-vertex subsets in an $n$-vertex graph that induce a copy of $H$? Denoting this number by $\ind(H,n)$ we have $0\leq \ind(H,n)\leq \binom{n}{k}$. Pippenger and Golumbic  defined the \emph{inducibilty} of $H$ as $\ind(H)=\lim_{n\to\infty} (\ind(H,n)/\binom{n}{k})$. This limit exists because $\ind(H,n)/\binom{n}{k}$ is monotone decreasing in $n$ (for $n\geq k$). Note that $0\leq \ind(H)\leq 1$.

In general, determining $\ind(H)$ is very hard and the precise value is only known for a few (explicit) classes of graphs $H$. Pippenger and Golumbic \cite{pippenger-golumbic} showed that for every $k$-vertex graph $H$ the lower bound
\begin{equation}\label{ind-gen-lower-bound}
\ind(H)\geq \frac{k!}{k^{k}-k}
\end{equation}
holds. This lower bound can be obtained from considering balanced iterated blow-ups of $H$, which we will formally define below. Huang, Lee and the first author \cite{fox-huang-lee} proved that for a randomly chosen graph $H$ this bound is almost surely tight. In fact, for a random graph $H$, they proved that the $n$-vertex graphs with the maximum number of induced copies of $H$ are precisely the balanced iterated blow-ups of $H$.  Independently, Yuster \cite{yuster} obtained the latter result for $n\leq 2^{\sqrt{k}}$ and concluded that almost surely $\ind(H)\leq (1+o_k(1))\cdot k!/(k^k-k)$ for a random $k$-vertex graph $H$.

Let us now define balanced iterated blow-ups, see also \cite{fox-huang-lee}. Given a graph $H$ with at least two vertices, a \emph{blow-up} of $H$ is a graph $\Gamma$ whose vertex set can be partitioned into non-empty subsets $W_i$ for $i\in V(H)$ such that for distinct $i,j\in V(H)$ the graph $\Gamma$ is complete between $W_i$ and $W_j$ if $i$ and $j$ are adjacent in $H$ and otherwise $\Gamma$ is empty between $W_i$ and $W_j$. The graph $\Gamma$ is a \emph{balanced blow-up} of $H$, if the sets $W_i$ can be chosen in such a way that their sizes differ by at most one. We call $\Gamma$ a \emph{balanced iterated blow-up} of $H$, if $\vert V(\Gamma)\vert<\vert V(H)\vert$ or if it is a balanced blow-up of $H$ where for each of the subsets $W_i$ the induced subgraph on $W_i$ is again a balanced iterated blow-up of $H$. Thus, a balanced iterated blow-up is a fractal-like construction, as an example see Figure \ref{fig1} showing a balanced iterated blow-up of the cycle $C_6$. Note that if $H$ has $k$ vertices and $n$ is not a power of $k$, there can be different graphs on $n$ vertices which are balanced iterated blow-ups of $H$. However, if $H$ is prime (see Definition \ref{defi-prime}), for each $n$, all $n$-vertex balanced iterated blow-ups of $H$ have the same number of induced copies of $H$.

\begin{figure}
\begin{center}
\begin{tikzpicture}[scale=2]
\input{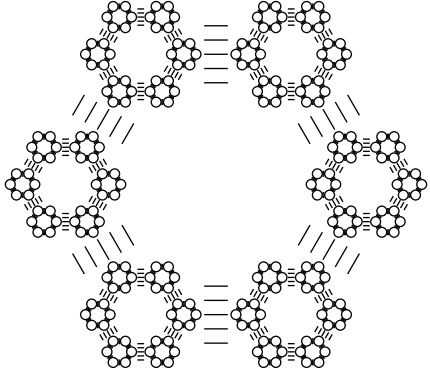}
\end{tikzpicture}
\end{center}
\caption{A balanced iterated blow-up of the cycle $C_6$.}
\label{fig1}
\end{figure}

The problem of determining $\ind(H)$ when $H$ is a path or cycle has received much attention. For cycles $C_k$ with $k\geq 5$, Pippenger and Golumbic \cite{pippenger-golumbic} conjectured that (\ref{ind-gen-lower-bound}) is sharp, in other words $\ind(C_k)=k!/(k^{k}-k)$. For $k\geq 6$, the currently best known upper bound is $\ind(C_k)\leq 2k!/k^{k}$ proved by Kr\'{a}l', Norin, and Volec \cite{kral-norin-volec} improving on earlier bounds by Pippenger and Golumbic \cite{pippenger-golumbic} and by Hefetz and Tyomkyn \cite{hefetz-tyomkyn}. For $k=5$, Balogh, Hu, Lidick\'{y}, and Pfender \cite{balogh-hu-lidicky-pfender} proved $\ind(C_k)= 5!/(5^{5}-5)$ using the Flag algebra method.

For paths, however, the situation is significantly different. Exoo \cite{exoo} observed that for a path $P_k$ on $k\geq 4$ vertices, for large $n$ an $n$-vertex balanced iterated blow-up of the cycle $C_{k+1}$  contains more induced copies of $P_k$ than an $n$-vertex balanced iterated blow-up of  $P_{k}$ itself. This is because in the blow-up of the cycle one can ``rotate'' the path in different ways (and this overcompensates the fact that the parts of the blow-up are slightly smaller). In particular, the lower bound (\ref{ind-gen-lower-bound}) is not sharp for paths $P_k$ on $k\geq 4$ vertices. For $k=4$ and $k=5$, even better constructions were obtained by Even-Zohar and Linial \cite{even-zohar-linial}, building upon a construction of Exoo \cite{exoo} for $k=4$. For large $k$, the balanced iterated blow-up of $C_{k+1}$ is the best known construction. The question of determining $\ind(P_k)$ is still open for all $k\geq 4$.

Note that a path $P_k$ on $k$ vertices can be obtained from deleting one vertex from a cycle $C_{k+1}$ on $k+1$ vertices. The reason for  balanced iterated blow-ups of $C_{k+1}$ having more induced copies of $P_k$ than balanced iterated blow-ups of $P_{k}$ is that the cycle $C_{k+1}$ has a lot of symmetries (and so there are many ways to embed $P_{k}$ into $C_{k+1}$). Put in a different way, $C_{k+1}$ is a Cayley graph of the abelian group $\mathbb{Z}/(k+1)\mathbb{Z}$ with generator 1. Let us consider more generally the situation where $H$ is a graph obtained by deleting a few vertices from a Cayley graph $\tH$ of an abelian group. Then, if $H$ is prime (see Definition \ref{defi-prime}), which is usually the case, the $n$-vertex balanced iterated blow-ups of $\tH$ contain more induced copies of $H$ than $n$-vertex balanced iterated blow-ups of $H$ (as long as we deleted sufficiently few vertices). 

\subsection{Results}

In this paper, we study the case where $\tH$ is a random Cayley graph of an abelian group and $H$ is obtained from $\tH$ by deleting a few vertices. We will show that in this case, almost surely all $n$-vertex graphs maximizing the number of induced copies of $H$ are balanced iterated blow-ups of $\tH$ (see Theorem \ref{thm-main-ind-H-structure} below).

Formally, a Cayley graph of an abelian group is defined as follows.

\begin{definition}\label{defi-cayley} Given an abelian group $G$ and a subset $\Lambda\su G\sm\lbrace 0\rbrace$ with $\Lambda=-\Lambda$, the \emph{Cayley graph} $\Cay(G,\Lambda)$ is the graph with vertex set $G$ in which two vertices $x,y\in G$ are connected if and only if $x-y\in \Lambda$.
\end{definition}

Note that due to $\Lambda=-\Lambda$ and $0\not\in \Lambda$, this is indeed a well-defined undirected graph (without loops).

For a given abelian group $G$ (additively written) and $0<p<1$, we construct a random Cayley graph with vertex set $G$ by choosing a subset $\Lambda\su G\sm\lbrace 0\rbrace$ randomly as follows.

\begin{procedure}\label{procedure-lambda} Let us choose a random subset $\Lambda\su G\sm\lbrace 0\rbrace$ by including each $\lbrace g, -g\rbrace\su G\sm\lbrace 0\rbrace$ into $\Lambda$ independently with probability $p$.
\end{procedure}

Note that  the random set $\Lambda$ by construction always satisfies $\Lambda=-\Lambda$.

Now, we are ready to state our main result. Roughly speaking, it states that for a graph $H$ which is obtained from a random Cayley graph of an abelian group by deleting a few vertices, the $n$-vertex graphs $\Gamma$ maximizing the number of induced copies of $H$ are balanced iterated blow-ups of the Cayley graph $\tH$.

\begin{theorem}\label{thm-main-ind-H-structure}Let $G$ be an abelian group with $\tk$ elements and assume that $0<p<1$ satisfies $\min(p,1-p)\geq 10^6(\log \tk)^{6/5}\tk^{-1/5}$. If $\Lambda\su G$ is chosen according to Procedure \ref{procedure-lambda}, then with probability $1-o(1)$ the Cayley graph $\tH=\Cay(G,\Lambda)$ satisfies the following: For every induced subgraph $H$ of $\tH$ on $k\geq \tk-\frac{1}{4}\log \tk$ vertices and for all $n$, every $n$-vertex graph $\Gamma$ with the maximum number of induced copies of $H$ is a balanced iterated blow-up of $\tH$.
\end{theorem}
In Theorem \ref{thm-main-ind-H-structure} and throughout this paper, the $o(1)$-term tends to zero as $\tk\to\infty$, independently of $p$.

Note that not all balanced iterated blow-ups of $\tH$ on $n$ vertices have precisely the same number of induced copies of $H$, if $H$ is a proper induced subgraph of $\tH$. However, if the number of vertices $n$ is a power of $\tk$, then the balanced iterated blow-up of $\tH$ is unique (up to isomorphism).

If $\tH$ is as in Theorem \ref{thm-main-ind-H-structure}, but $H$ is an induced subgraph on $k= \tk-\lceil\log \tk\rceil-1$ vertices, then almost surely for large $n$, the $n$-vertex balanced iterated blow-ups of $H$ contain more induced copies of $H$ than the $n$-vertex balanced iterated blow-ups of $\tH$. Therefore, the assumption $k\geq \tk-\frac{1}{4}\log \tk$ in Theorem \ref{thm-main-ind-H-structure} is tight up to the constant factor $\frac{1}{4}$.

Theorem \ref{thm-main-ind-H-structure} gives a strong structural result for the $n$-vertex graphs maximizing the number of induced copies of $H$. Using this, it is also possible to determine the inducibility of $H$, if we know the number of automorphisms of $H$.

In order to make this more clear, let us consider a second graph invariant, which was also introduced by Pippenger and Golumbic \cite{pippenger-golumbic} and is closely connected to the inducibility. A graph embedding $H\into \Gamma$ is an injective map $V(H)\into V(\Gamma)$ that sends edges of $H$ to edges of $\Gamma$ and non-edges of $H$ to non-edges of $\Gamma$. A graph automorphism of $H$ is an embedding of $H$ into itself. Let $\emb(H,\Gamma)$ be the number of embeddings $H\into \Gamma$ and set $\emb(H,n)=\max_\Gamma \emb(H,\Gamma)$, where the maximum is taken over all $n$-vertex graphs $\Gamma$. It is easy to see that $\emb(H,n)=\ind(H,n)\cdot \aut(H)$, where $\aut(H)=\emb(H,H)$ denotes the number of graph automorphisms of $H$ (note that every induced copy of $H$ in some graph $\Gamma$ gives precisely $\aut(H)$ embeddings $H\into \Gamma$). In particular, we can define
\[\emb(H)=\lim_{n\to\infty} \frac{\emb(H,n)}{k!\cdot \binom{n}{k}}=\lim_{n\to\infty} \frac{\emb(H,n)}{n^k}\]
and obtain $0\leq \emb(H)\leq 1$ and $\emb(H)=\ind(H)\cdot \aut(H)/k!$.

Pippenger and Golumbic \cite{pippenger-golumbic} introduced this invariant $\emb(H)$ without naming it, and there does not seem to exist a standard name in the literature. Let us therefore call $\emb(H)$ the \emph{embedding-inducibility} of $H$. Since $\emb(H)=\ind(H)\cdot \aut(H)/k!$, the problem of determining the inducibility and the embedding-inducibility of a graph $H$ are equivalent up to determining the number $\aut(H)$ of automorphisms of $H$. Furthermore, for every $n$, the $n$-vertex graphs $\Gamma$ maximizing $\emb(H,\Gamma)$ are the same as those maximizing the number of induced copies of $H$. In particular, Theorem \ref{thm-main-ind-H-structure} is equivalent to an identical structural result for the $n$-vertex graphs $\Gamma$ maximizing $\emb(H,\Gamma)$. In addition, we can also calculate the embedding-inducibility $\emb(H)$ using this structural result.

\begin{theorem}\label{thm-main-ind-H}Under the assumptions of Theorem \ref{thm-main-ind-H-structure}, with probability $1-o(1)$ the Cayley graph $\tH=\Cay(G,\Lambda)$ satisfies the following: For every induced subgraph $H$ of $\tH$ on $k\geq \tk-\frac{1}{4}\log \tk$ vertices, and for all $n$, every $n$-vertex graph $\Gamma$ maximizing $\emb(H,\Gamma)$ is a balanced iterated blow-up of $\tH$. Furthermore, if the group $G$ is not of the form $(\mathbb{Z}/2\mathbb{Z})^m$, then the embedding-inducibility $\emb(H)$ of every such induced subgraph $H$ is
\[\emb(H)=\frac{2}{\tk^{k-1}-1},\]
whereas if $G=(\mathbb{Z}/2\mathbb{Z})^m$ for some positive integer $m$, then the embedding-inducibility $\emb(H)$ of every such induced subgraph $H$ is
\[\emb(H)=\frac{1}{\tk^{k-1}-1}.\]
\end{theorem}

If we know the number $\aut(H)$ of automorphisms of $H$, then from Theorem \ref{thm-main-ind-H} we can easily compute the value of $\ind(H)=\emb(H)\cdot k!/\aut(H)$.

If we take $H\neq \tH$ in Theorem \ref{thm-main-ind-H-structure}, then with probability $1-o(1)$ we have $\aut(H)\leq 2(\tk-k)\leq \frac{1}{2}\log \tk$  (this follows from the fact that with probability $1-o(1)$ the graphs $\tH$ and $H$ satisfy the conditions in Definition \ref{defi-reasonable}). Hence we obtain
\[\ind(H)=\emb(H)\cdot \frac{k!}{\aut(H)}\geq \frac{1}{\tk^{k-1}-1}\cdot \frac{k!}{\ln k}.\]
Note that the term on the right-hand side is larger by a factor of roughly $k\cdot e^{-(\tk-k)}/\ln k\geq k^{3/4-o(1)}$ than the term $k!/(k^k-k)$ in (\ref{ind-gen-lower-bound}), which is obtained from taking a balanced iterated blow-up of $H$ itself. This means that for large $n$, $n$-vertex balanced iterated blow-ups of $\tH$ indeed contain many more induced copies of $H$ than $n$-vertex balanced iterated blow-ups of $H$ itself.

When taking $H=\tH$ in Theorem \ref{thm-main-ind-H-structure}, we see that for $\min(p,1-p)\geq 10^6(\log \tk)^{6/5}\tk^{-1/5}$ almost surely the $n$-vertex graphs $\Gamma$ maximizing the number of induced copies of $\tH$ are the balanced iterated blow-ups of $\tH$. In the terminology of \cite{fox-huang-lee} this means that $\tH$ is almost surely a fractalizer.

We also prove for $\min(p,1-p) \geq 10^3(\log \tk)^{1/2}\cdot \tk^{-1/5}$  that almost surely we have $\aut(\tH)=2\tk$ if the group $G$ is not of the form $(\mathbb{Z}/2\mathbb{Z})^m$, and $\aut(\tH)=\tk$ if $G=(\mathbb{Z}/2\mathbb{Z})^m$ for some positive integer $m$ (see Corollary \ref{coro-aut-tH}). Thus, using $\ind(\tH)=\emb(\tH)\cdot \tk!/\aut(\tH)$, Theorem \ref{thm-main-ind-H} gives the following corollary about the inducibility of the random Cayley graph $\tH$.

\begin{corollary}\label{tH-inducibility}Under the assumptions of Theorem \ref{thm-main-ind-H-structure}, with probability $1-o(1)$ the Cayley graph $\tH=\Cay(G,\Lambda)$ has inducibility
\[\ind(\tH)=\frac{\tk!}{\tk^{\tk}-\tk}.\]
\end{corollary}

We remark that the results for random graphs of Huang, Lee and the first author \cite{fox-huang-lee} do not apply for the graphs $H$ we are considering in the theorems above. This is because they only consider graphs that are far from having any symmetries, but in our case by construction $H$ is close to having symmetries coming from symmetries of $\tH$ (namely rotations and reflections, see the beginning of Section \ref{sect-overview}). In fact, even the answer is different, because in our case the optimal structures are not balanced iterated blow-ups of $H$, but  of the Cayley graph $\tH$. 

Our proof begins with some of the ideas in \cite{fox-huang-lee}. However, many of the techniques in \cite{fox-huang-lee} do not apply in our situation. Therefore a new approach is required here.

We remark that we did not optimize the absolute constants in our statements in this paper.

\subsection{Notation}

All $o(1)$-terms tend to zero as $\tk\to\infty$, independently of $p$. All groups and graphs in this paper are assumed to be finite.

Some authors only use the term Cayley graph, if the set $\Lambda$ in Definition \ref{defi-cayley} generates the group $G$, which is equivalent to the graph $\Cay(G,\Lambda)$ being connected. We will not make this restriction. However, for the range of $p$ we are considering in this paper, with probability $1-o(1)$ the set $\Lambda$ will be a generating set of $G$ anyway (see Corollary \ref{coro-lambda-generate}). So this distinction is not important. 

For a graph $H$ and non-empty disjoint subsets $X, Y\su V(H)$, we define the \emph{adjacency} $a_H(X,Y)$ as $a_H(X,Y)=1$ if there are at least $\frac{1}{2}\vert X\vert \vert Y\vert$ edges between $X$ and $Y$ and $a_H(X,Y)=0$ otherwise (so $a_H(X,Y)$ is the density between $X$ and $Y$ rounded to 1 or 0). By a slight abuse of notation, we write $a_H(x,Y)$ instead of $a_H(\lbrace x\rbrace,Y)$ for a vertex $x\not\in Y$ and $a_H(x,y)$ for $a_H(\lbrace x\rbrace,\lbrace y\rbrace)$ for two distinct vertices $x,y\in V(H)$. Note that $a_H(x,y)=1$ if $x$ and $y$ are adjacent and $a_H(x,y)=0$ otherwise, so $a_H(x,y)$ is simply an indicator for the usual vertex adjacency of $x$ and $y$. For a graph $H$ and a vertex $x\in V(H)$, we write $N(x)\su V(H)\sm \lbrace x\rbrace$ for the neighborhood of $x$.

\begin{definition}\label{defi-prime}A graph $H$ is called \emph{prime} if there is no subset $U\su V(H)$ with $2\leq \vert U\vert < \vert V(H)\vert$ such that every vertex outside $U$ is either complete to $U$ or has no edges to $U$.
\end{definition}

Note that in particular, every prime graph on at least three vertices is connected.

\section{Proof Overview}\label{sect-overview}

First, note that Theorem \ref{thm-main-ind-H} implies Theorem \ref{thm-main-ind-H-structure}, because the $n$-vertex graphs $\Gamma$ maximizing $\emb(H,\Gamma)$ are the same as those maximizing the number of induced copies of $H$. Also, we will see below that Corollary \ref{tH-inducibility} follows from Theorem \ref{thm-main-ind-H} using Corollary \ref{coro-aut-tH}. Thus, the main task will be to prove Theorem \ref{thm-main-ind-H}.

For a Cayley graph $\tH=\Cay(G,\Lambda)$ of an abelian group $G$ of size $\tk$, let us define the following automorphisms of $\tH$, which we will call the \emph{rotations} and \emph{reflections} of $\tH$. Recall that $\tH$ has vertex set $G$. The rotations of $\tH$ are given by $G\to G$, $x\mapsto x+g$ for some fixed $g\in G$. The reflections of $\tH$ are given by $G\to G$, $x\mapsto -x+g$ for some fixed $g\in G$. It is easy to see that these maps on the vertex set of $\tH$ indeed define automorphisms of $\tH$. Note that $\tH$ has exactly $\tk$ rotations and exactly $\tk$ reflections. If the group $G$ is not of the form $(\mathbb{Z}/2\mathbb{Z})^m$, then the rotations and reflections are all distinct, giving $2\tk$ different automorphisms of $\tH$. If $G=(\mathbb{Z}/2\mathbb{Z})^m$ for some positive integer $m$, then the set of rotations agrees with the set of reflections, and the rotations and reflections only give $\tk$ different automorphisms of $\tH$.

Our approach to proving Theorem \ref{thm-main-ind-H} is to show that almost surely $\tH$ satisfies the properties in the following definition with appropriate choices of the parameters.

\begin{definition}\label{defi-typical}For $0<q_0<\frac{1}{2}$ and $0<\delta_0<1$, a Cayley graph $\tH=\Cay(G,\Lambda)$ of an abelian group $G$ of size $\tk$ is called \emph{$(q_0,\delta_0)$-typical}, if it satisfies the following conditions:
\begin{itemize}
\item[(i)]For each vertex $v\in V(\tH)$, its degree $\deg(v)$ satisfies $q_0\tk\leq \deg(v)\leq (1-q_0)\tk$.
\item[(ii)]For any two distinct vertices $v,w\in V(\tH)$, there are at least $q_0\tk$ vertices in $V(\tH)\sm\lbrace v,w\rbrace$ that are adjacent to exactly one of the vertices $v$ and $w$.
\item[(iii)]There do not exist disjoint subsets $X,Y\su V(\tH)$ with sizes $\vert X\vert\geq 2\tk^{4/5}$ and $\vert Y\vert\geq 2\tk^{4/5}$ such that between the sets $X$ and $Y$ the graph $\tH$ is complete or empty.
\item[(iv)]For every subset $X\su V(\tH)$ of size $\vert X\vert\geq (1-\delta_0)\tk$ and every injective map $f:X\to V(\tH)$ with $a_{\tH}(v,w)=a_{\tH}(f(v),f(w))$ for all distinct $v,w\in X$, the following holds: Among the $\tk$ rotations and $\tk$ reflections of $\tH$, there exists some $\phi$ such that $f=\phi\vert_X$.
\end{itemize}
\end{definition}

Note that as long as $q_0>2\tk^{-1/5}$, every $(q_0,\delta_0)$-typical Cayley graph $\tH$ on $\tk$ vertices is connected. Indeed, by (i), every connected component has size at least $q_0\tk\geq 2\tk^{4/5}$, and therefore (iii) implies that there cannot exist two different connected components.

The following theorem states that, if $\Lambda\su G$ is chosen randomly according to Procedure \ref{procedure-lambda}, then $\tH=\Cay(G,\Lambda)$ is almost surely $(q_0,\delta_0)$-typical (for an appropriate choice of $q_0$ and $\delta_0$).

\begin{theorem}\label{thm-tH-typcial-whp}Let $G$ be an abelian group with $\tk$ elements and let $0<p<1$. Assume that $p'=\min(p,1-p)$ satisfies $p'\geq 10^3(\log \tk)^{1/2}\cdot \tk^{-1/5}$.
If $\Lambda\su G$ is chosen according to Procedure \ref{procedure-lambda}, then with probability $1-o(1)$ the Cayley graph $\tH=\Cay(G,\Lambda)$ is $(\frac{p'}{50},\frac{p'}{100})$-typical.
\end{theorem}
We will prove Theorem \ref{thm-tH-typcial-whp} in Section \ref{sect-tH-typcial-whp}.

Note that for $p$ and $p'$ as in Theorem \ref{thm-tH-typcial-whp}, we have $\frac{p'}{50}\geq 2\tk^{-1/5}$, so every $(\frac{p'}{50},\frac{p'}{100})$-typical Cayley graph is connected. Recalling that $\tH=\Cay(G,\Lambda)$ is connected if and only if $\Lambda$ is a generating set for $G$, we can immediately deduce the following corollary from Theorem \ref{thm-tH-typcial-whp}.

\begin{corollary}\label{coro-lambda-generate}Under the assumptions of Theorem \ref{thm-tH-typcial-whp}, with probability $1-o(1)$ the set $\Lambda$ will be a generating set of $G$.
\end{corollary}

We remark that in fact one only needs the much weaker lower bound $p\geq \Omega((\log \tk)\cdot \tk^{-1})$ in order for $\Lambda$ to be generating set of $G$ with probability $1-o(1)$. This is not hard to prove directly, but can also be deduced from a result of Pomerance \cite{pomerance} about the expectation of the number of independently picked random elements of an abelian group that are needed to generate the group.

Furthermore, note that taking $X=V(\tH)$ in condition (iv) in Definition \ref{defi-typical} implies that a $(q_0,\delta_0)$-typical Cayley graph $\tH$ on $\tk$ vertices has no additional automorphisms besides its $\tk$ rotations and $\tk$ reflections. In other words, for every $(q_0,\delta_0)$-typical Cayley graph $\tH=\Cay(G,\Lambda)$ on $\tk$ vertices (for any $0<q_0<\frac{1}{2}$ and $0<\delta_0<1$), we have that $\aut(\tH)=2\tk$ if $G$ is not of the form $(\mathbb{Z}/2\mathbb{Z})^m$, and $\aut(\tH)=\tk$ if $G=(\mathbb{Z}/2\mathbb{Z})^m$ for some positive integer $m$. Hence Theorem \ref{thm-tH-typcial-whp} also implies the following corollary.

\begin{corollary}\label{coro-aut-tH}Under the assumptions of Theorem \ref{thm-tH-typcial-whp}, with probability $1-o(1)$ the Cayley graph $\tH=\Cay(G,\Lambda)$ satisfies $\aut(\tH)=2\tk$ if $G$ is not of the form $(\mathbb{Z}/2\mathbb{Z})^m$, and $\aut(\tH)=\tk$ if $G=(\mathbb{Z}/2\mathbb{Z})^m$ for some positive integer $m$.
\end{corollary}

Note that Corollary \ref{tH-inducibility} follows directly from Theorem \ref{thm-main-ind-H} and Corollary \ref{coro-aut-tH}, using the fact that $\ind(\tH)=\emb(\tH)\cdot \tk!/\aut(\tH)$.


In Theorem \ref{thm-main-ind-H} we want to eventually prove a statement about large induced subgraphs $H$ of $\tH$. The following definition captures useful properties for $H$.

\begin{definition}\label{defi-reasonable}Let $\tH$ be a Cayley graph of an abelian group of size $\tk$ and let $H$ be an induced subgraph of $\tH$ on $k\geq \tk-\frac{1}{4}\log\tk$ vertices. For $0<q<\frac{1}{2}$ and $0<\delta<1$, we call $H$ a \emph{$(q,\delta)$-reasonable} induced subgraph of $\tH$, if it satisfies the following conditions:
\begin{itemize}
\item[(a)]$H$ is prime.
\item[(b)]For any two distinct vertices $v,w\in V(\tH)$, there are at least $qk$ vertices in $V(H)\sm\lbrace v,w\rbrace$ that are adjacent to exactly one of the vertices $v$ and $w$.
\item[(c)]For every subset $X\su V(\tH)$ of size $\vert X\vert\geq (1-\delta)k$ and every injective map $f:X\to V(\tH)$ with $a_{\tH}(v,w)=a_{\tH}(f(v),f(w))$ for all distinct $v,w\in X$, the following holds: Among the $\tk$ rotations and $\tk$ reflections of $\tH$, there exists some $\phi$ such that $f=\phi\vert_X$.
\end{itemize}
\end{definition}

\begin{lemma}\label{lem-typical-reasonable}Let $\tH$ be a $(q_0,\delta_0)$-typical Cayley graph of an abelian group of size $\tk$ and suppose that $0<q<\frac{1}{2}$ and $0<\delta<1$ are such that
\[q\leq q_0-\frac{\log \tk}{4\tk}\quad\text{and}\quad \delta\leq \delta_0-\frac{\log \tk}{4\tk}\quad\text{and}\quad q_0\geq 4\tk^{-1/5}+\frac{\log \tk}{4\tk}.\]
Then every induced subgraph $H$ of $\tH$ on at least $\tk-\frac{1}{4}\log\tk$ vertices is $(q,\delta)$-reasonable.
\end{lemma}
\begin{proof}Let $H$ be an induced subgraph of $\tH$ on $k\geq \tk-\frac{1}{4}\log\tk$ vertices. We need to check conditions (a), (b) and (c).

For (a), suppose there is a subset $X\su V(H)$ of size $2\leq \vert X\vert<k$, so that every vertex in $V(H)\sm X$ is either complete or empty to $X$. Let $v,w\in X$ be distinct and let us count the number of vertices in $V(\tH)\sm\lbrace v,w\rbrace$ that are adjacent to exactly one of the vertices $v$ and $w$. Note that no vertex in $V(H)\sm X$ has this property, hence this number is at most $\tk-(k-\vert X\vert)\leq \vert X\vert+\frac{1}{4}\log \tk$. On the other hand, by condition (ii) of $\tH$ this number needs to be at least $q_0\tk\geq 4\tk^{4/5}+\frac{1}{4}\log \tk$. So we can conclude $\vert X\vert\geq 4 \tk^{4/5}$. Furthermore, by condition (i) for $\tH$ any vertex in $V(H)\sm X$ can be complete or empty to at most $(1-q_0)\tk$ vertices. Hence
\[\vert X\vert\leq (1-q_0)\tk\leq \left(1-4\tk^{-1/5}-\frac{\log \tk}{4\tk}\right)\tk=\tk-4\tk^{4/5}-\frac{1}{4}\log \tk\leq k-4\tk^{4/5}\]
and therefore $\vert V(H)\sm X\vert\geq 4\tk^{4/5}$. In particular, we can choose a subset $Y\su V(H)\sm X$ of size $\vert Y\vert\geq 2\tk^{4/5}$ such that between $X$ and $Y$ the graph $H$ (and therefore also the graph $\tH$) is complete or empty. Recalling $\vert X\vert\geq 4 \tk^{4/5}$, this gives a contradiction to condition (iii) of $\tH$.

For (b), note that by condition (ii) of $\tH$ there exist at least $q_0\tk\geq q\tk+\frac{1}{4}\log \tk$ vertices in $V(\tH)\sm \lbrace v,w\rbrace$ that are adjacent to exactly one of the vertices $v$ and $w$. As $\vert V(\tH)\sm V(H)\vert =\tk-k\leq \frac{1}{4}\log\tk$, at least $q\tk\geq qk$ of these vertices lie in $V(H)\sm \lbrace v,w\rbrace$.

For (c), note that every subset $X\su V(H)$ with $\vert X\vert\geq (1-\delta)k$ satisfies
\[\vert X\vert\geq k-\left(\delta_0-\frac{\log \tk}{4\tk}\right)k\geq k-\left(\delta_0-\frac{\log \tk}{4\tk}\right)\tk=k+\frac{1}{4}\log \tk-\delta_0\tk\geq (1-\delta_0)\tk.\]
So (c) follows directly from condition (iv) for $\tH$.\end{proof}

The following theorem states that under certain restrictions for the parameters $q$ and $\delta$, every $(q,\delta)$-reasonable induced subgraph $H$ of $\tH$ satisfies the conclusion of Theorem \ref{thm-main-ind-H} (we will see below that this part of the conclusion actually implies the correct value for $\emb(H)$).

\begin{theorem}\label{thm-H-reasonable-blow-up}Let $\tk\geq 10^{200}$ and let $\tH$ be a Cayley graph of an abelian group of size $\tk$. Let $H$ be a $(q,\delta)$-reasonable induced subgraph of $\tH$ on $k\geq \tk-\frac{1}{4}\log \tk$ vertices with $q\geq 10^4(\log k)^{6/5}k^{-1/5}$ and  $\delta\geq 10^3(\log k)^{1/5}k^{-1/5}$. Then for all $n$, every $n$-vertex graph $\Gamma$ maximizing $\emb(H,\Gamma)$ is a balanced iterated blow-up of $\tH$.
\end{theorem}

We will prove Theorem \ref{thm-H-reasonable-blow-up} in Sections \ref{sect-prep-proof-blow-up} to \ref{sect-optimization}. The following claim characterizes how the embeddings $H\into \Gamma$ in Theorem \ref{thm-H-reasonable-blow-up} actually look like when $H$ is a blow-up of $\tH$.

\begin{claim}\label{claim-emb-into-blow}Let $H$ be a $(q,\delta)$-reasonable induced subgraph of some Cayley graph $\tH$ (for some $0<q<\frac{1}{2}$ and $0<\delta<1$). If $\Gamma$ is a blow-up of $\tH$ with parts $W_i$ for $i\in V(\tH)$, then every embedding $H\into \Gamma$ is either mapping all of $H$ into a single part $W_i$ or there exists a rotation or reflection $\phi$ of $\tH$ such that each vertex $x\in V(H)$ is mapped into $W_{\phi(x)}$.
\end{claim}
\begin{proof}
Let $\theta:H\into \Gamma$ be an embedding and assume that $\theta$ does not map all of $H$ into a single part $W_i$. Note that for each $i\in V(\tH)$, in the graph $H$ each vertex outside $\theta^{-1}(W_i)$ is either complete or empty to $\theta^{-1}(W_i)$. Recall that $H$ is prime by condition (a) in Definition \ref{defi-reasonable}. Since we assumed $\vert \theta^{-1}(W_i)\vert<\vert V(H)\vert$, this implies that $\vert \theta^{-1}(W_i)\vert\leq 1$ for every $i\in V(\tH)$. Hence there is an injective map $f:V(H)\to V(\tH)$ such that $\theta(x)\in W_{f(x)}$ for all $x\in V(H)$. Note that for all distinct $x,y\in V(H)$ we have $a_{\tH}(f(x),f(y))=a_{\Gamma}(\theta(x),\theta(y))=a_H(x,y)$. Thus, by property (c) in Definition \ref{defi-reasonable}, there must be a rotation or reflection $\phi$ of $\tH$ such that $f(x)=\phi(x)$ for all $x\in V(H)$. Then $\theta(x)\in W_{\phi(x)}$ for all  $x\in V(H)$, as desired.
\end{proof}

It is not hard to deduce Theorem \ref{thm-main-ind-H} from Theorem \ref{thm-tH-typcial-whp}, Lemma \ref{lem-typical-reasonable}, Theorem \ref{thm-H-reasonable-blow-up} and Claim \ref{claim-emb-into-blow}.

\begin{proof}[Proof of Theorem \ref{thm-main-ind-H}] Let $G$ be an abelian group of size $\tk\geq 10^{200}$, let $p'=\min(p,1-p)$ and assume that $p'\geq 10^6(\log \tk)^{6/5}\tk^{-1/5}$. By Theorem \ref{thm-tH-typcial-whp}, with probability $1-o(1)$ the Cayley graph $\tH=\Cay(G,\Lambda)$ is $(\frac{p'}{50},\frac{p'}{100})$-typical. We claim that whenever this happens, $\tH$ also satisfies the conclusion of Theorem \ref{thm-main-ind-H}. In order to check this, assume that $\tH$ is $(\frac{p'}{50},\frac{p'}{100})$-typical and let $H$ be an induced subgraph of $\tH$ on $k\geq \tk-\frac{1}{4}\log\tk$ vertices. Note that $p'\geq 10^6(\log \tk)^{6/5}\tk^{-1/5}\geq 10^6(\log \tk)/\tk$ and therefore
\[\frac{p'}{75}\leq \frac{p'}{50}-\frac{\log \tk}{4\tk}\quad\text{and}\quad \frac{p'}{150}\leq \frac{p'}{100}-\frac{\log \tk}{4\tk}\quad\text{and}\quad \frac{p'}{50}\geq 4\tk^{-1/5}+\frac{\log \tk}{4\tk}.\]
So by Lemma \ref{lem-typical-reasonable}, $H$ is a $(\frac{p'}{75},\frac{p'}{150})$-reasonable induced subgraph of $\tH$. Note that by $\tk \geq k\geq \tk-\frac{1}{4}\log\tk\geq \frac{3}{4}\tk$ we have
\[\frac{p'}{75}\geq \frac{4}{3}\cdot 10^4(\log \tk)^{6/5}\tk^{-1/5}\geq 10^4(\log k)^{6/5}\left(\frac{3}{4}\tk\right)^{-1/5}\geq 10^4(\log k)^{6/5}k^{-1/5}\]
and 
\[\frac{p'}{150}\geq \frac{4}{3}\cdot 10^3(\log \tk)^{6/5}\tk^{-1/5}\geq 10^3(\log k)^{1/5}\left(\frac{3}{4}\tk\right)^{-1/5}\geq 10^3(\log k)^{1/5}k^{-1/5}.\]
Thus, by Theorem \ref{thm-H-reasonable-blow-up} for all $n$, every $n$-vertex graph $G$ maximizing $\emb(H,G)$ is a balanced iterated blow-up of $\tH$.

It only remains to show the desired formula for $\emb(H)$. Let us first assume that the group $G$ is not of the form $(\mathbb{Z}/2\mathbb{Z})^m$. Then we need to show that $\emb(H)=2/(\tk^{k-1}-1)$.

If $n=\tk^\ell$ for some positive integer $\ell$, then there is a unique balanced iterated blow-up $\Gamma$ of $\tH$ on $n$ vertices. We claim that the number of embeddings $H\into \Gamma$ is $2\tk^\ell(\tk^{(k-1)(\ell-1)}+\tk^{(k-1)(\ell-2)}+\dots+1)$. Let us prove this by induction. If $\ell=1$, then $\Gamma=\tH$ and by condition (c) in Definition \ref{defi-reasonable} applied to $X=V(H)$, the only embeddings $H\into \tH$ are the restrictions of the $2\tk$ rotations and reflections of $\tH$. So the number of embeddings $H\into \Gamma$ is indeed $2\tk=2\tk^1\cdot 1$, if $\ell=1$.

Now let us assume $\ell>1$, then $\Gamma$ consists of $\tk$ parts, each of which has size $\tk^{\ell-1}$ and is itself a balanced iterated blow-up of $\tH$. By Claim \ref{claim-emb-into-blow}, each embedding $H\into \Gamma$ of must either map all vertices of $H$ into the same part of the blow-up or it must map each vertex of $H$ into a prescribed part according to one of the $2\tk$ rotations or reflections of $\tH$. There are $2\tk\cdot(\tk^{\ell-1})^k=2\tk^\ell\cdot \tk^{(k-1)(\ell-1)}$ embeddings $H\into \Gamma$ of the second kind (there are $2\tk$ choices for the rotation or reflection, and then for each vertex of $H$ there are $\tk^{\ell-1}$ potential images to choose from inside the part prescribed by the rotation or reflection). For the first kind, we just need to count the number of embeddings from $H$ into each of the parts. By induction, this number is $2\tk^{\ell-1}(\tk^{(k-1)(\ell-2)}+\dots+1)$ for each part and there are $\tk$ parts. Thus, all in all the number of embeddings $H\into \Gamma$ is indeed
\[2\tk^\ell\cdot \tk^{(k-1)(\ell-1)}+\tk\cdot 2\tk^{\ell-1}(\tk^{(k-1)(\ell-2)}+\dots+1)=2\tk^\ell(\tk^{(k-1)(\ell-1)}+\tk^{(k-1)(\ell-2)}+\dots+1).\]
Now we obtain
\begin{multline*}
\emb(H)=\lim_{\ell\to\infty}\frac{\emb(H,\tk^\ell)}{(\tk^\ell)^k}=\lim_{\ell\to\infty}\frac{2\tk^\ell(\tk^{(k-1)(\ell-1)}+\tk^{(k-1)(\ell-2)}+\dots+1)}{\tk^{k\ell}}\\
=\lim_{\ell\to\infty}\frac{2}{\tk^{k-1}}\left(1+\tk^{-(k-1)}+\dots+\tk^{-(\ell-1)(k-1)}\right)=\frac{2}{\tk^{k-1}}\cdot \frac{1}{1-\tk^{-(k-1)}}=\frac{2}{\tk^{k-1}-1},
\end{multline*}
as desired.

In the case where $G=(\mathbb{Z}/2\mathbb{Z})^m$ for some positive integer $m$, the argument is very similar. The only difference is that in this case we have only $\tk$ different rotations and reflections (because the $\tk$ rotations agree with the $\tk$ reflections). Hence all factors of $2$ disappear from the previous calculation, and in this case we obtain $\emb(H)=1/(\tk^{k-1}-1)$.
\end{proof}

Thus, it remains to prove Theorems \ref{thm-tH-typcial-whp} and \ref{thm-H-reasonable-blow-up}, which we will do in Section \ref{sect-tH-typcial-whp} and Sections \ref{sect-prep-proof-blow-up} to \ref{sect-optimization}, respectively.

\section{Preparations for the proof of Theorem \ref{thm-H-reasonable-blow-up}}\label{sect-prep-proof-blow-up}

Sections \ref{sect-prep-proof-blow-up} to \ref{sect-optimization} are devoted to the proof of Theorem \ref{thm-H-reasonable-blow-up}, with the main argument being given in Section \ref{sect-proof-blow-up}. In this section, we will prepare ourselves for the proof by fixing the set-up and collecting some useful tools.

\subsection{Set-up}

For Sections \ref{sect-prep-proof-blow-up} to \ref{sect-optimization}, let $\tH=\Cay(G,\Lambda)$ be a fixed Cayley graph of an abelian group $G$ of size $\tk\geq 10^{200}$. As in Theorem \ref{thm-H-reasonable-blow-up}, let $H$ be a $(q,\delta)$-reasonable induced subgraph of $\tH$ on $k\geq \tk-\frac{1}{4}\log \tk$ vertices, and assume that $q\geq 10^4(\log k)^{6/5}k^{-1/5}$ and $\delta\geq 10^3(\log k)^{1/5}k^{-1/5}$.

It is not difficult to check that the assumption $\tk\geq 10^{200}$ implies $\tk^{1/40}\geq 100\log \tk$ and therefore (using $k\geq \frac{1}{2}\tk$)
\begin{equation}\label{ineq-k-logtk}
k^{1/20}\geq \frac{1}{2}\cdot \tk^{1/20}\geq 5\cdot 10^3\cdot (\log \tk)^2.
\end{equation}

From (\ref{ineq-k-logtk}) we can see that $10^4(\log k)^{6/5}k^{-1/5}< \frac{1}{2}k^{-1/10}\leq \tk^{-1/10}\leq 10^{-20}$ and
\[10^{-2}\cdot \frac{10^4(\log k)^{6/5}k^{-1/5}}{\log (1/(10^4(\log k)^{6/5}k^{-1/5}))}<10^{-2}\cdot \frac{10^4(\log k)^{6/5}k^{-1/5}}{\log (k^{1/10})}=10^3(\log k)^{1/5}k^{-1/5}\leq \delta.\]
Furthermore, note that Definition \ref{defi-reasonable} is monotone in $q$ (if it holds for some value of $q$, the it also holds for all smaller values). By decreasing $q$, we can and will from now on assume that
\begin{equation}\label{ineq-q-lower-bound}
10^{-20}\geq q\geq 10^4(\log k)^{6/5}k^{-1/5}
\end{equation}
and 
\begin{equation}\label{ineq-delta-q}
\delta>10^{-2}\frac{q}{\log(1/q)}.
\end{equation}

From (\ref{ineq-q-lower-bound}), we in particular obtain
\begin{equation}\label{eq-ineq-q-k1}
\frac{\log k}{q}<\frac{k^{1/5}}{10^4}<\frac{\sqrt{k}}{10^4}<\frac{k}{10^4}.
\end{equation}

It might be easiest to follow the ideas of the proof when thinking of $q$ as a small constant. However, the proof works more generally for $q\geq 10^4(\log k)^{6/5}k^{-1/5}$. Before we go into the main part of the argument for proving Theorem \ref{thm-H-reasonable-blow-up} in the Section \ref{sect-proof-blow-up}, we devote the rest of this section to some preparatory lemmas.

\subsection{Signatures}

Huang, Lee and the first author \cite{fox-huang-lee} made the following definition, which will also be an important concept in our proof. In the literature, sets with the property in Definition \ref{defi-signature} are more commonly called locating sets.

\begin{definition}[\cite{fox-huang-lee}]\label{defi-signature} A signature of the graph $H$ is a subset $S\su V(H)$ such that $N(v)\cap S\neq N(w)\cap S$ whenever $v,w\in V(H)\sm S$ are distinct vertices.
\end{definition}

Note that whenever $S$ is a signature of $H$, and $S\su S'\su V(H)$, then $S'$ is a signature of $H$ as well. The following lemma is the reason why signatures are helpful for our purposes.

\begin{lemma}\label{lemma-signature-disjoint-sets}
Let $S$ be a signature of the graph $H$, and suppose we are given a function $f: S\to V(\Gamma)$ for some graph $\Gamma$. Let $H'$ be an induced subgraph of $H$ such that $S\su V(H')$. For every vertex $i\in V(H')\sm S$, let us define $V_i\su V(\Gamma)$ to be the set of vertices $v\in V(\Gamma)$ for which there exists at least one embedding $\theta: H'\into \Gamma$ with $\theta(s)=f(s)$ for all $s\in S$ and $\theta(i)=v$. Then the sets $V_i$ for $i\in V(H')\sm S$ are disjoint.
\end{lemma}

\begin{proof}
Note that for every $i\in V(H')\sm S$ and every $v\in V_i$, the adjacencies between $v$ and the vertices $f(s)$ for $s\in S$ must be the same as the adjacencies between $i$ and the vertices $s\in S$. But these adjacencies uniquely determine $i$, as $N(i)\cap S\neq N(j)\cap S$ for any vertex $j\in V(H')\sm S$ with $j\neq i$. Therefore a vertex $v$ cannot simultaneously lie in $V_i$ and $V_j$ for $i\neq j$. Thus, the sets $V_i$ for $i\in V(H')\sm S$ are indeed disjoint.
\end{proof}

\begin{lemma}\label{lemma-large-signature} Every subset $S\su V(H)$ of size $\vert S\vert \geq (1-q)k$ is a signature of $H$.
\end{lemma}

\begin{proof}Suppose that that for two distinct vertices $v,w\in V(H)\sm S$ we have $N(v)\cap S= N(w)\cap S$. Then all vertices in $S$ have the same adjacency to $v$ and $w$, so there could be at most $k-2-\vert S\vert <qk$ vertices in $V(H)\sm \lbrace v,w\rbrace$ that are adjacent to exactly one of the vertices $v$ and $w$. This would contradict condition (b) in Definition \ref{defi-reasonable}.
\end{proof}

The following lemma is very similar to \cite[Lemma 2.5]{fox-huang-lee}.

\begin{lemma}[see Lemma 2.5 in \cite{fox-huang-lee}]\label{lemma-small-signature}  For every subset $X\su V(H)$ of size $\vert X\vert \geq (1-\frac{q}{2})k$, there is a signature $S$ of $H$ of size $\vert S\vert \leq \frac{5}{q}\log k$ with $S\su X$.
\end{lemma}

\begin{proof} Let $t=\lfloor \frac{5}{q}\log k\rfloor$ and note that $t\geq  \frac{4}{q}\log k$. We choose the set $S\su X$ by taking $t$ independent random elements of $X$ (possibly with repetition). For any pair of distinct vertices $v,w\in V(H)$, by condition (b) in Definition \ref{defi-reasonable}, there are at most $(1-q)k$ vertices in $X\sm \lbrace v,w\rbrace\su V(H)\sm \lbrace v,w\rbrace$ with the same adjacency to $v$ and $w$. So the probability that each of the $t$ chosen vertices has the same adjacency to $v$ and $w$ is at most
\[\left(\frac{(1-q)k}{\vert X\vert}\right)^t\leq \left(\frac{(1-q)k}{(1-\frac{q}{2})k}\right)^t=\left(1-\frac{q/2}{1-\frac{q}{2}}\right)^t\leq \left(1-\frac{q}{2}\right)^t\leq e^{-qt/2}\leq e^{-2\log k}=k^{-2}.\]
In other words, for any pair of distinct vertices $v,w\in V(H)$, the probability of having $N(v)\cap S= N(w)\cap S$ is at most $k^{-2}$. By a union bound, we see that with probability at least $\frac{1}{2}$ we have  $N(v)\cap S\neq  N(w)\cap S$ for all distinct $v,w\in V(H)\sm S$ (which means that $S$ is a signature).
\end{proof}

Note that in Definition \ref{defi-signature} we can equivalently write $(N(v)\Delta N(w))\cap S\neq \emptyset$ instead of $N(v)\cap S\neq N(w)\cap S$. This motivates the following definition, which is a strengthening of the concept of a signature.

\begin{definition}\label{defi-super-signature} For $0<r<1$, an $r$-super-signature of the graph $H$ is a non-empty subset $S\su V(H)$ such that $\vert (N(v)\Delta N(w))\cap S\vert\geq r\vert S\vert$, whenever $v,w\in V(H)\sm S$ are distinct vertices.
\end{definition}

\begin{lemma}\label{lemma-small-super-signature} For every subset $X\su V(H)$ of size $\vert X\vert \geq (1-\frac{q}{2})k$, there is a $\frac{q}{4}$-super-signature $S$ of $H$ of size $\vert S\vert \leq \frac{33}{q}\log k$ with $S\su X$.
\end{lemma}

\begin{proof}Let $t=\lfloor \frac{33}{q}\log k\rfloor$ and note that $t\geq  \frac{32}{q}\log k$ and $t\leq \frac{33}{q}\log k< \frac{1}{100}\sqrt{k}<\frac{k}{2}< \vert X\vert$ by  (\ref{eq-ineq-q-k1}). Again, let us choose vertices $s_1,\dots, s_t\in X$ independently and uniformly at random and set $S=\lbrace s_1,\dots, s_t\rbrace$. Then $S$ always satisfies $S\su X$ and $\vert S\vert \leq t\leq \frac{33}{q}\log k$.

The probability that $s_1,\dots, s_t$ are distinct is
\[\frac{\vert X\vert}{\vert X\vert}\cdot \frac{\vert X\vert-1}{\vert X\vert} \dotsm \frac{\vert X\vert-t+1}{\vert X\vert}\geq \left(1-\frac{t}{\vert X\vert}\right)^t\geq 1-\frac{t^2}{\vert X\vert}\geq 1-\frac{(\sqrt{k}/100)^2}{k/2}\geq \frac{3}{4},\]
where in the second step we used Bernoulli's inequality (recall that $t< \vert X\vert$).

For the moment, fix distinct vertices $v,w\in H$. We want to find an upper bound for the probability that the number of indices $i=1,\dots,t$ with $s_i\in N(v)\Delta N(w)$ is smaller than $\frac{q}{4}t$. For $i=1,\dots,t$ let $Z_i$ be the indicator random variable of the event $s_i\in N(v)\Delta N(w)$. Let $z=\vert (N(v)\Delta N(w))\cap X\vert/\vert X\vert$ be the probability for $Z_i=1$. Since $\vert N(v)\Delta N(w)\vert\geq qk$ by condition (b) in Definition \ref{defi-reasonable} and $\vert X\vert \geq (1-\frac{q}{2})k$, we have $\vert (N(v)\Delta N(w))\cap X\vert\geq \frac{q}{2}k$ and therefore
\[z=\frac{\vert (N(v)\Delta N(w))\cap X\vert}{\vert X\vert}\geq \frac{qk/2}{k}=\frac{q}{2}.\]
Note that $Z_1+\dots+Z_t\sim B(t,z)$ is a binomially distributed random variable. Hence the Chernoff bound for lower tails of binomial random variables (see for example \cite[Theorem A.1.13]{alon-spencer}) yields
\[\P\left[Z_1+\dots+Z_t<\frac{q}{4}t\right]\leq \P\left[Z_1+\dots+Z_t<zt-\frac{z}{2}t\right]< e^{-(zt/2)^2/(2zt)}=e^{-zt/8}\leq e^{-qt/16} \leq k^{-2}.\]
Thus, for each pair of distinct vertices $v,w\in H$ the event that there are fewer than $\frac{q}{4}t$ indices $i=1,\dots,t$ with $s_i\in N(v)\Delta N(w)$ has probability at most $k^{-2}$. By the union bound over $\binom{k}{2}$ pairs of distinct vertices, we obtain that with probability at least $\frac{1}{2}$, for any distinct $v,w\in V(H)$ there are at least $\frac{q}{4}t$ indices $i=1,\dots,t$ with $s_i\in N(v)\Delta N(w)$.

All in all, with probability at least $\frac{1}{4}$, the vertices $s_1,\dots, s_t$ are distinct and for any distinct $v,w\in V(H)$ there are at least $\frac{q}{4}t$ indices $i=1,\dots,t$ with $s_i\in N(v)\Delta N(w)$. But then $\vert S\vert =t$ and for any distinct $v,w\in V(H)$ we have $\vert (N(v)\Delta N(w))\cap S\vert\geq \frac{q}{4}t=\frac{q}{4}\vert S\vert$. Thus, the desired $\frac{q}{4}$-super-signature exists.
\end{proof}

\subsection{Counting embeddings}

In this subsection we prove several lemmas bounding the number of graph embeddings with certain restriction. We will use these lemmas repeatedly during the proof of Theorem \ref{thm-H-reasonable-blow-up}.

Our bounds will be stated in term of the function $\E_\l(m)$ introduced in the following definition.

Note that for integers $\l\geq 1$ and $m\geq 0$, there is a unique way to write $m=m_1+\dots+m_{\l}$ with non-negative integers $m_1,\dots,m_{\l}$ such that $\lceil \frac{m}{\l}\rceil\geq m_1\geq \dots\geq m_\l\geq \lfloor \frac{m}{\l}\rfloor$.

\begin{definition}\label{defi-functions-E} For integers $\l\geq 1$ and $m\geq 0$, define
\[\E_\l(m)=m_1\dotsm m_\l,\]
where $m_1,\dots,m_{\l}$ are such that $\lceil \frac{m}{\l}\rceil\geq m_1\geq \dots\geq m_\l\geq \lfloor \frac{m}{\l}\rfloor$ and $m_1+\dots+m_{\l}=m$.
\end{definition}

Note that we have $\left\lfloor\frac{m}{\l}\right\rfloor^{\l}\leq \E_\l(m)\leq \left(\frac{m}{\l}\right)^{\l}$ (with equalities when $m$ is divisible by $\l$). If $m$ is large compared to $\l$, then $\E_\l(m)\sim \left(\frac{m}{\l}\right)^{\l}$, but for $m<\l$ we have $\E_\l(m)=0$. Also note that $\E_\l(m)$ is monotone increasing as a function of $m$ (for $\l$ fixed).

The following lemma states some useful properties of the function $\E_{\l}(m)$. The third part is a slight strengthening of \cite[Proposition 4.3(ii)]{fox-huang-lee}.

\begin{lemma}\label{lemma-functions-E}Let $\l, \l'\geq 1$ and $m,m'\geq 0$ be integers. Then the following statements hold.
\begin{itemize}
\item[(i)] For any non-negative integers $\tilde{m}_1,\dots,\tilde{m}_\l$ with $\tilde{m}_1+\dots+\tilde{m}_\l\leq m$, we have $\tilde{m}_1\dotsm \tilde{m}_\l\leq \E_\l(m)$.
\item[(ii)] $\E_\l(m)\cdot\ \E_{\l'}(m')\leq \E_{\l+\l'}(m+m')$.
\item[(iii)] If $m\geq \l$ and $\l'\leq \l$ as well as $m'\leq (1-\mu)m$ for some $\mu\geq 0$, then
\[\E_{\l'}(m')\leq e^{3(\l-\l')-\mu\l/2}\left(\frac{\l}{m}\right)^{\l-\l'}\E_\l(m).\]
\end{itemize}
\end{lemma}

Parts (i) and (ii) of Lemma \ref{lemma-functions-E} are relatively easy to see. For part (iii), the same proof as in \cite[proof of Proposition 4.3(ii)]{fox-huang-lee} works. For the reader's convenience, we provide a proof of all three parts of Lemma \ref{lemma-functions-E} in the appendix.

The following lemmas will be useful tools in the proof of Theorem \ref{thm-H-reasonable-blow-up}.

\begin{lemma}\label{lemma-signature-extensions} Let $\Gamma$ be a graph and let $U\su V(\Gamma)$ be a subset of its vertices. Suppose that $f:X\to V(\Gamma)$ is a function defined on a subset $X\su V(H)$ such that $X$ is a signature of $H$. Then the number of embeddings $\theta:H\into \Gamma$ with $\theta(v)=f(v)$ for every $v\in X$ and $\theta(v)\in U$ for all $v\in V(H)\sm X$ is at most $\E_{k-\vert X\vert}(\vert U\vert)$.
\end{lemma}
\begin{proof} Let $\l=k-\vert X\vert$ and enumerate the vertices in $V(H)\sm X$ from 1 to $\l$. For $i=1,\dots,\l$, let $U_i\su U$ be the set of vertices $u\in U$ for which there exists at least one embedding $\theta:H\into \Gamma$ satisfying the properties in the lemma and mapping vertex $i$ to $u$. Clearly, the number of embeddings $\theta:H\into \Gamma$ with the desired properties is at most $\vert U_1\vert \dotsm \vert U_{\l}\vert$, because each vertex $i$ needs to be mapped into $U_i$ and the images of the vertices in $X$ are already determined by $f$.

Lemma \ref{lemma-signature-disjoint-sets} (applied with $H'=H$) implies that the sets  $U_1, \dots, U_{\l}$ are disjoint. Hence $\vert U_1\vert + \dots+ \vert U_{\l}\vert\leq \vert U\vert$ and the number of embeddings $\theta:H\into \Gamma$ with the properties in the lemma is at most
\[\vert U_1\vert \dotsm \vert U_{\l}\vert\leq \E_{\l}(\vert U\vert)=\E_{k-\vert X\vert}(\vert U\vert),\]
where in the first inequality we used Lemma \ref{lemma-functions-E}(i).
\end{proof}

\begin{lemma}\label{lemma-embedding-small-set}Let $\Gamma$ be a graph on $n\geq k$ vertices and $U\su V(\Gamma)$ a subset of size $\vert U\vert\leq \gamma n$ for some $0<\gamma<1$. Furthermore, let $H'$ be an induced subgraph of $H$ on $k'\geq k-4$ vertices and let $X\su V(H')$ be a subset of size $\vert X\vert \geq \beta k$ for some $\gamma<\beta\leq \frac{q}{3}$. Assume that $\beta k\geq \frac{20}{q}(\log k)^2$. Then the number of embeddings $\theta:H'\into \Gamma$ with $\theta(x)\in U$ for every $x\in X$ is at most
\[ \left(\frac{e^4\gamma}{\beta}\right)^{\beta k}\cdot \frac{\E_k(n)}{n^{k-k'}}.\]
\end{lemma}
\begin{proof}By removing elements from the set $X$, we may assume without loss of generality that $X$ has size $\vert X\vert=\lceil \beta k\rceil\leq \beta k+1\leq \frac{q}{3}k+1\leq \frac{q}{2}k-4$ (for the last inequality see (\ref{eq-ineq-q-k1})). Then $\vert V(H')\sm X\vert \geq (1-\frac{q}{2})k$ and by Lemma \ref{lemma-small-signature} there exists a signature $S$ of $H$ of size $\vert S\vert\leq \frac{5}{q}\log k$ with $S\su V(H')\sm X$.

There are at most $n^{\vert S\vert}$ possibilities for $\theta\vert_S$. Let us fix one particular choice  for $\theta\vert_S$, then we have determined $\theta(s)$ for all $s\in S$. For each remaining vertex $v\in V(H')\sm S$, let $U_v\su V(\Gamma)$ be the set of possible images of $v$ when extending our chosen $\theta\vert_S$ to an embedding $\theta:H'\into \Gamma$ with the properties in the lemma. By the condition on $\theta$, we have $U_v\su U$ for $v\in X$. The number of extensions of $\theta\vert_S$ is at most $\prod_{v\in V(H')\sm S}\vert U_v\vert$. Lemma \ref{lemma-signature-disjoint-sets} implies that the sets $U_v$ for $v\in V(H')\sm S$ are disjoint. Thus, $\sum_{v\in X}\vert U_v\vert\leq \vert U\vert\leq \gamma n$ and $\sum_{V(H')\sm S}\vert U_v\vert\leq n$. Hence we obtain (using Lemma \ref{lemma-functions-E}(i) and Lemma \ref{lemma-functions-E}(iii))
\[\prod_{v\in V(H')\sm (S\cup X)}\vert U_v\vert\leq E_{k'-\vert S\vert-\vert X\vert}(n)\leq e^{3(k-k'+\vert S\vert+\vert X\vert)}\left(\frac{k}{n}\right)^{k-k'+\vert S\vert+\vert X\vert}\E_k(n).\]
and (using the inequality of arithmetic and geometric mean)
\[\prod_{v\in X}\vert U_v\vert\leq \left(\frac{\sum_{v\in X}\vert U_v\vert}{\vert X\vert}\right)^{\vert X\vert}\leq \left(\frac{\gamma n}{\beta k}\right)^{\vert X\vert}= \left(\frac{\gamma}{\beta}\right)^{\vert X\vert}\cdot \left(\frac{n}{k}\right)^{\vert X\vert}\leq \left(\frac{\gamma}{\beta}\right)^{\beta k}\cdot \left(\frac{n}{k}\right)^{\vert X\vert}\]
Thus, the number of extensions of this particular choice for $\theta\vert_S$ is at most
\[\prod_{v\in V(H')\sm S}\vert U_v\vert=\prod_{v\in V(H')\sm (S\cup X)}\vert U_v\vert \cdot \prod_{v\in X}\vert U_v\vert\leq e^{3(k-k'+\vert S\vert+\vert X\vert)}\left(\frac{k}{n}\right)^{k-k'+\vert S\vert}\left(\frac{\gamma}{\beta}\right)^{\beta k}\E_k(n).\]
Adding this up over all the at most $n^{\vert S\vert}$ choices for $\theta\vert_S$ gives that the total number of embeddings $\theta:H'\into \Gamma$ with $\theta(x)\in U$ for every $x\in X$ is at most
\[n^{\vert S\vert}\cdot e^{3(k-k'+\vert S\vert+\vert X\vert)}\left(\frac{k}{n}\right)^{k-k'+\vert S\vert}\left(\frac{\gamma}{\beta}\right)^{\beta k}\E_k(n)=
e^{3(k-k'+\vert S\vert+\vert X\vert)}k^{k-k'+\vert S\vert}\left(\frac{\gamma}{\beta}\right)^{\beta k}\frac{\E_k(n)}{n^{k-k'}}.\]
Now, noting that (recall that $q<1$ and $k\geq \tk/2\geq 10^{199}$)
\[e^{3(k-k'+\vert S\vert+\vert X\vert)}k^{k-k'+\vert S\vert}\leq e^{3\cdot (4+(5/q)\log k+\beta k+1)}e^{(\log k)(4+(5/q)\log k)}\leq e^{(20/q)(\log k)^2+3\beta k}\leq e^{4\beta k}\]
gives the desired bound.
\end{proof}

\begin{corollary}\label{coro-embedding-small-set}Let $\Gamma$ be a graph on $n\geq k$ vertices and $U\su V(\Gamma)$ a subset of size $\vert U\vert\leq \gamma n$ for some $0<\gamma<1$. Furthermore, let $H'$ be an induced subgraph of $H$ on $k'\geq k-4$ vertices and assume that $0<\beta<1$ satisfies $\gamma<\beta\leq \frac{q}{3}$ and $\beta k>\frac{20}{q}(\log k)^2$. Then the number of embeddings $\theta:H'\into \Gamma$ with $\theta(x)\in U$ for at least $\beta k$ vertices $x\in V(H')$ is at most
\[ \left(\frac{e^6\gamma}{\beta^2}\right)^{\beta k}\cdot \frac{\E_k(n)}{n^{k-k'}}.\]
\end{corollary}
\begin{proof} First, note that $\beta>\frac{20}{q}(\log k)^2/k>2/k$ and therefore $\log(1/\beta)<\log k<\beta k/20$. Hence $\log(1/\beta)+1<\beta k$ and therefore
\[(\log(1/\beta)+1)(\beta k+1)=(\log(1/\beta)+1)\beta k+(\log(1/\beta)+1)<(\log(1/\beta)+2)\beta k.\]
Let us apply Lemma \ref{lemma-embedding-small-set} for all choices of $X\su V(H')$ of size $X=\lceil \beta k\rceil$. The number of choices for $X$ is
\[\binom{k'}{\lceil \beta k\rceil}\leq \left(\frac{ek'}{\lceil \beta k\rceil}\right)^{\lceil \beta k\rceil}\leq \left(\frac{ek}{\beta k}\right)^{\beta k+1}=e^{(\log(1/\beta)+1)(\beta k+1)}< e^{(\log(1/\beta)+2)\beta k}=\left(\frac{e^2}{\beta}\right)^{\beta k}.\]
Multiplying this with the bound in Lemma \ref{lemma-embedding-small-set} gives the desired result.\end{proof}

\begin{lemma}\label{lemma-number-rot-refl}Fix any vertex $x\in V(\tH)$. Then, among the $\tk$ rotations $\phi$ of $\tH$, there are exactly $k$ rotations $\phi$ with $x\in \phi(V(H))$ and for the remaining $\tk-k$ rotations we have $x\not\in \phi(V(H))$. Similarly, among the $\tk$ reflections $\phi$ of $\tH$, there are exactly $k$ reflections $\phi$ with $x\in \phi(V(H))$ and for the remaining $\tk-k$ reflections we have $x\not\in \phi(V(H))$.
\end{lemma}
\begin{proof}First, let us count the number of rotations $\phi$ with $x\in \phi(V(H))$. Each rotation is given as $y\mapsto y+g$ for some $g\in G$ and in order to have $x\in \phi(V(H))$ there must be a vertex $y\in V(H)\su G$ with $x=y+g$. There are $k$ choices for $y$ and hence $k$ choices for $g$. This shows that there are exactly $k$ rotations $\phi$ with $x\in \phi(V(H))$. We can analogously prove that there are  exactly $k$ reflections $\phi$ with $x\in \phi(V(H))$ by considering the equation $x=-y+g$.
\end{proof}

\begin{lemma}\label{lemma-emb-H-prime}For every integer $n\geq \tk$, 
we have $\emb(H,n)\geq \tk^{1/4}\cdot \E_k(n)$.
\end{lemma}
\begin{proof} Let $\Gamma'$ be a balanced blow-up of $\tH$ with $n$ vertices. This means $\Gamma'$ is a blow-up of $\tH$ where the parts of the blow-up have sizes $n_{v}$ with $\lceil \frac{n}{\tk}\rceil\geq n_v\geq \lfloor \frac{n}{\tk}\rfloor$ for all $v\in V(\tH)$ and $n=\sum_{v\in V(\tH)}n_v$. Note that then $\prod_{v\in V(\tH)}n_v=\E_{\tk}(n)$.

If $\phi$ is a rotation of $\tH$, then we can form $\prod_{v\in V(H)}n_{\phi(v)}$ embeddings $H\into \Gamma'$ by mapping each vertex $v$ into the part of the blow-up belonging to $\phi(v)$. Thus, using the inequality between arithmetic and geometric mean, we obtain
\[\emb(H,\Gamma')\geq \sum_{\phi}\prod_{v\in V(H)}n_{\phi(v)}\geq \tk\left(\prod_{\phi}\prod_{v\in V(H)}n_{\phi(v)}\right)^{1/\tk}=\tk\left(\prod_{\phi}\prod_{v\in \phi(V(H))}n_{v}\right)^{1/\tk}\]
(here, the sum and the product are over the $\tk$ rotations $\phi$). By Lemma \ref{lemma-number-rot-refl}, every vertex $v\in V(\tH)$ is in the image $\phi(V(H))$ for exactly $k$ rotations $\phi$. Thus,
\[\emb(H,\Gamma')\geq \tk\left(\prod_{v\in V(\tH)}n_{v}^{k}\right)^{1/\tk}=\tk\left(\prod_{v\in V(\tH)}n_{v}\right)^{k/\tk}=\tk\cdot \E_{\tk}(n)^{k/\tk}.\]
Recall that $\E_{\tk}(n)\leq (n/\tk)^{\tk}$ and therefore, using Lemma \ref{lemma-functions-E}(iii), we have
\[\E_k(n)\leq e^{3(\tk-k)}\left(\frac{\tk}{n}\right)^{\tk-k}\E_{\tk}(n)\leq e^{3(\tk-k)}\left(\E_{\tk}(n)^{-1/\tk}\right)^{\tk-k}\E_{\tk}(n)=e^{3(\tk-k)}\E_{\tk}(n)^{k/\tk}.\]
Hence,
\[\emb(H,n)\geq \emb(H,\Gamma')\geq \tk\cdot \E_{\tk}(n)^{k/\tk}\geq \tk\cdot e^{-3(\tk-k)}\E_k(n)\geq \tk^{1/4}\cdot \E_k(n),\]
as $\tk -k\leq \frac{1}{4}\log \tk$. This finishes the proof of the lemma.
\end{proof}

\subsection{Final preparations}

The following lemma is Lemma 4.4 in \cite{fox-huang-lee}, we repeat the proof here for the reader's convenience. A similar lemma can also be found in \cite[Lemma 2.3]{hefetz-tyomkyn}.

\begin{lemma}[Lemma 4.4 in \cite{fox-huang-lee}]\label{lemma-each-vertex-many-embeddings} Let $n\geq 2$ and let $\Gamma$ be a graph on $n$ vertices with $\emb(H,\Gamma)=\emb(H,n)$. Then every vertex of $\,\Gamma$ is in the image of at least $\frac{k}{n+k}\emb(H,n)$ embeddings $H\into \Gamma$.
\end{lemma}
\begin{proof}Suppose some vertex $v\in V(\Gamma)$ is contained in the image of fewer than $\frac{k}{n+k}\emb(H,n)$ embeddings $H\into \Gamma$. When deleting this vertex $v$ from the graph $\Gamma$, we obtain
\[\emb(H,\Gamma\sm\lbrace v\rbrace)>\emb(H,\Gamma)-\frac{k}{n+k}\emb(H,n)=\frac{n}{n+k}\emb(H,n).\]
On average, each vertex $w\in V(\Gamma\sm\lbrace v\rbrace)$ is in the image of $\frac{k}{n-1}\emb(H, \Gamma\sm\lbrace v\rbrace)$ embeddings $H\into \Gamma\sm\lbrace v\rbrace$. Hence there exists a vertex $w\in V(\Gamma\sm\lbrace v\rbrace)$ that appears in at least $\frac{k}{n-1}\emb(H, \Gamma\sm\lbrace v\rbrace)\geq \frac{k}{n}\emb(H, \Gamma\sm\lbrace v\rbrace)$ embeddings $H\into \Gamma\sm\lbrace v\rbrace$. Now let the graph $\Gamma'$ be obtained from $\Gamma\sm\lbrace v\rbrace$ by making an additional copy of the vertex $w$ (unconnected to the original vertex $w$). Then $\Gamma'$ has $n$ vertices and
\begin{multline*}
\emb(H,\Gamma')\geq \emb(H, \Gamma\sm\lbrace v\rbrace)+\frac{k}{n}\emb(H, \Gamma\sm\lbrace v\rbrace)=\frac{n+k}{n}\emb(H, \Gamma\sm\lbrace v\rbrace)\\
>\frac{n+k}{n}\cdot \frac{n}{n+k}\emb(H,n)=\emb(H,n)
\end{multline*}
This is a contradiction, as $\emb(H,n)$ is the maximum number of embeddings of $H$ into a graph on $n$ vertices.
\end{proof}

\begin{corollary}\label{coro-each-vertex-many-embeddings} Let $n\geq \tk$ and let $\Gamma$ be a graph on $n$ vertices with $\emb(H,\Gamma)=\emb(H,n)$. Then every vertex of $\Gamma$ is in the image of at least $\E_k(n)/n$ embeddings $H\into \Gamma$.
\end{corollary}
\begin{proof}By Lemma \ref{lemma-emb-H-prime} and $n\geq \tk\geq k$, we have
\[\frac{k}{n+k}\emb(H,n)\geq \frac{k}{2n}\cdot \tk^{1/4}\cdot \E_k(n)\geq \frac{\E_k(n)}{n}.\]
So the statement follows immediately from Lemma \ref{lemma-each-vertex-many-embeddings}.
\end{proof}

The following lemma provides values for certain parameters that will play a crucial role in the proof of Theorem \ref{thm-H-reasonable-blow-up}.

\begin{lemma}\label{lemma-epsilon-parameters} Let us define the following parameters:
\begin{align*}
\eps_1&=\frac{q}{3}\\
\eps_2&=10^{-2}\frac{q}{\log(1/q)}\\
\eps_3&=10^{-5}\frac{q}{(\log(1/q))^2}\\
\eps_4&=10^{-7}\frac{q^2}{(\log(1/q))^2}\\
\eps_5&=10^{-19}\frac{q^4}{(\log(1/q))^4}.
\end{align*}
Then the following inequalities hold:
\begin{align}
\eps_5<\eps_4<\eps_3<\eps_2<\eps_1&< \frac{q}{2}<10^{-20}\label{ineq-eps1-q}\\
\eps_5<\eps_4<\eps_3<\eps_2&< \frac{1}{100}\label{ineq-eps2-100}\\
\log(1/\eps_2)\cdot \eps_2&< \frac{\log 2}{8}\eps_1\label{ineq-eps2-eps1}\\
\eps_2&< \delta\label{ineq-eps2-delta}\\
\log(1/\eps_3)\cdot \eps_3&< \frac{1}{100}\eps_2\label{ineq-eps3-eps2}\\
\eps_2>\eps_3>\eps_3^2&>\frac{10^6}{q}\cdot \frac{(\log k)^2}{k}\label{ineq-eps3-k}\\
\eps_4&< \frac{q}{20}\eps_3\label{ineq-eps4-eps3}\\
\eps_5&\leq \frac{1}{10^5}\eps_4^2\label{ineq-eps5-eps4}\\
\eps_5&<\frac{q}{10^4}\label{ineq-eps5-q}\\
\eps_1>\eps_2>\eps_3>\eps_4>\eps_5&> \frac{40}{q}\cdot \frac{(\log k)^2}{k}> \frac{10^3}{q}\cdot \frac{\log k}{k}>\frac{10^3}{k}.\label{ineq-eps5-logk-k}
\end{align}
\end{lemma}
The proof of Lemma \ref{lemma-epsilon-parameters} is a straightforward calculation. For the reader's convenience, we provide the details in the appendix. We remark that the actual values of the $\eps_i$ are not important, we just need the inequalities (\ref{ineq-eps1-q}) to (\ref{ineq-eps5-logk-k}) to hold.

\section{Proof of Theorem \ref{thm-H-reasonable-blow-up}}\label{sect-proof-blow-up}

Recall that we fixed $\tH$ and $H$ at the beginning of Section \ref{sect-prep-proof-blow-up}. Let us also fix $\eps_1$ to $\eps_5$ as in Lemma \ref{lemma-epsilon-parameters}.

The goal of this section is to prove Theorem \ref{thm-H-reasonable-blow-up}, apart from the proofs of several propositions and lemmas which we will postpone to the following sections. Note that every graph $\Gamma$ on $n<\tk$ is by definition a balanced iterated blow-up of $\tH$. Hence we may assume that $n\geq \tk$. Let $\Gamma$ be a graph on $n$ vertices maximizing the number of embeddings $H\into \Gamma$, which means $\emb(H,\Gamma)=\emb(H,n)$. We need to prove that $\Gamma$ is a balanced iterated blow-up of $\tH$.

Note that by Lemma \ref{lemma-emb-H-prime} we have
\begin{equation}\label{eq-emb-H-Gamma-big}
\emb(H,\Gamma)=\emb(H,n)\geq \tk^{1/4}\cdot \E_k(n).
\end{equation}
and by Corollary \ref{coro-each-vertex-many-embeddings} each vertex of $\Gamma$ is contained in the image of at least $\E_k(n)/n$ embeddings $H\into \Gamma$.

From now on, let us fix a signature $S\su V(H)$ of the graph $H$ of size $\vert S\vert \leq \frac{5}{q}\log k$. Such a signature exists by Lemma \ref{lemma-small-signature} applied to $X=V(H)$.

There are $n^{\vert S\vert}$ maps $\psi:S\to V(\Gamma)$. Hence, by the pigeonhole principle, there must be a map $\psi:S\to V(\Gamma)$ such that there exist at least $\emb(H,\Gamma)/n^{\vert S\vert}$ embeddings $H\into \Gamma$ extending $\psi$.

For every $i\in V(H)\sm S$, let $V_i'\su V(\Gamma)$ be the set of vertices $v\in V(\Gamma)$ for which there exists at least one embedding $H\into \Gamma$ extending $\psi$ and mapping vertex $i$ to $v$. Then each embedding $H\into \Gamma$ extending $\psi$ needs to map every vertex $i\in V(H)\sm S$ into the set $V_i'\su V(\Gamma)$. Note that by Lemma \ref{lemma-signature-disjoint-sets} (applied with $H'=H$) the sets $V_i'$ for $i\in V(H)\sm S$ are disjoint.

\begin{definition}\label{defi-bad}For distinct $i,j\in V(H)\sm S$, let us call a pair of vertices $(v_i,v_j)\in V_i'\times V_j'$ \emph{bad}, if $a_\Gamma(v_i,v_j)\neq a_H(i,j)$. For $i\in V(H)\sm S$, let us call a vertex $v\in V_i'$ \emph{bad}, if it is part of at least $\eps_5 n$ bad pairs $(v,w)$. In other words, $v\in V_i'$ is bad if there are at least $\eps_5 n$ different vertices $w\in \bigcup_{j\in V(H)\sm (S\cup \lbrace i\rbrace)} V_j'$ such that $(v,w)$ is a bad pair.
\end{definition}

For every $i\in V(H)\sm S$, let us define $V_i$ to be the set obtained from $V_i'$ by deleting all bad vertices in $V_i'$. Then the sets $V_i$ for $i\in V(H)\sm S$ are disjoint (since the sets $V_i'$ for $i\in V(H)\sm S$ are disjoint). We will see in Lemma \ref{lemma-bad-vertex} that every bad vertex is in the image of only few embeddings $H\into \Gamma$ extending $\psi$, and hence there are many embeddings $H\into \Gamma$ extending $\psi$ and mapping $i$ into $V_i$ for each $i\in V(H)\sm S$ (see Lemma \ref{lemma-embeddings-Vi}). These and other useful properties of the sets $V_i'$ and $V_i$ will be established in Section \ref{sect-properties-Vi}.

\begin{definition}\label{defi-loyal}For a rotation or reflection $\phi$ of $\tH$, an embedding $\theta:H\into\Gamma$ is called \emph{$\phi$-loyal} if the image of $\theta$ contains at most one vertex from $V_i$ for each $i\in V(H)\sm S$ and if for at least $(1-\eps_2)k$ vertices $x\in V(H)$ we have $\theta(x)\in V_{\phi(x)}$ (in particular $\phi(x)\in V(H)\sm S$ in order for $V_{\phi(x)}$ to be defined). An embedding $\theta:H\into\Gamma$ is called \emph{loyal}, if it is $\phi$-loyal for some rotation or reflection $\phi$ of $\tH$, and \emph{disloyal} otherwise.
\end{definition}

A crucial step in the proof of Theorem \ref{thm-H-reasonable-blow-up} is to show the following proposition.

\begin{proposition}\label{propo-few-unloyal}Let $x, x'\in V(H)$ be distinct vertices and let $z, z'\in V(\Gamma)$. Then there are at most $k^{-6}\cdot \E_k(n)/n^{2}$ disloyal embeddings $\theta:H\into\Gamma$ with $\theta(x)=z$ and $\theta(x')=z'$.
\end{proposition}

We will prove Proposition \ref{propo-few-unloyal} in Section \ref{sect-few-unloyal}, using the properties of the sets $V_i$ established in Section \ref{sect-properties-Vi}. Proposition \ref{propo-few-unloyal} gives the following corollary, which we will prove at the end of Section \ref{sect-few-unloyal}.

\begin{corollary}\label{coro-few-unloyal}Let $x\in V(H)$ and $z\in V(\Gamma)$. Then there are at most $k^{-6}\cdot \E_k(n)/n$ disloyal embeddings $\theta:H\into\Gamma$ with $\theta(x)=z$. Furthermore, the total number of disloyal embeddings $H\into\Gamma$ is at most $k^{-6}\cdot \E_k(n)$.
\end{corollary}

In order to establish the blow-up structure of $\Gamma$, let us define a subset $W_j\su V(\Gamma)$ for each $j\in V(\tH)$ as follows:

\begin{definition}\label{defi-Wj} For $j\in V(\tH)$, let $W_j\su V(\Gamma)$ consist of those vertices $w\in V(\Gamma)$ for which there are at least $(1-\eps_1)k$ vertices $i\in V(H)\sm(S\cup \lbrace j\rbrace)$ satisfying $w\not\in V_i$ and $a_\Gamma(w,V_i)=a_{\tH}(j,i)$.
\end{definition}

So, roughly speaking, $W_j$ consist of those vertices $w$ such that for most $i\in V(H)\sm S$, the adjacency between $w$ and the set $V_i$ is the same as the adjacency between $j$ and $i$. We wish to show that $\Gamma$ is a blow-up of $\tH$ with parts $W_j$ for $j\in V(\tH)$. The following three lemmas state important properties of the sets $W_j$ and will be proved in Section \ref{sect-properties-Wj}.

\begin{lemma}\label{lemma-loyal-unregulated}Let $\phi$ be a rotation or reflection of $\tH$. Then there are at most $k^{-7}\cdot \E_k(n)$ different $\phi$-loyal embeddings $\theta:H\into\Gamma$ that satisfy $\theta(x)\not\in W_{\phi(x)}$ for some $x\in V(H)$.
\end{lemma}

\begin{lemma}\label{lemma-phi-loyal-between-W}Let $j,j'\in V(\tH)$ be distinct and let $w\in W_j$ and $w'\in W_{j'}$. Assume that $a_{\Gamma}(w,w')\neq a_{\tH}(j,j')$. Then there are at most $k^{-4}\E_k(n)/n^2$ loyal embeddings $H\into \Gamma$ whose image contains both $w$ and $w'$.
\end{lemma}

\begin{lemma}\label{lemma-W-partition}The sets $W_j$ for $j\in V(\tH)$ are disjoint and their union is the whole vertex set $V(\Gamma)$.\end{lemma}

Thus, the sets $W_j$ for $j\in V(\tH)$ form a partition of $V(\Gamma)$. Let us now make the following definition:

\begin{definition}\label{defi-regulated}For a rotation or reflection $\phi$ of $\tH$, an embedding $\theta:H\into\Gamma$ is called \emph{$\phi$-regulated} if $\theta(x)\in W_{\phi(x)}$ for all $x\in V(H)$. An embedding $\theta:H\into\Gamma$ is called \emph{regulated}, if it is $\phi$-regulated for some rotation or reflection $\phi$ of $\tH$, and \emph{unregulated} otherwise.
\end{definition}

Our strategy for establishing the desired blow-up structure of $\Gamma$ with parts $W_j$ for $j\in V(\tH)$ is to analyze the regulated embeddings. The next lemma states that there are only few unregulated embeddings, and in the following corollary we will deduce that there must be many regulated embeddings.

\begin{lemma}\label{lemma-few-unregulated}There are at most $\E_k(n)$ unregulated embeddings $H\into\Gamma$.
\end{lemma}
\begin{proof}Recall that by Corollary \ref{coro-few-unloyal} there are at most $k^{-6}\E_k(n)$ disloyal embeddings $H\into\Gamma$. So it remains to bound the number of unregulated embeddings $H\into\Gamma$ that are $\phi$-loyal for some rotation or reflection $\phi$ of $\tH$.

So let us fix a rotation or reflection $\phi$ of $\tH$. Each unregulated $\phi$-loyal embedding $\theta:H\into\Gamma$ must satisfy $\theta(x)\not\in W_{\phi(x)}$ for some $x\in V(H)$ (since otherwise $\theta$ would be $\phi$-regulated). Thus, by Lemma \ref{lemma-loyal-unregulated} there are at most $k^{-7}\E_k(n)$ unregulated $\phi$-loyal embeddings $\theta:H\into\Gamma$.

Since there are at most $2\tk\leq 4k$ rotations and reflections $\phi$ of $\tH$, there are at most $4k^{-6}\E_k(n)$ unregulated loyal embeddings. So all in all, there can be at most $k^{-6}\E_k(n)+4k^{-6}\E_k(n)=5k^{-6}\E_k(n)\leq \E_k(n)$ unregulated embeddings $H\into\Gamma$.
\end{proof}

\begin{corollary}\label{coro-number-regulated-embeddings}There are at least $(\tk^{1/4}-1)\cdot \E_k(n)$ regulated embeddings $H\into \Gamma$.
\end{corollary}
\begin{proof}By (\ref{eq-emb-H-Gamma-big}) there are at least $\tk^{1/4}\cdot \E_k(n)$ embeddings $H\into \Gamma$ and by Lemma \ref{lemma-few-unregulated} at most $\E_k(n)$ of these embeddings are unregulated. Hence there are at least $(\tk^{1/4}-1)\cdot \E_k(n)$ regulated embeddings $H\into \Gamma$.
\end{proof}

\begin{claim}\label{claim-number-regulated}For every rotation or reflection $\phi$ of $\tH$, the number of $\phi$-regulated embeddings $H\into \Gamma$ is at most
\[\prod_{x\in V(H)}\vert W_{\phi(x)}\vert\leq \E_k(n).\]
\end{claim}
\begin{proof}
It follows from Definition \ref{defi-regulated} that the number of $\phi$-regulated embeddings is at most $\prod_{x\in V(H)}\vert W_{\phi(x)}\vert$. The inequality  follows from Lemma \ref{lemma-functions-E}(i) and the fact that the sets $W_{\phi(x)}$ are all disjoint by Lemma \ref{lemma-W-partition}.
\end{proof}

The following two statements will be needed later.

\begin{claim}\label{claim-embHn-small} $\emb(H,n)\leq 3k^{-k+1}n^k$.
\end{claim}
\begin{proof}
As there are at most $\E_k(n)$ different $\phi$-regulated embeddings $H\into \Gamma$ for each rotation or reflection $\phi$ of $\tH$, the total number of regulated embeddings $H\into \Gamma$ is at most $2\tk\cdot \E_k(n)$. Furthermore, by Lemma \ref{lemma-few-unregulated} there are at most $\E_k(n)$ unregulated embeddings $H\into \Gamma$. Thus,
\[\emb(H,n)\leq (2\tk+1)\cdot \E_k(n)\leq 3k\cdot \E_k(n)\leq 3k\cdot \left(\frac{n}{k}\right)^k=3k^{-k+1}n^k,\]
where we used that $3k\geq 3\tk-\log \tk\geq 2\tk+1$.
\end{proof}

\begin{lemma}\label{lemma-size-Wj}We have $1\leq \vert W_j\vert\leq k^{-4/5}\cdot n$ for every $j\in V(\tH)$.
\end{lemma}

We will prove Lemma \ref{lemma-size-Wj} in Section \ref{sect-properties-Wj}. In order to continue with our strategy to analyze the regulated embeddings, we make the following definition.

\begin{definition}\label{defi-significant}Let us call a rotation or reflection $\phi$ of $\tH$ \emph{significant} if the number of $\phi$-regulated embeddings $H\into \Gamma$ is at least $\E_k(n)/k$, and let us call $\phi$ \emph{insignificant} otherwise.
\end{definition}

Note that by Claim \ref{claim-number-regulated} for every significant rotation or reflection $\phi$ we must have
\[\prod_{x\in V(H)}\vert W_{\phi(x)}\vert\geq \frac{\E_k(n)}{k}.\]

\begin{claim}\label{claim-many-significant}There are at least $2\log \tk$ significant rotations and reflections $\phi$ of $\tH$.
\end{claim}
\begin{proof}
By definition, the number of embeddings $H\into \Gamma$ that are $\phi$-regulated for some insignificant $\phi$ is at most
\[2\tk\cdot \frac{\E_k(n)}{k}\leq 4k\cdot \frac{\E_k(n)}{k}=4\E_k(n).\]
So by Corollary \ref{coro-number-regulated-embeddings} and (\ref{ineq-k-logtk}), there must be at least
\[(\tk^{1/4}-5)\cdot \E_k(n)\geq (100(\log \tk)-5)\cdot \E_k(n)\geq 2\log \tk\cdot \E_k(n)\]
embeddings $H\into \Gamma$ that are $\phi$-regulated for some significant $\phi$. But for each $\phi$, by Claim \ref{claim-number-regulated}, there can be at most $\E_k(n)$ different $\phi$-regulated embeddings. Thus, the number of significant rotations and reflections $\phi$ of $\tH$ is at least $2\log \tk$.
\end{proof}

Recall that by Lemma \ref{lemma-W-partition} the sets $W_j$ for $j\in V(\tH)$ form a partition of $V(\Gamma)$. Let us now define a graph $\Gamma^*$ on the vertex set $V(\Gamma)$ by leaving $\Gamma$ unchanged inside each of the sets $W_j$ for $j\in V(\tH)$ and setting $a_{\Gamma^*}(w,w')=a_{\tH}(j,j')$ for all distinct $j,j'\in V(\tH)$ and $(w,w')\in W_j\times W_{j'}$. In other words, $\Gamma^*$ is a blow-up of $\tH$ with parts $W_j$ for $j\in V(\tH)$, and inside the parts $W_j$ the graph $\Gamma^*$ agrees with $\Gamma$. We wish to prove that $\Gamma=\Gamma^*$.

Since $\Gamma$ was chosen among all $n$-vertex graphs to maximize $\emb(H,\Gamma)$, we must have
\[\emb(H,\Gamma^*)\leq \emb(H,\Gamma).\]
Let $M$ be the number of pairs $\lbrace w,w'\rbrace$ of two distinct vertices in $V(\Gamma)$ such that $a_{\Gamma^*}(w,w')\neq a_{\Gamma}(w,w')$. In other words, $M$ is the number of pairs of vertices with different adjacencies in $\Gamma$ and $\Gamma^*$. For each such pair $\lbrace w,w'\rbrace$, the indices $j,j'\in V(\tH)$ with $w\in W_{j}$ and $w'\in W_{j'}$ must satisfy $j\neq j'$ and $a_{\Gamma}(w,w')\neq a_{\tH}(j,j')$. 

Our next goal is to prove that $M=0$. This will imply that $\Gamma=\Gamma^*$ and so $\Gamma$ will be a blow-up of $\tH$ with parts $W_j$ for $j\in V(\tH)$. Our strategy for proving $M=0$ will be to analyze the embeddings $H\into \Gamma$ and the embeddings $H\into \Gamma^*$. Recall that $\Gamma$ and $\Gamma^*$ have the same vertex set. So each embedding $H\into \Gamma$ gives an injective map $V(H)\into V(\Gamma^*)$ which may or may not be an embedding $H\into \Gamma^*$ (and vice versa).

\begin{lemma}\label{lemma-embeddings-Gamma-Gamma-star} There are at most $M\cdot k^{-3}\E_k(n)/n^2$ embeddings $H\into \Gamma$ that do not correspond to embeddings $H\into \Gamma^*$.
\end{lemma}
\begin{proof}For each embedding $\theta: H\into \Gamma$ that does not correspond to an embedding $H\into \Gamma^*$, the image of $\Gamma$ must contain two distinct vertices $w,w'\in V(\Gamma)$ with $a_{\Gamma^*}(w,w')\neq a_{\Gamma}(w,w')$. There are precisely $M$ pairs $\lbrace w,w'\rbrace$ of distinct vertices in $V(\Gamma)$ such that $a_{\Gamma^*}(w,w')\neq a_{\Gamma}(w,w')$. So it suffices to prove that for any such pair $\lbrace w,w'\rbrace$, there are at most $k^{-3}\E_k(n)/n^2$ embeddings $H\into \Gamma$ whose image contains both $w$ and $w'$.

So let us fix distinct $w,w'\in V(\Gamma)$ with $a_{\Gamma^*}(w,w')\neq a_{\Gamma}(w,w')$. Let $j,j'\in V(\tH)$ be the indices such that $w\in W_j$ and $w'\in W_{j'}$. Note that by the definition of $\Gamma^*$, the condition $a_{\Gamma^*}(w,w')\neq a_{\Gamma}(w,w')$ implies $j\neq j'$. Hence $a_{\Gamma^*}(w,w')= a_{\tH}(j,j')$ and therefore $a_{\Gamma}(w,w')\neq a_{\tH}(j,j')$. So by Lemma \ref{lemma-phi-loyal-between-W} there are at most $k^{-4}\E_k(n)/n^2$ loyal embeddings $H\into \Gamma$ whose image contains both $w$ and $w'$.

Let us now bound the number of disloyal embeddings $H\into \Gamma$ whose image contains both $w$ and $w'$. For any such disloyal embedding $\theta: H\into \Gamma$, there are $x,x'\in V(H)$ with $\theta(x)=w$ and $\theta(x')=w'$ (and note that $x\neq x'$, since $w\neq w'$). There are at most $k^2$ possibilities for $x,x'\in V(H)$, and for each of them there are at most $k^{-6}\E_k(n)/n^2$ disloyal embedding $\theta: H\into \Gamma$ with $\theta(x)=w$ and $\theta(x')=w'$ (by Proposition \ref{propo-few-unloyal}). Hence, the number of disloyal embeddings $H\into \Gamma$ whose image contains both $w$ and $w'$ is at most $k^{-4}\E_k(n)/n^2$.

So, all in all, there are indeed at most $2k^{-4}\E_k(n)/n^2< k^{-3}\E_k(n)/n^2$ embeddings $H\into \Gamma$ whose image contains both $w$ and $w'$. This finishes the proof.
\end{proof}

\begin{lemma}\label{lemma-Gamma-star-embeddings}Let $j,j'\in V(\tH)$ be distinct and let $w\in W_j$ and $w'\in W_{j'}$. Then there are at least $k^{-1}\E_k(n)/n^2$ embeddings $H\into \Gamma^*$ whose image contains both $w$ and $w'$.
\end{lemma}
\begin{proof}We claim that there is a significant rotation or reflection $\phi$ of $\tH$ such that $j,j'\in \phi(V(H))$. Note that by Lemma \ref{lemma-number-rot-refl} there are at most $2(\tk-k)\leq \frac{1}{2}\log \tk$ rotations and reflections $\phi$ with $j\not\in \phi(V(H))$. Similarly, there at most $2(\tk-k)\leq \frac{1}{2}\log \tk$ rotations or reflections $\phi$ with $j'\not\in \phi(V(H))$. Since by Claim \ref{claim-many-significant} there are at least $2\log \tk$ significant rotations and reflections, there must exist a significant rotation or reflection $\phi$ with $j,j'\in \phi(V(H))$.

Since $\phi$ is significant, we have
\[\vert W_j\vert\cdot \vert W_{j'}\vert \prod_{x\in V(H)\sm \lbrace \phi^{-1}(j),\phi^{-1}(j')\rbrace}\vert W_{\phi(x)}\vert=\prod_{x\in V(H)}\vert W_{\phi(x)}\vert\geq \frac{\E_k(n)}{k}.\]
As $\vert W_j\vert\leq n$ and $\vert W_{j'}\vert\leq n$, we can conclude
\[\prod_{x\in V(H)\sm \lbrace \phi^{-1}(j),\phi^{-1}(j')\rbrace}\vert W_{\phi(x)}\vert\geq \frac{1}{k}\cdot \frac{\E_k(n)}{n^2}.\]
On the other hand we can construct an embedding $\theta: H\into \Gamma^*$ by setting $\theta(\phi^{-1}(j))=w\in W_j$, $\theta(\phi^{-1}(j'))=w'\in W_{j'}$ and choosing any $\theta(x)\in W_{\phi(x)}$ for all $x\in V(H)\sm \lbrace \phi^{-1}(j),\phi^{-1}(j')\rbrace$. As we then have $\theta(x)\in W_{\phi(x)}$ for all $x\in V(H)$, it is easy to see that each such $\theta$ is indeed an embedding $H\into \Gamma^*$. Clearly, the image of $\theta$ contains both $w$ and $w'$. The number of different embeddings $\theta: H\into \Gamma^*$ we can construct in this way is $\prod_{x\in V(H)\sm \lbrace \phi^{-1}(j),\phi^{-1}(j')\rbrace}\vert W_{\phi(x)}\vert\geq k^{-1}\E_k(n)/n^2$, which finishes the proof of the lemma.\end{proof}

\begin{corollary}\label{coro-gamma-star-embeddings} There are at least $M\cdot 2k^{-3}\E_k(n)/n^2$ embeddings $H\into \Gamma^*$ that do not correspond to embeddings $H\into \Gamma$.
\end{corollary}
\begin{proof}Consider any of the $M$ pairs $\lbrace w,w'\rbrace$ of distinct vertices in $V(\Gamma)$ with $a_{\Gamma^*}(w,w')\neq a_{\Gamma}(w,w')$. Let $j,j'\in V(\tH)$ be the indices such that $w\in W_j$ and $w'\in W_{j'}$. Note that by the definition of $\Gamma^*$, the condition $a_{\Gamma^*}(w,w')\neq a_{\Gamma}(w,w')$ implies $j\neq j'$. By Lemma \ref{lemma-Gamma-star-embeddings} there are at least $k^{-1}\E_k(n)/n^2$ embeddings $H\into \Gamma^*$ whose image contains both $w$ and $w'$. As $a_{\Gamma^*}(w,w')\neq a_{\Gamma}(w,w')$, none of these embeddings corresponds to an embedding $H\into \Gamma$.

Each embedding $\theta: H\into \Gamma^*$ that does not correspond to an embedding $H\into \Gamma$ may be counted towards these $k^{-1}\E_k(n)/n^2$ embeddings for at most $\binom{k}{2}\leq k^2/2$ pairs $\lbrace w,w'\rbrace$ (since the image of $\theta$ only consists of $k$ vertices). Hence the total number of embeddings $H\into \Gamma^*$ that do not correspond to embeddings $H\into \Gamma$ is at least
\[\frac{M\cdot k^{-1}\E_k(n)/n^2}{k^2/2}=M\cdot 2k^{-3}\frac{\E_k(n)}{n^2},\]
as desired.
\end{proof}

Combining Lemma \ref{lemma-embeddings-Gamma-Gamma-star} and Corollary \ref{coro-gamma-star-embeddings}, we obtain
\[\emb(H,\Gamma^*)\geq \emb(H,\Gamma)-M\cdot k^{-3}\frac{\E_k(n)}{n^2}+M\cdot 2k^{-3}\frac{\E_k(n)}{n^2}=\emb(H,\Gamma)+M\cdot k^{-3}\frac{\E_k(n)}{n^2}.\]
Recall that on the other hand $\emb(H,\Gamma^*)\leq \emb(H,\Gamma)$. As $\E_k(n)>0$ (since $n\geq \tk\geq k$), this implies $M=0$. Thus, $\Gamma=\Gamma^*$ is a blow-up of $\tH$ with parts $W_j$ for $j\in V(\tH)$.

So far, we have shown that for all positive integers $n\geq \tk$, every $n$-vertex graph $\Gamma$ maximizing $\emb(H,\Gamma)$ is a blow-up of $\tH$.

Let $\Gamma'$ be any $n$-vertex blow-up of $\tH$ with parts $U_j\su V(\Gamma')$ for $j\in V(\tH)$. Then $\sum_{j\in V(\tH)}\vert U_j\vert=n$. By Claim \ref{claim-emb-into-blow}, every embedding $H\into \Gamma'$ is either mapping all of $H$ into a single part $U_j$ or there exists a rotation or reflection $\phi$ of $\tH$ such that each vertex $x\in V(H)$ is mapped into $U_{\phi(x)}$. Conversely, for every rotation or reflection $\phi$ of $\tH$, any map $\theta:V(H)\to \Gamma'$ with $\theta(x)\in U_{\phi(x)}$ for all  $x\in V(H)$ indeed defines an embedding $H\into \Gamma'$. Thus, 
\[\emb(H,\Gamma')=\sum_{\phi}\left(\prod_{x\in V(H)}\vert U_{\phi(x)}\vert\right)+\sum_{j\in V(\tH)}\emb(H,\Gamma'[U_j]),\]
where the fist sum is over all the different rotations and reflections $\phi$ of $\tH$. In particular we obtain
\[\emb(H,\Gamma')\leq \sum_{\phi}\left(\prod_{x\in V(H)}\vert U_{\phi(x)}\vert\right)+\sum_{j\in V(\tH)}\emb(H,\vert U_j\vert).\]
Note that equality holds if and only if for each part $U_j$, the graph $\Gamma'[U_j]$, given its number $\vert U_j\vert$ of vertices, maximizes the number of embeddings of $H$.

In particular, we can conclude that for the graph $\Gamma$, each of its parts $W_j$ must maximize the number of embeddings of $H$ among all graphs with $\vert W_j\vert$ vertices.

For non-negative integers $n_j$ for $j\in V(\tH)$ with $\sum_{j\in V(\tH)} n_j=n$, set
\[N_\phi((n_j)_{j\in V(\tH)})=\prod_{x\in V(H)} n_{\phi(x)}=\prod_{j\in \phi(V(H))}n_j\]
for every rotation or reflection $\phi$ of $\tH$. Furthermore, set
\[T((n_j)_{j\in V(\tH)})=\sum_{\phi}N_\phi((n_j)_{j\in V(\tH)})+\sum_{j\in V(\tH)}\emb(H,n_j).\]
Then $T((n_j)_{j\in V(\tH)})$ is precisely the number of embeddings $H\into \Gamma'$ for an $n$-vertex blow-up $\Gamma'$ of $\tH$ with parts of sizes $n_j$ for $j\in V(\tH)$ and such that each part is maximizing the number of embeddings of $H$ (among all graphs with the same number $n_j$ of vertices).

In particular,
\[\emb(H,\Gamma)=T((\vert W_j\vert)_{j\in V(\tH)}),\]
and as $\Gamma$ maximizes the number of embeddings of $H$ among all $n$-vertex graphs, we must have
\[T((n_j)_{j\in V(\tH)})\leq T((\vert W_j\vert)_{j\in V(\tH)})\]
for any non-negative integers $n_j$ for $j\in V(\tH)$ with $\sum_{j\in V(\tH)} n_j=n$.

Note that for each rotation or reflection $\phi$ of $\tH$, the number of $\phi$-regulated embeddings $H\into \Gamma$ is precisely 
\[N_\phi((\vert W_j\vert)_{j\in V(\tH)})=\prod_{x\in V(H)} \vert W_{\phi(x)}\vert\]
Hence, by Corollary \ref{coro-number-regulated-embeddings} we have
\[\sum_{\phi}N_\phi((\vert W_j\vert)_{j\in V(\tH)})\geq (\tk^{1/4}-1)\cdot \E_k(n).\]
Furthermore recall that by Lemma \ref{lemma-size-Wj} we have $1\leq \vert W_j\vert\leq k^{-4/5}\cdot n$ for all $j\in V(\tH)$. Thus, the following Proposition implies that $\vert \vert W_j\vert-\vert W_{j'}\vert\vert\leq 1$ for all $j,j'\in V(\tH)$. In other words, $\Gamma$ is a balanced blow-up with parts $W_j$ for $j\in V(\tH)$.

\begin{proposition}\label{propo-optimization} Let $n_j$ for $j\in V(\tH)$ be non-negative integers with $\sum_{j\in V(\tH)} n_j=n$ and such that $T((n_j)_{j\in V(\tH)})$ is maximized among all choices of $n_j\geq 0$ with $\sum_{j\in V(\tH)} n_j=n$. Let us furthermore assume that
\[\sum_{\phi}N_\phi((n_j)_{j\in V(\tH)})\geq (\tk^{1/4}-1)\cdot \E_k(n)\]
(where the sum is over all the different rotations and reflections $\phi$ of $\tH$) and $1\leq n_j\leq k^{-4/5}\cdot n$ for all $j\in V(\tH)$. Then $\vert n_j-n_{j'}\vert\leq 1$ for all $j,j'\in V(\tH)$.
\end{proposition}

We will prove Proposition \ref{propo-optimization} in Section \ref{sect-optimization}.

So we have seen that for all positive integers $n\geq \tk$, every $n$-vertex graph $\Gamma$ maximizing $\emb(H,\Gamma)$ is a balanced blow-up of $\tH$, and furthermore each of its parts must also maximize the number of embeddings of $H$ among all graphs with the same number vertices. Thus, each of the parts must also be a balanced blow-up of $\tH$. Continuing like that, we can conclude that every $n$-vertex graph $\Gamma$ maximizing $\emb(H,\Gamma)$ is a balanced iterated blow-up of $\tH$. This proves Theorem \ref{thm-H-reasonable-blow-up}.

\section{Proofs of Lemmas \ref{lemma-loyal-unregulated}, \ref{lemma-phi-loyal-between-W}, \ref{lemma-W-partition}, and \ref{lemma-size-Wj}}
\label{sect-properties-Wj}

In this section, we will study the properties of the sets $W_j$ for $j\in V(\tH)$ and in particular prove Lemmas \ref{lemma-loyal-unregulated}, \ref{lemma-phi-loyal-between-W}, \ref{lemma-W-partition}, and \ref{lemma-size-Wj}.

\begin{claim}\label{claim-W-disjoint}The sets $W_j$ for $j\in V(\tH)$ are all disjoint.\end{claim}
\begin{proof}Suppose there is a vertex $w\in W_j\cap W_{j'}$ for distinct $j,j'\in V(\tH)$. Then there must be at least $(1-2\eps_1)k$ vertices $i\in V(H)\sm(S\cup \lbrace j,j'\rbrace)$ satisfying $w\not\in V_i$ as well as $a_\Gamma(w,V_i)=a_{\tH}(j,i)$ and $a_\Gamma(w,V_i)=a_{\tH}(j',i)$. In particular, we must have $a_{\tH}(j,i)=a_{\tH}(j',i)$ for at least $(1-2\eps_1)k$ vertices $i\in V(H)\sm \lbrace j,j'\rbrace$. So there can be at most $2\eps_1 k$ vertices in $V(H)\sm \lbrace j,j'\rbrace$ that are adjacent to exactly one of the vertices $j$ and $j'$. As $2\eps_1<q$ by (\ref{ineq-eps1-q}), this contradicts condition (b) in Definition \ref{defi-reasonable}. Hence, there cannot be a vertex $w\in W_j\cap W_{j'}$ for distinct $j,j'\in V(\tH)$.
\end{proof}

The following lemma is the main tool for proving Lemmas \ref{lemma-loyal-unregulated} and  \ref{lemma-phi-loyal-between-W}.

\begin{lemma}\label{lemma-phi-loyal-wrong-W}Let $\phi$ be a rotation or reflection of $\tH$ and fix vertices $x, x'\in V(H)$ and $w, w'\in V(\Gamma)$ such that $x\neq x'$ and $w\not\in W_{\phi(x)}$. Then there are at most $k^{-8}\cdot \E_k(n)/n^{2}$ different $\phi$-loyal embeddings $\theta:H\into\Gamma$ with $\theta(x)=w$ and $\theta(x')=w'$.
\end{lemma}
\begin{proof}Let $I\su V(H)\sm(S\cup \lbrace \phi(x)\rbrace)$ denote the set of vertices $i\in V(H)\sm(S\cup \lbrace \phi(x)\rbrace)$ such that $w\not\in V_i$ and $a_\Gamma(w,V_i)\neq a_{\tH}(\phi(x),i)$. As $w\not\in W_{\phi(x)}$, by Definition \ref{defi-Wj} we have (noting that there is at most one $i$ with $w\in V_i$, since the sets $V_i$ are all disjoint)
\[\vert I\vert\geq \vert V(H)\sm(S\cup \lbrace \phi(x)\rbrace)\vert-1-(1-\eps_1)k=k-\vert S\vert -2-(1-\eps_1)k=\eps_1 k-\vert S\vert -2.\]
Recall that $\vert S\vert \leq \frac{5}{q}\log k$. Furthermore, let us define $I'\su V(\tH)\sm(\phi^{-1}(S)\cup \lbrace x\rbrace)$ by $I'=\phi^{-1}(I)$. Then
\[\vert I'\vert=\vert I\vert\geq \eps_1 k-\vert S\vert -2\geq \eps_1 k-\frac{5}{q}\log k -2.\]
and for each $y\in I'$ we have $w\not\in V_{\phi(y)}$ and $a_\Gamma(w,V_{\phi(y)})\neq a_{\tH}(\phi(x),\phi(y))$.

For each $\phi$-loyal embedding $\theta:H\into\Gamma$ with $\theta(x)=w$ and $\theta(x')=w'$, there is a set $Y\su V(H)\sm \lbrace x,x'\rbrace$ of size $\vert Y\vert \geq (1-\eps_2)k-2$ such that for all $y\in Y$ we have $\phi(y)\in V(H)\sm S$ and $\theta(y)\in V_{\phi(y)}$. For each such $\theta$, we can choose such a set $Y$ with $\vert Y\vert = \lceil(1-\eps_2)k-2\rceil$ (by deleting some of the elements if $Y$ is larger).

The number of subsets $Y\su V(H)\sm \lbrace x,x'\rbrace$ of size $\vert Y\vert = \lceil(1-\eps_2)k-2\rceil$ is
\begin{multline*}
\binom{k-2}{ \lceil(1-\eps_2)k-2\rceil}\leq \binom{k}{ k-2-\lfloor\eps_2 k\rfloor}=\binom{k}{ \lfloor\eps_2 k\rfloor+2}\leq \left(\frac{ek}{\lfloor\eps_2 k\rfloor+2}\right)^{\lfloor\eps_2 k\rfloor+2}\leq \left(\frac{ek}{\eps_2 k}\right)^{\eps_2 k+2}\\
=\left(\frac{e}{\eps_2}\right)^{\eps_2 k+2}\leq \left(\eps_2^{-5/4}\right)^{(5/4)\eps_2 k}=\eps_2^{-(25/16)\eps_2 k}\leq \eps_2^{-2\eps_2 k},
\end{multline*}
using that $\eps_2\leq e^{-4}$ by (\ref{ineq-eps2-100}) and $2\leq \frac{1}{4}\eps_2 k$ by (\ref{ineq-eps5-logk-k}).

For each subset $Y\su V(H)\sm \lbrace x,x'\rbrace$ of size $\vert Y\vert = \lceil(1-\eps_2)k-2\rceil$ with $\phi(y)\in V(H)\sm S$ for all $y\in Y$, let us bound the number of $\phi$-loyal embeddings $\theta:H\into\Gamma$ with $\theta(x)=w$ and $\theta(x')=w'$ such that $\theta(y)\in V_{\phi(y)}$ for all $y\in Y$. First, let us consider the number of possibilities for $\theta\vert_{Y\cup \lbrace x,x'\rbrace}$. We need $\theta(x)=w$ and $\theta(x')=w'$. Furthermore, for each $y\in Y$ we need to choose $\theta(y)\in V_{\phi(y)}\sm \lbrace w\rbrace$ such that
\[a_{\Gamma}(w,\theta(y))=a_{\Gamma}(\theta(x),\theta(y))=a_H(x,y)=a_{\tH}(\phi(x), \phi(y)).\]
For each $y\in Y\sm I'$, there are clearly at most $\vert V_{\phi(y)}\vert$ possibilities to choose $\theta(y)\in V_{\phi(y)}$. On the other hand, for each $y\in Y\cap I'$, we have $a_\Gamma(w,V_{\phi(y)})\neq a_{\tH}(\phi(x),\phi(y))$ (and $w\not\in V_{\phi(y)}$). Hence there are at most $\frac{1}{2}\vert V_{\phi(y)}\vert$ possibilities for $\theta(y)\in V_{\phi(y)}$ with $a_{\Gamma}(w,\theta(y))=a_{\tH}(\phi(x), \phi(y))$. So all in all we see that the number of possibilities for $\theta\vert_{Y\cup \lbrace x,x'\rbrace}$ is at most
\[\prod_{y\in Y\sm I'}\vert V_{\phi(y)}\vert\cdot \prod_{y\in Y\cap I'}\frac{1}{2}\vert V_{\phi(y)}\vert=2^{-\vert Y\cap I'\vert}\prod_{y\in Y}\vert V_{\phi(y)}\vert\leq 2^{-\vert Y\cap I'\vert}\E_{\vert Y\vert}\left(\sum_{y\in Y}\vert V_{\phi(y)}\vert\right).\]
Note that (using $\tk\leq k+\frac{1}{4}\log \tk\leq k+\log k$)
\[\vert Y\cap I'\vert\geq \vert Y\vert+\vert I'\vert -\tk\geq (1-\eps_2)k-2+\eps_1 k-\frac{5}{q}\log k -2-k-\log k\geq \eps_1 k-\eps_2 k-\frac{6}{q}\log k-4.\]
As $\eps_2<\frac{1}{4}\eps_1$ by (\ref{ineq-eps2-eps1}), $\frac{6}{q}\log k\leq \frac{1}{8}\eps_1 k$ by (\ref{ineq-eps5-logk-k}) and $4\leq \frac{1}{8}\eps_1 k$ also by (\ref{ineq-eps5-logk-k}), we obtain $\vert Y\cap I'\vert\geq \frac{1}{2}\eps_1 k$. Hence the number of possibilities for $\theta\vert_{Y\cup \lbrace x,x'\rbrace}$ is at most
\[2^{-\eps_1 k/2}\E_{\vert Y\vert}\left(\sum_{y\in Y}\vert V_{\phi(y)}\vert\right)=2^{-\eps_1 k/2}\E_{\vert Y\vert}(n-\vert U\vert),\]
where $U$ denotes the set $V(\Gamma)\sm \bigcup_{y\in Y} V_{\phi(y)}$ (recall that the sets $V_{\phi(y)}$ are all disjoint).

Now, for each of the possible maps $\theta\vert_{Y\cup \lbrace x,x'\rbrace}$, let us bound the number of ways to extend $\theta\vert_{Y\cup \lbrace x,x'\rbrace}$ to a $\phi$-loyal embedding $\theta:H\into\Gamma$. By Definition \ref{defi-loyal}, the image of $\theta$ can only contain at most one vertex from each of the sets $V_{\phi(y)}$ for $y\in Y$ and since we already have $\theta(y)\in V_{\phi(y)}$, we must have $\theta(v)\in U$ for all $v\in V(H)\sm (Y\cup \lbrace x,x'\rbrace)$. Note that the set $Y\cup \lbrace x,x'\rbrace$ has size $\vert Y\vert +2\geq (1-\eps_2)k\geq (1-q)k$, so by Lemma \ref{lemma-large-signature} the set $Y\cup \lbrace x,x'\rbrace$ is a signature. Hence we can apply Lemma \ref{lemma-signature-extensions} and obtain that the number of $\phi$-loyal embeddings $\theta:H\into\Gamma$ extending $\theta\vert_{Y\cup \lbrace x,x'\rbrace}$ is at most $\E_{k-\vert Y\vert-2}(\vert U\vert)$.

So for each of the possibilities for $Y$, the number of $\phi$-loyal embeddings $\theta:H\into\Gamma$ is at most
\[2^{-\eps_1 k/2}\E_{\vert Y\vert}(n-\vert U\vert)\cdot \E_{k-\vert Y\vert-2}(\vert U\vert)\leq 2^{-\eps_1 k/2}\E_{k-2}(n)\leq 2^{-\eps_1 k/2}e^{6}\left(\frac{k}{n}\right)^{2}\E_k(n)\]
(here we used Lemma \ref{lemma-functions-E}(ii) and Lemma \ref{lemma-functions-E}(iii), recalling that $n\geq \tk\geq k$).

Multiplying this by the number of possibilities for $Y$, we see that the number of $\phi$-loyal embeddings $\theta:H\into\Gamma$ with $\theta(x)=w$ and $\theta(x')=w'$ is at most
\begin{multline*}
\eps_2^{-2\eps_2 k}\cdot 2^{-\eps_1 k/2}e^{6}\left(\frac{k}{n}\right)^{2}\E_k(n)=\exp\left(2\log(1/\eps_2)\eps_2 k-\frac{\log 2}{2}\eps_1 k\right)e^{6}k^{2}\frac{\E_k(n)}{n^2}\\
\leq \exp\left(-\frac{\log 2}{4}\eps_1 k\right)e^{6}k^{2}\frac{\E_k(n)}{n^2}\leq \exp\left(-10(\log k)-6\right)e^{6}k^{2}\frac{\E_k(n)}{n^2}=k^{-8}\frac{\E_k(n)}{n^2},
\end{multline*}
where for the first inequality we used (\ref{ineq-eps2-eps1}) and for the second inequality we used (\ref{ineq-eps5-logk-k}). This finishes the proof of Lemma \ref{lemma-phi-loyal-wrong-W}.
\end{proof}

From Lemma \ref{lemma-phi-loyal-wrong-W} we obtain the following corollary, from which we will deduce Lemma  \ref{lemma-loyal-unregulated}.

\begin{corollary}\label{coro-phi-loyal-wrong-W}Let $\phi$ be a rotation or reflection of $\tH$ and let $x\in V(H)$ and $w\in V(\Gamma)$ be such that $w\not\in W_{\phi(x)}$. Then there are at most $k^{-8}\cdot \E_k(n)/n$ different $\phi$-loyal embeddings $\theta:H\into\Gamma$ with $\theta(x)=w$.
\end{corollary}
\begin{proof}Fix any $x'\in V(H)$ with $x'\neq x$ (recall $\vert V(H)\vert = k \geq \tk/2\geq 10^{199}$). Then by Lemma \ref{lemma-phi-loyal-wrong-W}, for each $w'\in V(\Gamma)$ there are at most $k^{-8}\cdot \E_k(n)/n^{2}$ different $\phi$-loyal embeddings $\theta:H\into\Gamma$ with $\theta(x)=w$ and $\theta(x')=w'$. Adding this up for all $w'\in V(\Gamma)$ shows that there are at most $k^{-8}\cdot \E_k(n)/n$ different $\phi$-loyal embeddings $\theta:H\into\Gamma$ with $\theta(x)=w$.
\end{proof}

\begin{proof}[Proof of Lemma \ref{lemma-loyal-unregulated}]
It suffices to show that for each fixed $x\in V(H)$, there are at most $k^{-8}\cdot \E_k(n)$ different $\phi$-loyal embeddings $\theta:H\into\Gamma$ with $\theta(x)\not\in W_{\phi(x)}$. For every $w\in V(\Gamma)\sm W_{\phi(x)}$, by Corollary \ref{coro-phi-loyal-wrong-W} there are at most $k^{-8}\cdot \E_k(n)/n$ different $\phi$-loyal embeddings $\theta:H\into\Gamma$ with $\theta(x)=w$. Adding this up for all $\vert V(\Gamma)\sm W_{\phi(x)}\vert\leq n$ choices for $w$ gives the desired statement.
\end{proof}

Furthermore, Corollary \ref{coro-phi-loyal-wrong-W} also yields the following statement, which we will need later in order to prove Lemma \ref{lemma-W-partition}.

\begin{corollary}\label{coro-vertex-W-rest}Let $w\in V(\Gamma)$ and assume that $w\not\in W_j$ for all $j\in V(\tH)$. Then $w$ is contained in the image of at most $k^{-5}\E_k(n)/n$ different loyal embeddings $H\into \Gamma$.
\end{corollary}
\begin{proof}Fix some rotation or reflection $\phi$ of $\tH$. For each $x\in V(H)$, by Corollary \ref{coro-phi-loyal-wrong-W} there are at most $k^{-8}\cdot \E_k(n)/n$ different $\phi$-loyal embeddings $\theta:H\into\Gamma$ with $\theta(x)=w$. Summing this up, there are at most $k^{-7}\cdot \E_k(n)/n$ different $\phi$-loyal embeddings $H\into\Gamma$ whose image contains $w$. Since there are at most $2\tk\leq 4k\leq k^2$ rotations and reflections $\phi$ of $\tH$, all in all we obtain that there are at most $k^{-5}\E_k(n)/n$ loyal embeddings $H\into \Gamma$ whose image contains $w$.
\end{proof}

Finally, we can also deduce Lemma \ref{lemma-phi-loyal-between-W} from Lemma \ref{lemma-phi-loyal-wrong-W}.

\begin{proof}[Proof of Lemma \ref{lemma-phi-loyal-between-W}]
Fix some rotation or reflection $\phi$ of $\tH$. For any $\phi$-loyal embedding $\theta: H\into \Gamma$ whose image contains both $w$ and $w'$, there  are $x,x'\in V(H)$ with $\theta(x)=w$ and $\theta(x')=w'$. As $j\neq j'$, by Claim \ref{claim-W-disjoint} the sets $W_j$ and $W_{j'}$ are disjoint and therefore $w\neq w'$. Hence we must have $x\neq x'$. Furthermore $a_\Gamma(w,w')=a_\Gamma(\theta(x),\theta(x'))=a_H(x,x')=a_{\tH}(\phi(x),\phi(x'))$. So the assumption $a_\Gamma(w,w')\neq  a_{\tH}(j,j')$ implies $a_{\tH}(\phi(x),\phi(x'))\neq a_{\tH}(j,j')$ and in particular $j\neq \phi(x)$ or $j'\neq \phi(x')$. By Claim \ref{claim-W-disjoint}, this implies $w\not\in W_{\phi(x)}$ or $w'\not\in W_{\phi(x')}$.

For any choice of $x,x'\in V(H)$ with $x\neq x'$ and $w\not\in W_{\phi(x)}$, by Lemma \ref{lemma-phi-loyal-wrong-W} there at most $k^{-8}\E_k(n)/n^2$ different $\phi$-loyal embeddings $\theta: H\into \Gamma$ with $\theta(x)=w$ and $\theta(x')=w'$. For any choice of $x,x'\in V(H)$ with $x\neq x'$ and $w'\not\in W_{\phi(x')}$ the same statement holds analogously (in this case apply Lemma  \ref{lemma-phi-loyal-wrong-W} with the roles of $x$ and $x'$ interchanged and the roles of $w$ and $w'$ interchanged).

There are at most $k^2$ possibilities to choose $x,x'\in V(H)$ with $x\neq x'$ and $w\not\in W_{\phi(x)}$ or $w'\not\in W_{\phi(x')}$. Hence there are at most  $k^{-6}\E_k(n)/n^2$ different $\phi$-loyal embeddings $H\into \Gamma$ whose image contains both $w$ and $w'$.

There are at most $2\tk\leq 4k\leq k^2$ rotations and reflections $\phi$ of $\tH$. Summing over all possible $\phi$ yields that there are at most  $k^{-4}\E_k(n)/n^2$ loyal embeddings $H\into \Gamma$ whose image contains both $w$ and $w'$.\end{proof}

For Lemma \ref{lemma-W-partition}, we have already seen in Claim \ref{claim-W-disjoint} that the sets $W_j$ are disjoint. The following claim states that their union is $V(\Gamma)$, as desired.

\begin{claim}\label{claim-W-union}The union of the sets $W_j$ for all $j\in V(\tH)$ is the whole vertex set $V(\Gamma)$.
\end{claim}
\begin{proof}Suppose there was a vertex $w\in V(\Gamma)$ with $w\not\in W_j$ for all $j\in V(\tH)$. By Corollary \ref{coro-vertex-W-rest}, the vertex $w$ is contained in the image of at most $k^{-5}\E_k(n)/n$ loyal embeddings $H\into \Gamma$. Furthermore, for every $x\in V(H)$, by Corollary \ref{coro-few-unloyal} there are at most $k^{-6}\E_k(n)/n$ disloyal embeddings $\theta: H\into \Gamma$ with $\theta(x)=w$. Hence, summing over all $x\in V(H)$, the vertex $w$ is contained in the image of at most $k^{-5}\E_k(n)/n$ disloyal embeddings $H\into \Gamma$. Thus, all in all $w$ is in the image of at most $2k^{-5}\E_k(n)/n$ embeddings $H\into \Gamma$. As $k>2$, this is a contradiction to Corollary \ref{coro-each-vertex-many-embeddings}. Thus, there cannot be a vertex $w\in V(\Gamma)$ with $w\not\in W_j$ for all $j\in V(\tH)$.
\end{proof}

Now Lemma \ref{lemma-W-partition} follows immediately from Claims \ref{claim-W-disjoint} and \ref{claim-W-union}. Finally, let us prove Lemma \ref{lemma-size-Wj}.

\begin{proof}[Proof of Lemma \ref{lemma-size-Wj}]
Suppose that $W_j=\emptyset$ for some $j\in V(\tH)$. Then, by Definition \ref{defi-regulated}, there cannot be any $\phi$-regulated embeddings $H\into \Gamma$ for any rotation or reflection $\phi$ with $j\in\phi(V(H))$. However, by Lemma \ref{lemma-number-rot-refl} there are at most $2(\tk-k)\leq \frac{1}{2}\log \tk$ rotations or reflections $\phi$ with $j\not\in\phi(V(H))$ and by Claim \ref{claim-number-regulated} for each of them there are at most $\E_k(n)$ different $\phi$-regulated embeddings $H\into \Gamma$. Hence the total number of regulated embeddings is most $\frac{1}{2}(\log \tk)\cdot \E_k(n)<(\tk^{1/4}-1)\cdot \E_k(n)$, using (\ref{ineq-k-logtk}). This is a contradiction to Corollary \ref{coro-number-regulated-embeddings}. Thus, we must have $\vert W_j\vert\geq 1$ for every $j\in V(\tH)$.

Now suppose that $\vert W_j\vert> k^{-4/5}\cdot n$ for some $j\in V(\tH)$. We claim that then for each rotation or reflection $\phi$ of $\tH$, the number of $\phi$-regulated embeddings $H\into \Gamma$ is at most $n\cdot \E_{k-1}(n-k^{-4/5}n)$. In order to check this claim, let us fix some rotation or reflection $\phi$. If $j\in \phi(V(H))$, let $y\in V(H)$ be such that $\phi(y)=j$. Otherwise, fix some arbitrary $y\in V(H)$. Then in either case we have
\[\sum_{x\in V(H)\sm\lbrace y\rbrace}\vert W_{\phi(x)}\vert\leq n-\vert W_j\vert\leq n-k^{-4/5}n,\]
since $W_j$ and $W_{\phi(x)}$ for $x\in V(H)\sm\lbrace y\rbrace$ are disjoint subsets of $V(\Gamma)$, see Claim \ref{claim-W-disjoint}. Therefore, by Claim \ref{claim-number-regulated}, the number of $\phi$-regulated embeddings $H\into \Gamma$ is at most
\[\prod_{x\in V(H)}\vert W_{\phi(x)}\vert=\vert W_{\phi(y)}\vert\cdot \prod_{x\in V(H)\sm\lbrace y\rbrace}\vert W_{\phi(x)}\vert\leq n\cdot \E_{k-1}(n-k^{-4/5}n),\]
using Lemma \ref{lemma-functions-E}(i).

Thus, for each rotation or reflection $\phi$ of $\tH$, the number of $\phi$-regulated embeddings $H\into \Gamma$ is indeed at most $n\cdot \E_{k-1}(n-k^{-4/5}n)$. But then the total number of regulated embeddings $H\into \Gamma$ is at most
\[2\tk\cdot n\cdot \E_{k-1}(n-k^{-4/5}n)\leq 2\tk n\cdot e^3\exp(-k^{-4/5}\cdot k/2)\cdot \frac{k}{n}\E_k(n)\leq 4e^3\cdot k^2\cdot \exp(-k^{1/5}/2)\cdot \E_k(n).\]
But by (\ref{ineq-k-logtk}) and $\tk\geq 10^{200}$ we have $\exp(k^{1/5}/2)\geq \tk^3\geq  4e^3\cdot k^2$ and hence this number is at most $\E_k(n)$. This again contradicts Corollary \ref{coro-number-regulated-embeddings}. Thus, we have $\vert W_j\vert\leq k^{-4/5}\cdot n$ for every $j\in V(\tH)$.
\end{proof}

\section{Preparations for the proof of Proposition \ref{propo-few-unloyal}}\label{sect-properties-Vi}

In this section, we will establish several important properties of the sets $V_i'$ and $V_i$ for $i\in V(H)\sm S$. These properties will be used in the proof of Proposition \ref{propo-few-unloyal} in the next section.

Recall that $S\su V(H)$ is a signature of $H$ of size $\vert S\vert\leq \frac{5}{q}\log k$, and $\psi :S\to V(\Gamma)$ is a map such that there are at least $\emb(H,\Gamma)/n^{\vert S\vert}$ embeddings $H\into \Gamma$ extending $\psi$. Furthermore, recall from the beginning of Section \ref{sect-proof-blow-up} that for each $i\in V(H)\sm S$ the set $V_i'\su V(\Gamma)$ consists of those vertices $v\in V(\Gamma)$ for which there exists an embedding $H\into \Gamma$ extending $\psi$ and mapping $i$ to $v$. We saw that the sets $V_i'$ for $i\in V(H)\sm S$ are all disjoint. Recall that we defined in Definition \ref{defi-bad} what it means for a vertex $v\in V_i'$ to be bad.

\begin{lemma}\label{lemma-bad-vertex}Let $i\in V(H)\sm S$ and assume that $v\in V_i'$ is a bad vertex. Then there are at most $\E_k(n)/n^{\vert S\vert+1}$ embeddings $\theta: H\into \Gamma$ with $\theta(i)=v$ and $\theta(s)=\psi(s)$ for all $s\in S$.
\end{lemma}
\begin{proof}Let $U\su V(\Gamma)$ be the complement of the set of vertices $w\in V(\Gamma)$ such that $(v,w)$ is a bad pair. Since $v$ is a bad vertex, there are at least $\eps_5 n $ such bad pairs, so we have $\vert U\vert\leq (1-\eps_5)n$.

Recall that every embedding $\theta: H\into \Gamma$ extending $\psi$ satisfies $\theta(j)\in V_j'$ for all $j\in V(H)\sm S$. We claim that every embedding $\theta$ with the properties in the lemma must furthermore satisfy $\theta(j)\in U$ for all $j\in V(H)\sm (S\cup \lbrace i\rbrace)$. Indeed, for each $j\in V(H)\sm (S\cup \lbrace i\rbrace)$ we have
\[a_\Gamma(v,\theta(j))=a_\Gamma(\theta(i),\theta(j))=a_H(i,j).\]
As $\theta(j)\in V_j'$ and $v\in V_i'$, this means that $(v,\theta(j))$ is not a bad pair, so $\theta(j)\in U$.

Now let $X=S\cup\lbrace i\rbrace$ (note that $X$ is a signature of $H$) and define $f:X\to V(\Gamma)$ by $f(s)=\psi(s)$ for $s\in S$ and $f(i)=v$. Then all the embeddings $\theta: H\into \Gamma$ with the properties in the lemma must satisfy $\theta(x)=f(x)$ for all $x\in X$ and $\theta(j)\in U$ for all $j\in  V(H)\sm X$. So by Lemma \ref{lemma-signature-extensions} the number of such embeddings is at most
\[\E_{k-\vert X\vert}(\vert U\vert)\leq \E_{k-\vert S\vert-1}((1-\eps_5) n)\leq e^{3(\vert S\vert+1)-\eps_5 k/2}\left(\frac{k}{n}\right)^{\vert S\vert+1}\E_k(n)=
e^{(\log k+3)(\vert S\vert+1)-\eps_5 k/2}\frac{\E_k(n)}{n^{\vert S\vert+1}}.\]
Here we used Lemma \ref{lemma-functions-E}(iii), recalling that $n\geq \tk\geq k$. As
\[(\log k+3)(\vert S\vert+1)\leq (\log k+3)\left(\frac{5}{q}\log k+1\right)\leq 2\log k\cdot \frac{6}{q}\log k=\frac{12}{q}(\log k)^2\leq \frac{\eps_5 k}{2}\]
by (\ref{ineq-eps5-logk-k}), this finishes the proof of the lemma.
\end{proof}

Recall that for each $i\in V(H)\sm S$, we defined $V_i$ to be the set obtained from $V_i'$ by deleting all bad vertices in $V_i'$.

\begin{lemma}\label{lemma-embeddings-Vi} There are at least $\E_k(n)/n^{\vert S\vert}$ embeddings $\theta: H\into \Gamma$ with $\theta(s)=\psi(s)$ for all $s\in S$ and $\theta(i)\in V_i$ for all $i\in V(H)\sm S$.
\end{lemma}
\begin{proof} By definition of the sets $V_i'$, every embedding $\theta: H\into \Gamma$ with $\theta(s)=\psi(s)$ for all $s\in S$ must satisfy $\theta(i)\in V_i'$ for all $i\in V(H)\sm S$. The total number of embeddings $\theta: H\into \Gamma$ with $\theta(s)=\psi(s)$ for all $s\in S$ is at least
\[\frac{\emb(H,\Gamma)}{n^{\vert S\vert}}\geq \tk^{1/4}\cdot \frac{\E_k(n)}{n^{\vert S\vert}}\geq 2\cdot \frac{\E_k(n)}{n^{\vert S\vert}},\]
where we used (\ref{eq-emb-H-Gamma-big}) for the first inequality.

For every $i\in V(H)\sm S$ and every vertex $v\in V_i'\sm V_i$, by Lemma \ref{lemma-bad-vertex} there are at most $\E_k(n)/n^{\vert S\vert+1}$ embeddings $\theta: H\into \Gamma$ with $\theta(s)=\psi(s)$ for all $s\in S$ and $\theta(i)=v$. Hence, for each $i\in V(H)\sm S$, there are at most $\vert V_i'\sm V_i\vert\cdot \E_k(n)/n^{\vert S\vert+1}$ embeddings $\theta: H\into \Gamma$ with $\theta(s)=\psi(s)$ for all $s\in S$ and $\theta(i)\not\in V_i$.
Thus, in total there are at most
\[\sum_{i\in V(H)\sm S}\vert V_i'\sm V_i\vert\cdot \frac{\E_k(n)}{n^{\vert S\vert+1}}\leq n\cdot \frac{\E_k(n)}{n^{\vert S\vert+1}}=\frac{\E_k(n)}{n^{\vert S\vert}}\]
embeddings $\theta: H\into \Gamma$ with $\theta(s)=\psi(s)$ for all $s\in S$ and $\theta(i)\not\in V_i$ for some $i\in V(H)\sm S$.

So there must be at least $2\E_k(n)/n^{\vert S\vert}-\E_k(n)/n^{\vert S\vert}=\E_k(n)/n^{\vert S\vert}$ embeddings $\theta: H\into \Gamma$ with $\theta(s)=\psi(s)$ for all $s\in S$ and $\theta(i)\in V_i$ for all $i\in V(H)\sm S$.
\end{proof}

On the other hand, the number of embeddings $\theta: H\into \Gamma$ with $\theta(s)=\psi(s)$ for all $s\in S$ and $\theta(i)\in V_i$ for all $i\in V(H)\sm S$ is clearly at most $\prod_{i\in V(H)\sm S} \vert V_i\vert$. So we obtain from Lemma \ref{lemma-embeddings-Vi} that
\begin{equation}\label{ineq-product-Vi}
\prod_{i\in V(H)\sm S} \vert V_i\vert\geq \frac{\E_k(n)}{n^{\vert S\vert}}.
\end{equation}
In particular, we can conclude that each of the sets $V_i$ for $i\in V(H)\sm S$ is non-empty.

As the sets $V_i'$ for $i\in V(H)\sm S$ are all disjoint, the sets $V_i$ for $i\in V(H)\sm S$ are disjoint as well. Set
\[V_{\textnormal{rest}}=V(\Gamma)\sm\bigcup_{i\in V(H)\sm S} V_i,\]
then the sets $V_{\textnormal{rest}}$ and $V_i$ for all $i\in V(H)\sm S$ together form a partition of $V(\Gamma)$.

\begin{lemma}\label{lemma-V-small}For every $i\in V(H)\sm S$, we have
\[\vert V_{i}\vert+\vert V_{\textnormal{rest}}\vert \leq \frac{24}{q}\cdot \frac{(\log k)^2}{k}n.\]
\end{lemma}
\begin{proof}Fix some $i\in V(H)\sm S$ and let $\mu=(\vert V_i\vert +\vert V_{\textnormal{rest}}\vert)/n$. Then
\[\sum_{j\in V(H)\sm(S\cup \lbrace i\rbrace)}\vert V_j\vert= n-\vert V_i\vert -\vert V_{\textnormal{rest}}\vert=(1-\mu)n.\]
Hence  Lemma \ref{lemma-functions-E}(i) and Lemma \ref{lemma-functions-E}(iii) yield (recall that $n\geq \tk\geq k$)
\[\prod_{j\in V(H)\sm(S\cup \lbrace i\rbrace)}\vert V_j\vert\leq \E_{k-\vert S\vert-1}((1-\mu)n)\leq e^{3(\vert S\vert+1)-\mu k/2}\left(\frac{k}{n}\right)^{\vert S\vert+1}\E_k(n).\]
So
\[\prod_{j\in V(H)\sm S}\vert V_j\vert\leq n\cdot \prod_{j\in V(H)\sm(S\cup \lbrace i\rbrace)}\vert V_j\vert\leq  n\cdot e^{3(\vert S\vert+1)-\mu k/2}\left(\frac{k}{n}\right)^{\vert S\vert+1}\E_k(n)=k^{\vert S\vert+1}e^{3(\vert S\vert+1)-\mu k/2}\frac{\E_k(n)}{n^{\vert S\vert}}.\]
Combining this with (\ref{ineq-product-Vi}), we obtain
\[1\leq k^{\vert S\vert+1}e^{3(\vert S\vert+1)-\mu k/2}=e^{(3+\log k)(\vert S\vert+1)}e^{-\mu k/2}.\]
In other words,
\[\frac{1}{2}\mu k\leq (3+\log k)(\vert S\vert+1)\leq 2(\log k) \cdot \left(\frac{5}{q}(\log k)+1\right)\leq 2(\log k) \cdot \frac{6}{q}(\log k)=\frac{12}{q}(\log k)^2.\]
So we can conclude that
\[\mu\leq \frac{24}{q}\cdot \frac{(\log k)^2}{k}.\]
Thus,
\[\vert V_{i}\vert+\vert V_{\textnormal{rest}}\vert= \mu n\leq \frac{24}{q}\cdot \frac{(\log k)^2}{k}n,\]
which is the desired bound.
\end{proof}

The next lemma provides an upper bound for the number of disjoint pairs $\lbrace i,j\rbrace\su V(H)\sm S$ such that $a_\Gamma(V_i, V_j)\neq a_H(i, j)$.

\begin{lemma}\label{lemma-Vi-disjoint-pairs-wrong-adj}Suppose $i_1,\dots,i_m$ and  $j_1,\dots,j_m$ are distinct elements of $V(H)\sm S$ such that for each $\l=1,\dots,m$ we have $a_\Gamma(V_{i_\l}, V_{j_\l})\neq a_H(i_\l, j_\l)$. Then $m \leq \frac{20}{q} (\log k)^2$.
\end{lemma}
\begin{proof}Consider a random map $\theta: V(H)\to V(\Gamma)$ chosen as follows: Let $\theta(s)=\psi(s)$ for all $s\in S$ and choose $\theta(i)\in V_i$ uniformly at random, independently for all $i\in V(H)\sm S$.

We want to bound the probability that $\theta$ forms an embedding $H\into \Gamma$. For each $\l=1,\dots,m$, we have $a_\Gamma(V_{i_\l}, V_{j_\l})\neq a_H(i_\l, j_\l)$. Hence, by the definition of $a_\Gamma(V_{i_\l}, V_{j_\l})$, the probability for $a_\Gamma(\theta(i_\l),\theta(j_\l))=a_H(i_\l, j_\l)$ is at most $\frac{1}{2}$. As $i_1,\dots,i_m, j_1,\dots,j_m$ are distinct, these events are independent for all $\l=1,\dots,m$. Thus, the probability that $a_\Gamma(\theta(i_\l),\theta(j_\l))=a_H(i_\l, j_\l)$ for all $\l=1,\dots,m$ is at most $2^{-m}$. Therefore the probability that $\theta$ forms an embedding $H\into \Gamma$ is at most $2^{-m}$.

Hence there are at most
\[2^{-m}\prod_{i\in V(H)\sm S} \vert V_i\vert\]
embeddings $\theta: H\into \Gamma$ with $\theta(s)=\psi(s)$ for all $s\in S$ and $\theta(i)\in V_i$ for all $i\in V(H)\sm S$. Combining this with Lemma \ref{lemma-embeddings-Vi}, we obtain
\[\prod_{i\in V(H)\sm S} \vert V_i\vert\geq 2^{m}\cdot \frac{\E_k(n)}{n^{\vert S\vert}}.\]
On the other hand, $\sum_{i\in V(H)\sm S}\vert V_i\vert\leq n$ and therefore by Lemma \ref{lemma-functions-E}(i) and Lemma \ref{lemma-functions-E}(iii) (recall that $n\geq \tk\geq k$) we have
\[\prod_{i\in V(H)\sm S} \vert V_i\vert\leq \E_{k-\vert S\vert}(n)\leq e^{3\vert S\vert}\left(\frac{k}{n}\right)^{\vert S\vert}\E_k(n)=e^{(3+\log k)\vert S\vert}\frac{\E_k(n)}{n^{\vert S\vert}}.\]
Combining the previous two inequalities, we obtain
\[2^{m}\leq e^{(3+\log k)\vert S\vert}.\]
Hence
\[m\leq \frac{1}{\log 2}\cdot (3+\log k)\vert S\vert\leq \frac{1}{\log 2}\cdot 2(\log k)\cdot \frac{5}{q}\log k=\frac{10}{(\log 2)\cdot q}\cdot (\log k)^2\leq \frac{20}{q}\cdot (\log k)^2,\]
as desired.
\end{proof}

The lemmas above give useful properties of the sets $V_i$ that we will use later. From now on, we can disregard the map $\psi:S\to V(\Gamma)$. Instead of studying the embeddings $H\into \Gamma$ extending $\psi$, we will now consider all embeddings $H\into \Gamma$.

\begin{lemma}\label{lemma-typ1}Let $x, x'\in V(H)$ be distinct vertices and let $z, z'\in V(\Gamma)$. Then there are at most $k^{-7}\E_k(n)/n^2$ embeddings $\theta: H\into \Gamma$ with $\theta(x)=z$ and $\theta(x')=z'$ as well as $\theta(y)\in V_{\textnormal{rest}}$ for at least $\eps_3 k/3$ vertices $y\in V(H)\sm \lbrace x,x'\rbrace$.
\end{lemma}
\begin{proof}Set $\beta=\frac{1}{3}\eps_3$ and $\gamma=\frac{24}{q}(\log k)^2/k$. Then by (\ref{ineq-eps3-k}) and (\ref{ineq-eps1-q}) we have $0<\gamma<\beta\leq \frac{q}{3}$ and also $\beta k>\frac{20}{q}(\log k)^2$. Furthermore observe, using that $9e^6\leq 3^8= 81^2\leq 10^4$ and using (\ref{ineq-eps3-k})
\[\frac{e^6\gamma}{\beta^2}=\frac{1}{\eps_3^2}\cdot 9e^6\cdot \frac{24}{q}\cdot \frac{(\log k)^2}{k}\leq \frac{1}{4\eps_3^2}\cdot 10^4\cdot \frac{100}{q}\cdot \frac{(\log k)^2}{k}=\frac{1}{4\eps_3^2}\cdot \frac{10^6}{q}\cdot \frac{(\log k)^2}{k}\leq \frac{1}{4}.\]
By (\ref{ineq-eps5-logk-k}), this implies
\[\left(\frac{e^6\gamma}{\beta^2}\right)^{\beta k}\leq \left(\frac{1}{4}\right)^{\eps_3 k/3}<e^{-\eps_3 k/3}\leq e^{-7\log k}=k^{-7}.\]
By Lemma \ref{lemma-V-small}, we have
\[\vert V_{\textnormal{rest}}\vert\leq \frac{24}{q}\cdot \frac{(\log k)^2}{k}n=\gamma n.\]
Let $H'=H\sm \lbrace x,x'\rbrace$. Then by Corollary \ref{coro-embedding-small-set}, the number of embeddings $\theta: H'\into \Gamma$ with $\theta(y)\in V_{\textnormal{rest}}$ for at least $\eps_3 k/3=\beta k$ vertices $y\in V(H')$ is at most
\[ \left(\frac{e^6\gamma}{\beta^2}\right)^{\beta k}\cdot \frac{\E_k(n)}{n^{2}}\leq k^{-7}\cdot \frac{\E_k(n)}{n^{2}}.\]
In particular, there can also be at most $k^{-7}\E_k(n)/n^2$ embeddings $\theta: H\into \Gamma$ with $\theta(x)=z$ and $\theta(x')=z'$ as well as $\theta(y)\in V_{\textnormal{rest}}$ for at least $\eps_3 k/3$ vertices $y\in V(H)\sm \lbrace x,x'\rbrace$.
\end{proof}

\begin{lemma}\label{lemma-typ2}Let $x, x'\in V(H)$ be distinct vertices and let $z, z'\in V(\Gamma)$. Then there are at most $k^{-7}\E_k(n)/n^2$ embeddings $\theta: H\into \Gamma$ with $\theta(x)=z$ and $\theta(x')=z'$ and such that the image of $\theta$ contains two vertices from the same set $V_i$ for some $i\in V(H)\sm S$.
\end{lemma}
\begin{proof}We claim that for any distinct vertices $x_1, x_2\in V(H)$ and for any distinct vertices $z_1,z_2\in V(\Gamma)$ that lie in the same set $V_i$ for some $i\in V(H)\sm S$, there are at most
\[k^{-9}\cdot \frac{\E_k(n)}{n^{\vert \lbrace x_1,x_2,x,x'\rbrace\vert}}\]
embeddings $\theta: H\into \Gamma$ with $\theta(x)=z$, $\theta(x')=z'$, $\theta(x_1)=z_1$ and $\theta(x_2)=z_2$.

First, let us check that this claim implies the desired statement. For any embedding $\theta: H\into \Gamma$ with the properties in the lemma, there exist distinct vertices $x_1, x_2\in V(H)$ such that $\theta(x_1)$ and $\theta(x_2)$ lie in the same set $V_i$ for some $i\in V(H)\sm S$. There are at most $k^2$ possibilities for $x_1, x_2\in V(H)$.

If $x_1$ and $x_2$ are distinct from $x$ and $x'$, then $\vert \lbrace x_1,x_2,x,x'\rbrace\vert=4$ and there are at most $n^2$ possibilities for $z_1=\theta(x_1)$ and $z_2=\theta(x_2)$ lying in the same set $V_i$ for some $i\in V(H)\sm S$. By the claim above, each of these possibilities yields at most $k^{-9}\E_k(n)/n^{4}$ options for $\theta:H\into \Gamma$. So for fixed $x_1, x_2\in V(H)$ distinct from $x$ and $x'$, there are at most $k^{-9}\E_k(n)/n^{2}$ options for $\theta$.

If $x_1=x$ but $x_2\neq x'$, then $\vert \lbrace x_1,x_2,x,x'\rbrace\vert=3$ and there are at most $n$ possibilities for $z_2=\theta(x_2)$ lying in the same set $V_i$ as $z_1=\theta(x_1)=\theta(x)=z$. By the claim above, each of these possibilities yields at most $k^{-9}\E_k(n)/n^{3}$ options for $\theta:H\into \Gamma$. So for fixed $x_1, x_2\in V(H)$ with $x_1=x$ and $x_2\neq x'$, there are at most $k^{-9}\E_k(n)/n^{2}$ options for $\theta$. The case of $x_1=x'$ and $x_2\neq x$, the case of $x_2=x$ and $x_1\neq x'$, and the case of $x_2=x'$ and $x_1\neq x$ are analogous.

If $x_1=x$ and $x_2=x'$, then $\vert \lbrace x_1,x_2,x,x'\rbrace\vert=2$ and there is at most one possibility for $z_1=\theta(x_1)$ and $z_2=\theta(x_2)$, since $z_1=\theta(x_1)=\theta(x)=z$ and $z_2=\theta(x_2)=\theta(x')=z'$. By the claim above, there are at most $k^{-9}\E_k(n)/n^{2}$ options for $\theta:H\into \Gamma$ if $x_1=x$ and $x_2=x'$. The case of $x_1=x'$ and $x_2=x$ is analogous.

So we have seen that for any distinct vertices $x_1, x_2\in V(H)$, there are at most $k^{-9}\E_k(n)/n^2$ embeddings $\theta: H\into \Gamma$ with $\theta(x)=z$ and $\theta(x')=z'$ and such $\theta(x_1)$ and $\theta(x_2)$ lie in the same set $V_i$ for some $i\in V(H)\sm S$. As there are at most $k^2$ possibilities for $x_1, x_2\in V(H)$, this means that there are at most $k^{-7}\E_k(n)/n^2$ embeddings $\theta: H\into \Gamma$ with the properties in the lemma.

It remains to prove the claim above. So let us fix distinct vertices $x_1, x_2\in V(H)$ and distinct vertices $z_1,z_2\in V(\Gamma)$ such that $z_1, z_2\in V_i$ for some $i\in V(H)\sm S$. Let $m=\vert \lbrace x_1,x_2,x,x'\rbrace\vert$ and note that $2\leq m\leq 4$. Define $x_t\in \lbrace x,x'\rbrace$ for $3\leq t\leq m$ such that $x_1,\dots, x_m$ are distinct and $\lbrace x_1,\dots,x_m\rbrace=\lbrace x_1,x_2,x,x'\rbrace$. In other words, for $m=4$ let $x_3$ and $x_4$ be $x$ and $x'$, for $m=3$ let $x_3\in \lbrace x,x'\rbrace$ be distinct from $x_1$ and $x_2$, and for $m=2$ the elements $x_1,x_2$ already satisfy $\lbrace x_1,x_2\rbrace=\lbrace x_1,x_2,x,x'\rbrace$. For $\l=3,\dots,m$ let $z_\l=z$ if $x_\l=x$ and $z_\l=z'$ if $x_\l=x'$. With this notation, we need to prove that there are at most $k^{-9}\E_k(n)/n^m$ embeddings $\theta: H\into \Gamma$ with $\theta(x_\l)=z_\l$ for $\l=1,\dots,m$.

Let $U\su V(\Gamma)$ be the set of vertices $w\in V(\Gamma)\sm \lbrace z_1,z_2\rbrace$ such that $a_\Gamma(w,z_1)\neq a_\Gamma(w,z_2)$, and let $Y\su V(H)\sm  \lbrace x_1,\dots,x_m\rbrace$ be the set of vertices $y\in V(H)\sm \lbrace x_1,\dots,x_m\rbrace$ such that $a_H(y,x_1)\neq a_H(y,x_2)$. Finally, let $H'=H\sm \lbrace x_1,\dots,x_m\rbrace$, so $\vert V(H')\vert = k-m\geq k-4$.

We claim that for each embedding $\theta: H\into \Gamma$ with $\theta(x_\l)=z_\l$ for $\l=1,\dots,m$, we must have $\theta(y)\in U$ for all $y\in Y$. Indeed, each $y\in Y$ must satisfy $a_\Gamma(\theta(y),z_1)=a_\Gamma(\theta(y),\theta(x_1))=a_H(y,x_1)$ and $a_\Gamma(\theta(y),z_2)=a_\Gamma(\theta(y),\theta(x_2))=a_H(y,x_2)$. As $a_H(y,x_1)\neq a_H(y,x_2)$, we can conclude $a_\Gamma(\theta(y),z_1)\neq a_\Gamma(\theta(y),z_2)$, so $\theta(y)\in U$ (note that $\theta(y)\neq z_1$ since $y\neq x_1$ and $\theta$ is injective, analogously $\theta(y)\neq z_2$).

Hence each embedding $\theta: H\into \Gamma$ with $\theta(x_\l)=z_\l$ for $\l=1,\dots,m$ restricts to an embedding $\theta: H'\into \Gamma$ with $\theta(y)\in U$ for all $y\in Y$. Using the conditions $\theta(x_\l)=z_\l$ for $\l=1,\dots,m$, the original embedding $\theta$ can be reconstructed from its restriction to $H'$. Hence it suffices to prove that there are at most $k^{-9}\E_k(n)/n^m$ embeddings $\theta: H'\into \Gamma$ with $\theta(y)\in U$ for all $y\in Y$.

Let $\gamma=3\eps_5$ and $\beta=\frac{q}{3}$. Note that then by (\ref{ineq-eps5-q})
\begin{equation}\label{ineq-proof-lemma-typ2}
\frac{\gamma}{\beta}=9\cdot \frac{\eps_5}{q}\leq 9\cdot 10^{-4}< 10^{-3}< 3^{-5}< e^{-5}.
\end{equation}
In particular, $0<\gamma<\beta=\frac{q}{3}$. Furthermore, $\beta k=\frac{q}{3}k\geq \frac{20}{q}(\log k)^2$ by (\ref{eq-ineq-q-k1}).

By condition (b) in Definition \ref{defi-reasonable}, there are at least $qk$ vertices $y\in V(H)\sm \lbrace x_1,x_2\rbrace$ with $a_H(y,x_1)\neq a_H(y,x_2)$. Hence there are at least $qk-(m-2)$ such vertices $y\in V(H)\sm \lbrace x_1,\dots,x_m\rbrace$. Thus, using $qk\geq 3$ by (\ref{ineq-q-lower-bound}), we obtain
\[\vert Y\vert\geq qk-(m-2)\geq qk-2\geq \frac{q}{3}k=\beta k.\]

Let us now prove $\vert U\vert\leq \gamma n$. Every vertex $w\in U$ is contained in $V_{\textnormal{rest}}$, in $V_{i}$ or in $V_j$ for some $j\in V(H)\sm S$ with $j\neq i$. If $w\in V_j$ for some for some $j\in V(H)\sm S$ with $j\neq i$, then by $a_\Gamma(z_1, w)\neq a_\Gamma(z_2, w)$ we must have $a_\Gamma(z_1, w)\neq a_H(i,j)$ or $a_\Gamma(z_2, w)\neq a_H(i,j)$. As $z_1,z_2\in V_i\su V_i'$, this means that $(z_1,w)$ is a bad pair or $(z_2,w)$ is a bad pair (see Definition \ref{defi-bad}). Since $z_1\in V_i$, the vertex $z_1$ is not bad and therefore it is part of at most $\eps_5 n$ bad pairs $(z_1,w)$. In particular, there can only be at most $\eps_5 n$ vertices $w\in U$ such that $(z_1,w)$ is a bad pair. Similarly, there can be at most $\eps_5 n$ vertices $w\in U$ such that $(z_2,w)$ is a bad pair. Hence there can be at most $2\eps_5 n$ vertices $w\in U$ with $w\in V_j$ for some for some $j\in V(H)\sm S$ with $j\neq i$. Thus,
\[\vert U\vert\leq 2\eps_5 n+\vert V_{i}\vert +\vert V_{\textnormal{rest}}\vert.\]
By Lemma \ref{lemma-V-small} and (\ref{ineq-eps5-logk-k}), we have
\[\vert V_{i}\vert +\vert V_{\textnormal{rest}}\vert\leq \frac{24}{q}\cdot \frac{(\log k)^2}{k}n<\eps_5 n.\]
Thus, $\vert U\vert\leq 3\eps_5 n=\gamma n$ as desired.

Hence by Lemma \ref{lemma-embedding-small-set}, the number of embeddings $\theta: H'\into \Gamma$ with $\theta(y)\in U$ for all $y\in Y$ is at most
\[\left(\frac{e^4\gamma}{\beta}\right)^{\beta k}\cdot \frac{\E_k(n)}{n^{m}} \leq \left(\frac{1}{e}\right)^{qk/3}\cdot \frac{\E_k(n)}{n^{m}}=e^{-qk/3}\cdot \frac{\E_k(n)}{n^{m}}\leq e^{-9\log k}\cdot \frac{\E_k(n)}{n^{m}}=k^{-9}\cdot \frac{\E_k(n)}{n^{m}},\]
where in the first inequality we used (\ref{ineq-proof-lemma-typ2}) and in the second inequality we used (\ref{eq-ineq-q-k1}). This finishes the proof of Lemma \ref{lemma-typ2}.
\end{proof}

\section{Proof of Proposition \ref{propo-few-unloyal}}\label{sect-few-unloyal}

In this section, we will finally prove Proposition \ref{propo-few-unloyal}. At the end of the section, we will also deduce Corollary \ref{coro-few-unloyal} from Proposition \ref{propo-few-unloyal}.

First, let us make several definitions.

\begin{definition}\label{defi-aligned}Let $A\su V(H)$ and let $f: A\to V(H)\sm S$ be an injective function. Then an embedding $\theta: H\into \Gamma$ is called \emph{$f$-aligned} if $\theta(v)\in V_{f(a)}$ for all $v\in A$ and the image of $\theta$ contains at most one vertex from $V_i$ for each $i\in V(H)\sm S$.
\end{definition}

\begin{definition}Let $A\su V(H)$ and let $f: A\to V(H)\sm S$ be an injective function. A pair $\lbrace v,v'\rbrace$ of distinct vertices in $A$ is called \emph{$f$-consistent} if $a_H(f(v),f(v'))=a_H(v,v')$, and it is called \emph{$f$-inconsistent} otherwise. The function $f: A\to V(H)\sm S$ is called \emph{locally mostly consistent} if each vertex $v\in A$ is part of at most $\eps_4 k$ different $f$-inconsistent pairs $\lbrace v,v'\rbrace\su A$.
\end{definition}

Recall that we defined the concept of an $r$-super-signature in Definition \ref{defi-super-signature}.

\begin{definition}Let $B\su V(H)$ be such that $\vert B\vert\geq (1-\eps_3)k$ and let $f: B\to V(H)\sm S$ be an injective function. Then $f$ is called \emph{fancy} if there exists a $\frac{q}{4}$-super-signature $T\su B$ of $H$ of size $\vert T\vert \leq \frac{33}{q}\log k$ such that for each vertex $v\in B\sm T$ the number of $f$-inconsistent pairs $\lbrace v,t\rbrace$ with $t\in T$ is at most $\frac{q}{10}\cdot \vert T\vert$.
\end{definition}

Also recall that we defined the notion of a loyal embedding in Definition \ref{defi-loyal}. The importance of the definitions above for our proof of Proposition \ref{propo-few-unloyal} is due to the following lemma.

\begin{lemma}\label{lemma-five-properties}Every embedding $\theta: H\into \Gamma$ has at least one of the following five properties:
\begin{itemize}
\item[(I)] $\theta(v)\in V_{\textnormal{rest}}$ for at least $\eps_3 k/2$ vertices $v\in V(H)$.
\item[(II)] The image of $\theta$ contains two vertices from the same set $V_i$ for some $i\in V(H)\sm S$.
\item[(III)] There  is some subset $A\su V(H)$ such that $\theta$ is $f$-aligned for some injective function $f: A\to V(H)\sm S$ such that $f$ is not locally mostly consistent.
\item[(IV)] There is some subset $B\su V(H)$ of size $\vert B\vert\geq (1-\eps_3)k$ such that $\theta$ is $f$-aligned for some fancy injective function $f: B\to V(H)\sm S$ such that there are at least $\eps_2 k/4$ mutually disjoint $f$-inconsistent pairs in $B$.
\item[(V)] $\theta$ is loyal.
\end{itemize}
\end{lemma}

According to this lemma, every disloyal embedding $\theta: H\into \Gamma$ needs to satisfy one of the properties (I) to (IV) above. In order to prove Proposition \ref{propo-few-unloyal}, we will show that for each of those four properties there are only few embeddings $\theta: H\into \Gamma$ with that property and the conditions in Proposition \ref{propo-few-unloyal}.

Before we can prove Lemma \ref{lemma-five-properties}, we need another lemma.

\begin{lemma}\label{lemma-locally-mostly-consistent-fancy}Let $A\su V(H)$ be of size $\vert A\vert\geq (1-\frac{\eps_3}{2})k$ and let $f:A\to V(H)\sm S$ be an injective function that is locally mostly consistent. Then there exists a subset $B\su A$ of size $\vert B\vert\geq (1-\eps_3)k$, such that the restriction $f\vert_B: B\to V(H)\sm S$ is fancy.
\end{lemma}
\begin{proof}As $\vert A\vert \geq (1-\frac{q}{2})k$ by (\ref{ineq-eps1-q}), by Lemma \ref{lemma-small-super-signature} there exists a $\frac{q}{4}$-super-signature $T$ of $H$ of size $\vert T\vert \leq \frac{33}{q}\log k$ with $T\su A$. Since $f$ is locally mostly consistent, every vertex $t\in T$ is contained in at most $\eps_4 k$ different $f$-inconsistent pairs $\lbrace v,t\rbrace$ with $v\in A\sm T$. In particular, there are at most $\eps_4 k\cdot \vert T\vert$ different $f$-inconsistent pairs $\lbrace v,t\rbrace$ with $v\in A\sm T$ and $t\in T$. Hence, using (\ref{ineq-eps4-eps3}) there can be at most
\[\frac{\eps_4 k\cdot \vert T\vert}{\frac{q}{10}\cdot \vert T\vert}=\frac{10}{q}\eps_4 k\leq \frac{10}{q}\cdot \frac{q}{20} \eps_3\cdot k=\frac{\eps_3}{2}k\]
vertices $v\in A\sm T$ that are contained in at least $\frac{q}{10}\cdot \vert T\vert$ different $f$-inconsistent pairs $\lbrace v,t\rbrace$ with $t\in T$. Let $B$ be obtained from $A$ by deleting all the vertices $v\in A\sm T$ with this property. Then $B\su A$ has size $\vert B\vert \geq \vert A\vert -\frac{\eps_3}{2}k\geq (1-\eps_3)k$ and satisfies $T\su B$. Furthermore, for each vertex $v\in B\sm T$ the number of $f$-inconsistent pairs $\lbrace v,t\rbrace$ with $t\in T$ is at most $\frac{q}{10}\cdot \vert T\vert$. Thus, $f\vert_B: B\to V(H)\sm S$ is fancy.
\end{proof}

Now we are ready for the proof of Lemma \ref{lemma-five-properties}.

\begin{proof}[Proof of Lemma \ref{lemma-five-properties}]Suppose $\theta: H\into \Gamma$ is an embedding that does not satisfy any of the five properties (I) to (V).

Let $A\su V(H)$ be the set of vertices $v\in V(H)$ with $\theta(v)\not\in V_{\textnormal{rest}}$. As $\theta$ does not satisfy property (I), we have $\vert A\vert\geq (1-\frac{\eps_3}{2})k$. For each $v\in A$, there is a unique index $i\in V(H)\sm S$ with $\theta(v)\in V_i$. For each $v\in A$, set $f_1(v)=i$ for this index $i\in V(H)\sm S$ with $\theta(v)\in V_i$. Then we obtain a function $f_1: A\to V(H)\sm S$ such that $\theta(v)\in V_{f_1(v)}$ for all $v\in A$.

Since $\theta$ does not satisfy (II), the image of $\theta$ contains at most one vertex from $V_i$ for each $i\in V(H)\sm S$.  Thus, the function $f_1: A\to V(H)\sm S$ is injective. We can furthermore conclude that $\theta$ is $f_1$-aligned.

Note that $f_1$ must be locally mostly consistent, since otherwise $\theta$ would have property (III). Thus, we can apply Lemma \ref{lemma-locally-mostly-consistent-fancy} and obtain a subset $B\su A$ of size $\vert B\vert\geq (1-\eps_3)k$, such that the restriction of $f_1$ to $B$ is fancy. Let $f_2=f_1\vert_B$ this restriction, then $f_2:B\to V(H)\sm S$ is injective and fancy and furthermore $\theta$ is $f_2$-aligned.

As $\theta$ does not satisfy (IV), there cannot be at least $\eps_2 k/4$ mutually disjoint $f_2$-inconsistent pairs in $B$. Let $\lbrace v_1,v_1'\rbrace, \dots, \lbrace v_\l,v_\l\rbrace$ be a maximal collection of mutually disjoint $f_2$-inconsistent pairs in $B$. Then $\l\leq \eps_2 k/4$. Let $Y$ be obtained from $B$ by deleting $v_1,\dots,v_\l$ and $v_1',\dots,v_\l'$. Then, using (\ref{ineq-eps3-eps2}), we obtain
\[\vert Y\vert=\vert B\vert-2\l\geq (1-\eps_3)k-2\cdot \frac{\eps_2}{4}k\geq \left(1-\frac{\eps_2}{2}\right)k-2\cdot \frac{\eps_2}{4}k=(1-\eps_2)k.\]
Furthermore, there are no $f_2$-inconsistent pairs in $Y$, because $\lbrace v_1,v_1'\rbrace, \dots, \lbrace v_\l,v_\l\rbrace$ was a maximal collection of mutually disjoint $f_2$-inconsistent pairs in $B$. This means that for all distinct vertices $v,v'\in Y$ we have $a_H(f_2(v),f_2(v'))=a_H(v,v')$.

Let $f_3$ be the restriction of $f_2$ to $Y$ (which is also the restriction of $f_1$ to $Y$). Then $f_3: Y\to V(H)\sm S$ is injective and $a_H(f_3(v),f_3(v'))=a_H(v,v')$ for all distinct vertices $v,v'\in Y$. Furthermore $\vert Y\vert\geq (1-\eps_2)k\geq (1-\delta)k$ by (\ref{ineq-eps2-delta}). So by condition (c) in Definition \ref{defi-reasonable}, there exists a rotation or reflection $\phi$ of $\tH$ such that $f_3=\phi\vert_Y$. So for each $y\in Y\su A$ we have $f_1(y)=f_2(y)=f_3(y)=\phi(y)$. Thus, $\phi(y)\in V(H)\sm S$ and $\theta(y)\in V_{f_1(y)}=V_{\phi(y)}$ for all $y\in Y$. Recall that $\vert Y\vert\geq (1-\eps_2)k$ and that the image of $\theta$ contains at most one vertex from $V_i$ for each $i\in V(H)\sm S$. Hence $\theta$ is $\phi$-loyal. But then $\theta$ satisfies (V), contradiction.\end{proof}

Recall that our goal is to prove Proposition \ref{propo-few-unloyal}. So let $x, x'\in V(H)$ be distinct vertices and let $z, z'\in V(\Gamma)$. We need to show that there are at most $k^{-6}\cdot \E_k(n)/n^{2}$ disloyal embeddings $\theta:H\into\Gamma$ with $\theta(x)=z$ and $\theta(x')=z'$. By Lemma \ref{lemma-five-properties}, each disloyal embedding $H\into \Gamma$ has at least one of the properties (I) to (IV). We will prove that for each of the properties (I) to (IV), there are at most $k^{-7}\cdot \E_k(n)/n^{2}$  embeddings $\theta:H\into\Gamma$ with that property and with $\theta(x)=z$ and $\theta(x')=z'$. Then the total number of disloyal embeddings $\theta:H\into\Gamma$ with $\theta(x)=z$ and $\theta(x')=z'$ is at most
\[4\cdot k^{-7}\cdot \frac{\E_k(n)}{n^2}\leq k^{-6}\cdot \frac{\E_k(n)}{n^2}\]
as desired. 

For property (I), the desired bound is given by the following claim.

\begin{claim}There are at most $k^{-7}\E_k(n)/n^2$ embeddings $\theta: H\into \Gamma$ with $\theta(x)=z$ and $\theta(x')=z'$ as well as $\theta(v)\in V_{\textnormal{rest}}$ for at least $\eps_3 k/2$ vertices $v\in V(H)$.
\end{claim}
\begin{proof}For each such embedding $\theta$, there are at least
\[\frac{\eps_3}{2}k-2\geq \frac{\eps_3}{3}k\]
vertices $v\in V(H)\sm \lbrace x,x'\rbrace$ with $\theta(v)\in V_{\textnormal{rest}}$ (for the inequality we used (\ref{ineq-eps3-k})). So by Lemma \ref{lemma-typ1}, there can by at most $k^{-7}\E_k(n)/n^2$ such embeddings $\theta: H\into \Gamma$.
\end{proof}

For property (II), the desired bound already follows from Lemma \ref{lemma-typ2}. For property (III), the desired bound is given by the following lemma.

\begin{lemma}There are at most $k^{-7}\E_k(n)/n^2$ embeddings $\theta: H\into \Gamma$ with $\theta(x)=z$ and $\theta(x')=z'$ and such that $\theta$ has property (III).
\end{lemma}
\begin{proof}In order for $\theta$ to have property (III), $\theta$ must be $f$-aligned for some injective function $f: A\to V(H)\sm S$ with $A\su V(H)$ such that $f$ is not locally mostly consistent. In particular, there must exist some vertex $y\in A$ that is part of at least $\eps_4 k$ different $f$-inconsistent pairs $\lbrace y,v\rbrace$ for $v\in A\sm \lbrace y\rbrace$. Let
\[A'=\lbrace v\in A\sm \lbrace y\rbrace\mid a_H(f(y),f(v))\neq a_H(y,v)\rbrace.\]
In other words, $A'$ consists of those $v\in A\sm \lbrace y\rbrace$ such that $\lbrace y,v\rbrace$ is an $f$-inconsistent pair. Then $\vert A'\vert\geq \eps_4 k$. Note that for each $v\in A'$ we have $a_\Gamma(\theta(y), \theta(v))=a_H(y,v)$, hence
$a_\Gamma(\theta(y), \theta(v))\neq a_H(f(y),f(v))$. On the other hand, $\theta(y)\in V_{f(y)}\su V_{f(y)}'$ and $\theta(v)\in V_{f(v)}\su V_{f(v)}'$, since $\theta$ is $f$-aligned. Furthermore, $f(y)\neq f(v)$ as $f$ is injective. Thus, $(\theta(y), \theta(v))$ is a bad pair for each $v\in A'\su V(H)\sm \lbrace y\rbrace$ (see Definition \ref{defi-bad}). In particular, there are at least $\vert A'\vert-2\geq \eps_4 k-2\geq \eps_4 k/2$ vertices $v\in V(H)\sm \lbrace x,x',y\rbrace$ such that $(\theta(y), \theta(v))$ is a bad pair (the inequality $2\leq \eps_4 k/2$ follows from (\ref{ineq-eps5-logk-k})). Note that $\theta(y)\in V_i$ for $i=f(y)\in V(H)\sm S$. So we have seen that for each embedding $\theta: H\into \Gamma$ with the properties in the lemma, there exists a vertex $y\in V(H)$ with $\theta(y)\in V_i$ for some $i\in V(H)\sm S$ and such that $(\theta(y), \theta(v))$ is a bad pair for at least $\eps_4 k/2$ vertices $v\in V(H)\sm \lbrace x,x',y\rbrace$.

There are $k$ possibilities for the vertex $y\in V(H)$. We will prove that for each $y\in V(H)$, there are at most $k^{-8}\E_k(n)/n^2$ embeddings $\theta: H\into \Gamma$ with $\theta(x)=z$ and $\theta(x')=z'$ as well as $\theta(y)\in V_i$ for some $i\in V(H)\sm S$ and such that $(\theta(y), \theta(v))$ is a bad pair for at least $\eps_4 k/2$ vertices $v\in V(H)\sm \lbrace x,x',y\rbrace$. Summing this for all $k$ possibilities for $y\in V(H)$ yields that there can be at most $k^{-7}\E_k(n)/n^2$ embeddings $\theta: H\into \Gamma$ with the properties in the lemma.

So let us fix some $y\in V(H)$. If $y\not\in \lbrace x,x'\rbrace$, there are at most $n$ possibilities for $\theta(y)\in V(\Gamma)$ with $\theta(y)\in V_i$ for some $i\in V(H)\sm S$. If $y\in \lbrace x, x'\rbrace$, then $\theta(y)$ is already determined by the conditions $\theta(x)=z$ and $\theta(x')=z'$, so there is at most one possibility for $\theta(y)$. So it suffices to show that for every fixed $y\in V(H)$ and every fixed $\theta(y)\in V(\Gamma)$ with $\theta(y)\in V_i$ for some $i\in V(H)\sm S$, there are at most $k^{-8}\E_k(n)/n^{\vert \lbrace x,x',y\rbrace\vert}$ embeddings $\theta: H\into \Gamma$ with $\theta(x)=z$ and $\theta(x')=z'$ and such that $(\theta(y), \theta(v))$ is a bad pair for at least $\eps_4 k/2$ vertices $v\in V(H)\sm \lbrace x,x',y\rbrace$.

So let us fix $y\in V(H)$ and $\theta(y)\in V(\Gamma)$ with $\theta(y)\in V_i$ for some $i\in V(H)\sm S$. Set $H'=H\sm \lbrace x,x',y\rbrace$, then $\vert V(H')\vert\geq k-3$. Any embedding $\theta$ with $\theta(x)=z$ and $\theta(x')=z'$ and the fixed value of $\theta(y)$ can be reconstructed from its restriction to $H'$. Let $U$ be the set of vertices $u\in V(\Gamma)$ such that $(\theta(y), u)$ is a bad pair. Since $\theta(y)\in V_i$ for some $i\in V(H)\sm S$, by the definition of $V_i$ the vertex $\theta(y)$ is not bad and therefore $\vert U\vert\leq \eps_5 n$. Every embedding $\theta: H\into \Gamma$ such that $(\theta(y), \theta(v))$ is a bad pair for at least $\eps_4 k/2$ vertices $v\in V(H)\sm \lbrace x,x',y\rbrace$ restricts to an embedding $\theta\vert_{H'}:H'\into \Gamma$ with $\theta(v)\in U$ for at least $\eps_4 k/2$ vertices $v\in V(H')$. Note that $0<\eps_5<\frac{1}{2}\eps_4<\frac{q}{3}$ by (\ref{ineq-eps5-eps4}) and (\ref{ineq-eps1-q}) and $\frac{1}{2}\eps_4 k>\frac{20}{q}(\log k)^2$ by (\ref{ineq-eps5-logk-k}). Furthermore, note that by (\ref{ineq-eps5-eps4}) we have
\[\frac{\eps_5}{\eps_4^2}\leq \frac{1}{10^5}\leq \frac{1}{10}\cdot \frac{1}{81^2}\leq \frac{1}{4}3^{-8}\leq \frac{1}{4}e^{-8}.\]
So by Corollary \ref{coro-embedding-small-set} the number of embeddings $\theta: H'\into \Gamma$ with $\theta(v)\in U$ for at least $\eps_4 k/2$ vertices $v\in V(H')$ is at most
\[ \left(\frac{e^6\eps_5}{(\eps_4/2)^2}\right)^{\eps_4 k/2}\cdot \frac{\E_k(n)}{n^{k-\vert V(H')\vert}}\leq \left(e^{-2}\right)^{\eps_4 k/2}\cdot \frac{\E_k(n)}{n^{k-\vert V(H')\vert}}=e^{-\eps_4 k}\cdot \frac{\E_k(n)}{n^{\vert \lbrace x,x',y\rbrace\vert}}\leq k^{-8}\cdot \frac{\E_k(n)}{n^{\vert \lbrace x,x',y\rbrace\vert}},\]
where we used $\eps_4 k\geq 8\log k$ by (\ref{ineq-eps5-logk-k}). Thus, for each fixed $y\in V(H)$ and fixed $\theta(y)\in V(\Gamma)$ with $\theta(y)\in V_i$ for some $i\in V(H)\sm S$, there are at most $k^{-8}\E_k(n)/n^{\vert \lbrace x,x',y\rbrace\vert}$ embeddings $\theta: H\into \Gamma$ with $\theta(x)=z$ and $\theta(x')=z'$ and such that $(\theta(y), \theta(v))$ is a bad pair for at least $\eps_4 k/2$ vertices $v\in V(H)\sm \lbrace x,x',y\rbrace$.
\end{proof}

For property (IV), the desired bound will follow from the following two lemmas.

\begin{lemma}\label{lemma-typ4-f}There are at most $e^{4\log(1/\eps_3)\eps_3 k}$ fancy injective functions $f: B\to V(H)\sm S$ with $B\su V(H)$ of size $\vert B\vert\geq (1-\eps_3)k$.
\end{lemma}
\begin{proof}
First of all, as $\vert B\vert\geq (1-\eps_3)k$ and $1<\eps_3 k<k/3$ by (\ref{ineq-eps1-q}) and (\ref{ineq-eps5-logk-k}), there are at most
\[\sum_{\l=\lceil(1-\eps_3)k\rceil}^{k}\binom{k}{ \l}=\sum_{\l=0}^{\lfloor \eps_3 k\rfloor}\binom{k}{\l}
\leq (\eps_3 k+1)\binom{k}{ \lceil \eps_3 k\rceil}\leq 2\eps_3k\left(\frac{ek}{\lceil \eps_3 k\rceil}\right)^{\lceil \eps_3 k\rceil}
\leq  k\left(\frac{ek}{\eps_3 k}\right)^{\eps_3 k+1}
\leq k\left(\frac{e}{\eps_3}\right)^{2\eps_3 k}\]
possibilities for $B\su V(H)$.

Now let us fix $B\su V(H)$ of size $\vert B\vert\geq (1-\eps_3)k$. For every fancy injective function $f: B\to V(H)\sm S$, there exists a $\frac{q}{4}$-super-signature $T\su B$ of $H$ of size $\vert T\vert \leq \frac{33}{q}\log k$ such that for each vertex $v\in B\sm T$ the number of $f$-inconsistent pairs $\lbrace v,t\rbrace$ with $t\in T$ is at most $\frac{q}{10}\cdot \vert T\vert$.

Note that there are at most
\[\sum_{\l=1}^{\lfloor \frac{33}{q}\log k\rfloor}\binom{\vert B\vert}{ \l}\leq \sum_{\l=1}^{\lfloor \frac{33}{q}\log k\rfloor}\binom{k}{\l}\leq k^{\lfloor (33/q)\log k\rfloor}\leq k^{(33/q)\log k}=e^{(33/q)(\log k)^2}\]
possibilities for $T$.

Let us now fix $B$ and $T\su B$ and let us bound the number of possibilities for $f: B\to V(H)\sm S$. Clearly, there are at most $k^{\vert T\vert}$ possibilities for $f\vert_T$. So let us now fix $f\vert_T$. For each vertex $v\in B\sm T$, let $U_v\su V(H)\sm S$ be the set of those vertices $u\in V(H)\sm S$ for which there exists at least one extension $f$ (with the properties above) of the chosen map $f\vert_T$ such that $f(v)=u$. Then the number of extensions $f: B\to V(H)\sm S$ with the desired properties is at most $\prod_{v\in B\sm T}\vert U_v\vert$. Since $f$ needs to be injective, all the sets $U_v$ for $v\in B\sm T$ are disjoint from the image of $f\vert_T$.

We claim that the sets $U_v$ for $v\in B\sm T$ are mutually disjoint. Suppose there is a vertex $u\in U_v\cap U_{v'}$ for distinct $v,v'\in B\sm T$. Then there are extensions $f$ and $f'$ (with the properties above) of $f\vert_T$ satisfying $f(v)=u$ and $f'(v')=u$. The number of $f$-inconsistent pairs $\lbrace v,t\rbrace$ with $t\in T$ is at most $\frac{q}{10}\cdot \vert T\vert$ and the number of $f'$-inconsistent pairs $\lbrace v',t\rbrace$ with $t\in T$ is also at most $\frac{q}{10}\cdot \vert T\vert$. Hence for at least $(1-\frac{2}{10}q)\cdot \vert T\vert$ vertices $t\in T$ we have both $a_H(f(v),f(t))=a_H(v,t)$ and $a_H(f'(v'),f'(t))=a_H(v',t)$. But $f(v)=f'(v)=u$ and $f(t)=f'(t)=f\vert_T (t)$ (since both $f$ and $f'$ are extensions of $f\vert_T$). Hence we obtain
\[a_H(v,t)=a_H(f(v),f(t))=a_H(u,f\vert_T(t))=a_H(f'(v'),f'(t))=a_H(v',t)\]
for at least $(1-\frac{2}{10}q)\cdot \vert T\vert=(1-\frac{q}{5})\cdot \vert T\vert$ vertices $t\in T$. Thus, $\vert (N(v)\Delta N(v'))\cap T\vert\leq \frac{q}{5}\cdot \vert T\vert$. As $\vert T\vert>0$, this implies $\vert (N(v)\Delta N(v'))\cap T\vert< \frac{q}{4}\cdot \vert T\vert$, which is a contradiction to $T$ being a $\frac{q}{4}$-super-signature (see Definition \ref{defi-super-signature}).

Hence the sets $U_v$ for $v\in B\sm T$ are indeed mutually disjoint. All of them are subsets of $V(H)$ that are disjoint from the image of $f\vert_T$ (note that this image has size $\vert T\vert$, because $f$ must be injective). Thus, $\sum_{v\in B\sm T}\vert U_v\vert\leq k-\vert T\vert$. So, using Lemma \ref{lemma-functions-E}(i), the number of extensions of $f\vert_T$ to a map $f: B\to V(H)\sm S$ with the desired properties is at most
\[\prod_{v\in B\sm T}\vert U_v\vert\leq \E_{\vert B\vert-\vert T\vert}(k-\vert T\vert).\]

We claim that $\E_{\vert B\vert-\vert T\vert}(k-\vert T\vert)\leq 2^{\eps_3 k}$. Indeed, note that $\vert B\vert-\vert T\vert\geq (1-\eps_3) k-\frac{33}{q}\log k\geq (1-\frac{1}{4}) k-\frac{1}{4}k=k/2$ by (\ref{ineq-eps1-q}) and (\ref{eq-ineq-q-k1}), and therefore
\[1\leq \frac{k-\vert T\vert}{\vert B\vert-\vert T\vert}=1+\frac{k-\vert B\vert}{\vert B\vert-\vert T\vert}\leq 1+\frac{\eps_3 k}{k/2}=1+2\eps_3<2.\]
So the integers $m_1,\dots,m_{\vert B\vert-\vert T\vert}$ in the definition of $\E_{\vert B\vert-\vert T\vert}(k-\vert T\vert)$ (see Definition \ref{defi-functions-E}) satisfy
\[2\geq \left\lceil \frac{k-\vert T\vert}{\vert B\vert-\vert T\vert}\right\rceil\geq m_1\geq \dots\geq m_{\vert B\vert-\vert T\vert}\geq \left\lfloor \frac{k-\vert T\vert}{\vert B\vert-\vert T\vert}\right\rfloor=1.\]
As
\[m_1+\dots +m_{\vert B\vert-\vert T\vert}=k-\vert T\vert=(\vert B\vert-\vert T\vert)+(k-\vert B\vert),\]
there are precisely $k-\vert B\vert$ twos among $m_1,\dots,m_{\vert B\vert-\vert T\vert}$, and the remaining variables are one. Thus,
\[\E_{\vert B\vert-\vert T\vert}(k-\vert T\vert)=m_1\dotsm m_{\vert B\vert-\vert T\vert}=2^{k-\vert B\vert}\leq 2^{\eps_3 k}.\]

So we have seen that for fixed $B$ and fixed $T\su B$, there are at most $k^{\vert T\vert}$ possibilities to choose $f\vert_T$ and each of these choices can be extended to at most $\E_{\vert B\vert-\vert T\vert}(k-\vert T\vert)\leq 2^{\eps_3 k}$ maps $f: B\to V(H)\sm S$ with the desired properties. Hence, given $B$ and $T\su B$, the number of possibilities for $f$ is at most
\[k^{\vert T\vert}\cdot 2^{\eps_3 k}\leq k^{(33/q)\log k}\cdot 2^{\eps_3 k}\leq e^{(33/q)(\log k)^2}\cdot 2^{\eps_3 k}.\]

Recall that there are at most $k(e/\eps_3)^{2\eps_3 k}$ choices for $B$ and for each of them, there are at most $e^{(33/q)(\log k)^2}$ choices for $T$. Hence the total number of injective fancy functions as in the lemma is at most
\[k\left(\frac{e}{\eps_3}\right)^{2\eps_3 k}\cdot e^{(33/q)(\log k)^2}\cdot e^{(33/q)(\log k)^2}\cdot 2^{\eps_3 k}\leq \left(\frac{2e}{\eps_3}\right)^{2\eps_3 k}e^{(67/q)(\log k)^2}. \]
As $(67/q)(\log k)^2\leq \eps_3 k$ by (\ref{ineq-eps3-k}), this number is at most
\[\left(\frac{2e}{\eps_3}\right)^{2\eps_3 k}e^{\eps_3 k}\leq \left(\frac{2e^2}{\eps_3}\right)^{2\eps_3 k}\leq \left(\frac{1}{\eps_3^2}\right)^{2\eps_3 k}=\left(\frac{1}{\eps_3}\right)^{4\eps_3 k}=e^{4\log(1/\eps_3)\eps_3 k},\]
where for the second inequality we used that $\eps_3\leq 10^{-20}\leq 1/(2e^2)$ by (\ref{ineq-eps1-q}).
\end{proof}

\begin{lemma}\label{lemma-typ4-theta}Let $B\su V(H)$ be of size $\vert B\vert\geq (1-\eps_3)k$ and let $f: B\to V(H)\sm S$ be an injective function such that there are at least $\eps_2 k/4$ mutually disjoint $f$-inconsistent pairs in $B$. Then there are at most $e^{-\eps_2 k/20}\E_k(n)/n^2$ different $f$-aligned embeddings $\theta:H\into \Gamma$ with $\theta(x)=z$ and $\theta(x')=z'$.
\end{lemma}

\begin{proof}Let $B'=B\sm \lbrace x,x'\rbrace$. Then there are still at least $\eps_2 k/4-2\geq (3/16)\eps_2 k$ mutually disjoint $f$-inconsistent pairs in $B'$ (note that $2\leq \eps_2 k/16$ by (\ref{ineq-eps3-k})). So let $\lbrace v_1,w_1\rbrace,\dots, \lbrace v_m,w_m\rbrace$ be mutually disjoint $f$-inconsistent pairs in $B'$ with $m\geq (3/16)\eps_2 k$. Then $v_1,\dots,v_m$ and $w_1,\dots,w_m$ are distinct elements of $B'$ and for each $\l=1,\dots,m$ we have $a_H(f(v_\l),f(w_\l))\neq a_H(v_\l, w_\l)$. As $f$ is injective, $f(v_1),\dots,f(v_m)$ and $f(w_1),\dots,f(w_m)$ are distinct elements of $V(H)\sm S$. So by Lemma \ref{lemma-Vi-disjoint-pairs-wrong-adj}, there can be at most $(20/q)(\log k)^2$ different indices $1\leq \l\leq m$ with $a_\Gamma(V_{f(v_\l)}, V_{f(w_\l)})\neq a_H(f(v_\l), f(w_\l))$. Note that $(20/q)(\log k)^2\leq \eps_2 k/16$ by (\ref{ineq-eps3-k}). Thus, there are at least
\[m-\frac{\eps_2 k}{16}\geq \frac{3}{16}\eps_2 k-\frac{\eps_2 k}{16}=\frac{\eps_2 k}{8}\]
indices $1\leq \l\leq m$ with $a_\Gamma(V_{f(v_\l)}, V_{f(w_\l)})= a_H(f(v_\l), f(w_\l))$. Without loss of generality, let us assume that $a_\Gamma(V_{f(v_\l)}, V_{f(w_\l)})= a_H(f(v_\l), f(w_\l))$ for $\l=1,\dots,\lceil \eps_2 k/8\rceil$. Then for $\l=1,\dots,\lceil \eps_2 k/8\rceil$ we have $a_\Gamma(V_{f(v_\l)}, V_{f(w_\l)})\neq a_H(v_\l, w_\l)$.

We need to bound the number of $f$-aligned embeddings $\theta:H\into \Gamma$. First, let us consider the number of possibilities for $\theta\vert_{B'\cup \lbrace x,x'\rbrace}$. We need to choose $\theta(v)\in V_{f(v)}$ for all $v\in B'$. Furthermore, $\theta(x)=z$ and $\theta(x')=z'$ are already determined.

Suppose we choose $\theta(v)\in V_{f(v)}$ uniformly at random, independently for all $v\in B'$. Recall that we have $a_\Gamma(V_{f(v_\l)}, V_{f(w_\l)})\neq a_H(v_\l, w_\l)$ for $\l=1,\dots,\lceil \eps_2 k/8\rceil$. Hence for each $\l=1,\dots,\lceil \eps_2 k/8\rceil$, the probability for $a_\Gamma(\theta(v_\l), \theta(w_\l))=a_H(v_\l, w_\l)$ is at most $\frac{1}{2}$. As $f(v_1),\dots,f(v_m), f(w_1),\dots,f(w_m)$ are distinct, these events are independent for all $\l=1,\dots,\lceil \eps_2 k/8\rceil$. Thus, the probability that $a_\Gamma(\theta(v_i), \theta(w_i))=a_H(v_\l, w_\l)$ for all $\l=1,\dots,\lceil \eps_2 k/8\rceil$ is at most $2^{-\lceil \eps_2 k/8\rceil}\leq 2^{-\eps_2 k/8}$. In particular, the probability that $\theta\vert_{B'\cup \lbrace x,x'\rbrace}$ forms an embedding is at most $2^{-\eps_2 k/8}$.

Hence, by Lemma \ref{lemma-functions-E}(i), the number of possibilities for $\theta\vert_{B'\cup \lbrace x,x'\rbrace}$ is at most
\[2^{-\eps_2 k/8}\cdot \prod_{v\in B'}\vert V_{f(v)}\vert\leq 2^{-\eps_2 k/8}\cdot\E_{\vert B'\vert}(n-\vert U\vert),\]
where $U=V(\Gamma)\sm \bigcup_{v\in B'}V_{f(v)}$ denotes the complement of the union of the sets  $V_{f(v)}$ for $v\in B'$ (note that these sets are all disjoint).

We claim that for every vertex $w\in V(H)\sm (B'\cup \lbrace x,x'\rbrace)$ we must have $\theta(w)\in U$. Suppose that $\theta(w)\in V_{f(v)}$ for some $v\in B'$. Note that we also have $\theta(v)\in V_{f(v)}$ and $v\neq w$. Hence the image of $\theta$ would contain two vertices from the set $V_{f(v)}$ (as $\theta$ must be injective), but then $\theta$ cannot be $f$-aligned (see Definition \ref{defi-aligned}). Hence we must indeed have $\theta(w)\in U$ for all $w\in V(H)\sm (B'\cup \lbrace x,x'\rbrace)$.

Note that by Lemma \ref{lemma-large-signature} the set $B'\cup \lbrace x,x'\rbrace$ is a signature, since we have $\vert B'\cup \lbrace x,x'\rbrace\vert\geq \vert B\vert\geq (1-\eps_3)k\geq (1-q)k$. Thus, for every fixed $\theta\vert_{B'\cup \lbrace x,x'\rbrace}$, by Lemma \ref{lemma-signature-extensions} there are at most $\E_{k-\vert B'\vert-2}(\vert U\vert)$ possibilities to extend $\theta\vert_{B'\cup \lbrace x,x'\rbrace}$ to an embedding $\theta:H\into \Gamma$ with the desired properties.

Thus, all in all the number of $f$-aligned embeddings $\theta:H\into \Gamma$ with $\theta(x)=z$ and $\theta(x')=z'$ is at most
\[2^{-\eps_2 k/8}\cdot\E_{\vert B'\vert}(n-\vert U\vert)\cdot \E_{k-\vert B'\vert-2}(\vert U\vert)\leq 2^{-\eps_2 k/8}\cdot \E_{k-2}(n)\leq 2^{-\eps_2 k/8}e^6k^2\frac{\E_k(n)}{n^2},\]
where we used Lemma \ref{lemma-functions-E}(ii) and Lemma \ref{lemma-functions-E}(iii), recalling that $n\geq \tk\geq k$. Note that by $k\geq \tk/2\geq 10^{199}$ and (\ref{ineq-eps3-k}), we have
\[2^{-\eps_2 k/8}e^6k^2\leq 2^{-\eps_2 k/8}k^3\leq e^{-\eps_2 k/16}e^{3\log k}<e^{-\eps_2 k/16}e^{\eps_2 k/80}=e^{-\eps_2 k/20}.\]
Hence there are at most $e^{-\eps_2 k/20}\E_k(n)/n^2$ different $f$-aligned embeddings $\theta:H\into \Gamma$ with $\theta(x)=z$ and $\theta(x')=z'$.
\end{proof}

Now  the desired bound for property (IV) is given by the following claim.

\begin{claim}There are at most $k^{-7}\E_k(n)/n^2$ embeddings $\theta: H\into \Gamma$ with $\theta(x)=z$ and $\theta(x')=z'$ and such that $\theta$ has property (IV).
\end{claim}
\begin{proof}
In order for $\theta$ to have property (IV), $\theta$ must be $f$-aligned for some fancy injective function $f: B\to V(H)\sm S$ with $B\su V(H)$ of size $\vert B\vert\geq (1-\eps_3)k$ such that there are at least $\eps_2 k/4$ mutually disjoint $f$-inconsistent pairs in $B$. By Lemma \ref{lemma-typ4-f}, there are at most $e^{4\log(1/\eps_3)\eps_3 k}$ possibilities for the function $f$. By Lemma \ref{lemma-typ4-theta}, for each fixed $f$ there are at most $e^{-\eps_2 k/20}\E_k(n)/n^2$ $f$-aligned embeddings $\theta:H\into \Gamma$ with $\theta(x)=z$ and $\theta(x')=z'$. So all in all the number of possibilities for $\theta$ with the properties in the claim is at most
\[e^{4\log(1/\eps_3)\eps_3 k}\cdot e^{-\eps_2 k/20}\frac{\E_k(n)}{n^2}=e^{4\log(1/\eps_3)\eps_3 k}\cdot e^{-\eps_2 k/25}e^{-\eps_2 k/100}\frac{\E_k(n)}{n^2}\leq k^{-7}\frac{\E_k(n)}{n^2},\]
where we used $\eps_2/25>4\log(1/\eps_3)\eps_3$ by (\ref{ineq-eps3-eps2}) and $\eps_2 k/100>7\log k$ by (\ref{ineq-eps3-k}).
\end{proof}

This finishes the proof of Proposition \ref{propo-few-unloyal}. Finally, let us deduce Corollary \ref{coro-few-unloyal}.

\begin{proof}[Proof of Corollary \ref{coro-few-unloyal}]
For the first statement, fix some $x'\in V(H)$ with $x'\neq x$ (recall $\vert V(H)\vert = k \geq \tk/2\geq 10^{199}$). Then by Proposition \ref{propo-few-unloyal}, for each $z'\in V(\Gamma)$ there are at most $k^{-6}\cdot \E_k(n)/n^{2}$ disloyal embeddings $\theta:H\into\Gamma$ with $\theta(x)=z$ and $\theta(x')=z'$. Adding this up for all $z'\in V(\Gamma)$ shows that there are at most $k^{-6}\cdot \E_k(n)/n$ disloyal embeddings $\theta:H\into\Gamma$ with $\theta(x)=z$.

For the second statement, fix  some vertex $x\in V(H)$. For each $z\in V(\Gamma)$ there are at most $k^{-6}\cdot \E_k(n)/n$ disloyal embeddings $\theta:H\into\Gamma$ with $\theta(x)=z$. Adding this up for all $z\in V(\Gamma)$ shows that the total number of disloyal embeddings $H\into\Gamma$ is at most $k^{-6}\cdot \E_k(n)$.
\end{proof}

\section{Proof of Proposition \ref{propo-optimization}}\label{sect-optimization}

\subsection{Some preparations}

\begin{lemma}\label{lemma-optimization0} For any positive integer $m$, we have $\emb(H,m)\leq k^{-k+2}m^k$.
\end{lemma}
\begin{proof}
Recall that our arguments in Section \ref{sect-proof-blow-up} are valid for all $n\geq \tk$. Hence, from Claim \ref{claim-embHn-small} we obtain $\emb(H,m)\leq 3k^{-k+1}m^k\leq k^{-k+2}m^k$ for all $m\geq \tk$.

If $m<k$, we clearly have $\emb(H,m)=0\leq k^{-k+2}m^k$. So it only remains to consider the case that $k\leq m<\tk$. But then, using Claim \ref{claim-embHn-small} for $n=\tk$,
\[\emb(H,m)\leq \emb(H,\tk)\leq 3k^{-k+1}\tk^k\leq 3k^{-k+1}\left(k+\frac{1}{2}\log k\right)^k\leq 3k^{-k+1}e^{(\log k)/2}k^k\leq k^{-k+2}m^k,\]
as $k\geq \tk-\frac{1}{4}\log \tk\geq \tk-\frac{1}{2}\log k$ and $k\geq \tk/2\geq 10^{199}$.
\end{proof}

The two inequalities in the following lemma were already observed by Pippenger and Golumbic, see the proofs of Propositions 1 and 2 in \cite{pippenger-golumbic}. We repeat the proof here for the reader's convenience.

\begin{lemma}\label{lemma-optimization1} For any integer $m\geq 2$, we have
\[\frac{k}{m-1}\cdot \emb(H,m-1)\leq \emb(H,m)-\emb(H,m-1)\leq \frac{k}{m}\cdot \emb(H,m).\]
\end{lemma}
\begin{proof}
Let $\Gamma_m$ be a graph on $m$ vertices with $\emb(H,\Gamma_m)=\emb(H,m)$. On average, each vertex $v\in V(\Gamma_m)$ appears in $\frac{k}{m}\emb(H,\Gamma_m)$ embeddings $H\into \Gamma_m$. Thus, there exists a vertex $v\in V(\Gamma_m)$ that appears in at most $\frac{k}{m}\emb(H,\Gamma_m)$ embeddings $H\into \Gamma_m$. Now, let us delete this vertex $v$ from $\Gamma_m$ and observe that
\[\emb(H,m-1)\geq \emb(H,\Gamma_m\sm\lbrace v\rbrace)\geq \emb(H,\Gamma_m)-\frac{k}{m}\emb(H,\Gamma_m)=\left(1-\frac{k}{m}\right)\emb(H,m).\]
This proves the second inequality.

For the first inequality, let $\Gamma_{m-1}$ be a graph on $m-1$ vertices with $\emb(H,\Gamma_{m-1})=\emb(H,m-1)$. On average, each vertex $v\in V(\Gamma_{m-1})$ appears in $\frac{k}{m-1}\emb(H,\Gamma_{m-1})$ embeddings $H\into \Gamma_{m-1}$. Thus, there exists a vertex $v\in V(\Gamma_{m-1})$ that appears in at least $\frac{k}{m-1}\emb(H,\Gamma_{m-1})$ embeddings $H\into \Gamma_{m-1}$. Now, let the graph $\Gamma_{m}'$ be obtained from $\Gamma_{m-1}$ by making an additional copy of the vertex $v$ (say, unconnected to the original vertex $v$). Then $\Gamma_{m}'$ has $m$ vertices and
\[\emb(H,m)\geq \emb(H,\Gamma_{m}')\geq \emb(H,\Gamma_{m-1})+\frac{k}{m-1}\emb(H,\Gamma_{m-1})
=\left(1+\frac{k}{m-1}\right)\emb(H,m-1),\]
which indeed proves the first inequality.
\end{proof}

\begin{lemma}\label{lemma-optimization2} For any integer $m\geq 3$, we have
\[\emb(H,m)-2\emb(H,m-1)+\emb(H,m-2)\leq 2k^{-k+4}\cdot m^{k-2}.\]
\end{lemma}
\begin{proof}
First,  note that applying the second inequality in Lemma \ref{lemma-optimization1} both for $m$ and for $m-1$, we obtain
\begin{multline}\label{eq-ineq-optimization-prep}
\emb(H,m)-\emb(H,m-2)=(\emb(H,m)-\emb(H,m-1))+(\emb(H,m-1)-\emb(H,m-2))\\
\leq \frac{k}{m}\cdot \emb(H,m)+\frac{k}{m-1}\cdot \emb(H,m-1)\leq \frac{2k}{m}\cdot \emb(H,m),
\end{multline}
where in the last step we used the comparison between the first and third term in Lemma \ref{lemma-optimization1}.

Also note that the first inequality in Lemma \ref{lemma-optimization1} applied to $m-1$ gives
\[\emb(H,m-1)-\emb(H,m-2)\geq \frac{k}{m-2}\cdot \emb(H,m-2)\geq \frac{k}{m}\cdot \emb(H,m-2).\]

Using this together with the second inequality in Lemma \ref{lemma-optimization1}, we obtain
\begin{multline*}
\emb(H,m)-2\emb(H,m-1)+\emb(H,m-2)\\
=(\emb(H,m)-\emb(H,m-1))-(\emb(H,m-1)-\emb(H,m-2))\\
\leq \frac{k}{m}\cdot \emb(H,m)-\frac{k}{m}\cdot \emb(H,m-2)\leq \frac{2k^2}{m^2}\cdot \emb(H,m)\leq 2k^{-k+4}m^{k-2},
\end{multline*}
where in the second-last step we used (\ref{eq-ineq-optimization-prep}) and in the last step Lemma \ref{lemma-optimization0}.
\end{proof}

\begin{corollary}\label{coro-optimization3} For positive integers $m'<m$ with $m-m'\geq 2$, we have
\[\emb(H,m)-\emb(H,m-1)-\emb(H,m'+1)+\emb(H,m')\leq (m-m')\cdot 2k^{-k+4}\cdot m^{k-2}.\]
\end{corollary}
\begin{proof} For all $\l=m'+2,\dots,m$ we have by Lemma \ref{lemma-optimization2}
\[(\emb(H,\l)-\emb(H,\l-1))-(\emb(H,\l-1)-\emb(H,\l-2))\leq 2k^{-k+4}\l^{k-2}\leq 2k^{-k+4}m^{k-2}.\]
Thus,
\begin{multline*}
\emb(H,m)-\emb(H,m-1)-\emb(H,m'+1)+\emb(H,m')\\
=(\emb(H,m)-\emb(H,m-1))-(\emb(H,m'+1)-\emb(H,m'))\\
=\sum_{\l=m'+2}^{m}(\emb(H,\l)-\emb(H,\l-1))-(\emb(H,\l-1)-\emb(H,\l-2))\\
\leq \sum_{\l=m'+2}^{m}2k^{-k+4}m^{k-2}=(m-m'-1)\cdot 2k^{-k+4}m^{k-2}< (m-m')\cdot 2k^{-k+4}\cdot m^{k-2},
\end{multline*}
as desired.
\end{proof}

\subsection{Proof}

Now, we finally prove Proposition \ref{propo-optimization}. Let $n_j$ for $j\in V(\tH)$ be non-negative integers with $\sum_{j\in V(\tH)} n_j=n$ satisfying the assumptions of Proposition \ref{propo-optimization}. In particular, this means that
\[T((n_j')_{j\in V(\tH)})\leq T((n_j)_{j\in V(\tH)})\]
for any non-negative integers $n_j'$ for $j\in V(\tH)$ with $\sum_{j\in V(\tH)} n_j'=n$. Also recall that we are assuming $n\geq \tk\geq k$ (this assumption was made at the beginning of Section \ref{sect-proof-blow-up}).

To simplify notation, set
\[N_\phi=N_\phi((n_j)_{j\in V(\tH)})=\prod_{j\in \phi(V(H))}n_j\]
for all rotations and reflections $\phi$ of $\tH$, and set
\[N=\sum_{\phi}N_\phi.\]
Then by the assumptions on $(n_j)_{j\in V(\tH)}$ we have $N\geq (\tk^{1/4}-1)\cdot \E_k(n)$. Furthermore, by Lemma \ref{lemma-functions-E}(i) we have $N_\phi\leq \E_k(n)$ for every rotation or reflection $\phi$.

Let $a\in V(\tH)$ and $b\in V(\tH)$ be chosen such that $n_a$ is maximal and $n_b$ is minimal among the $n_j$ for $j\in V(\tH)$. Then, for proving Proposition \ref{propo-optimization}, it suffices to show $n_a-n_b\leq 1$. So let us assume for contradiction that $n_a-n_b\geq 2$.

Note that the assumptions on $(n_j)_{j\in V(\tH)}$ imply that $1\leq n_b\leq n_a\leq  k^{-4/5}\cdot n$.

\begin{lemma}\label{lemma-optimization4} For every subset $I\su V(\tH)$ of size $\vert I\vert=k-2$ we have
\[\prod_{i\in I}n_i\leq 2e^6k^{7/4}\cdot \frac{N}{n^2}.\]
\end{lemma}
\begin{proof}
Note that $\sum_{i\in I}n_i\leq \sum_{j\in V(\tH)} n_j=n$. So from Lemma \ref{lemma-functions-E}(i) and (iii) we obtain
\[\prod_{i\in I}n_i\leq \E_{k-2}(n)\leq e^6\left(\frac{k}{n}\right)^2\E_k(n)\leq 2e^6k^{7/4}\cdot \frac{N}{n^2},\]
where in the last step we used that $N\geq (\tk^{1/4}-1)\cdot \E_k(n)\geq \frac{1}{2}k^{1/4}\cdot \E_k(n)$.
\end{proof}

\begin{lemma}\label{lemma-optimization5} Let $I,J\su V(\tH)$ be subsets with $\vert I\vert=\vert J\vert=k-1$. Then
\[\prod_{i\in I}n_i-\prod_{j\in J}n_j\leq (n_a-n_b)\cdot 2e^6k^{7/4}(\log \tk)\cdot \frac{N}{n^2}.\]
\end{lemma}
\begin{proof}
Note that $\vert V(\tH)\sm I\vert=\vert V(\tH)\sm J\vert=\tk-k+1\leq \frac{1}{2}\log \tk$ and therefore
\[\vert I\cap J\vert\geq \tk-2\cdot \frac{1}{2}\log \tk=\tk-\log \tk.\]
Let $m=k-1-\vert I\cap J\vert$, then $m\leq \log \tk$ and $m=\vert I\sm J\vert=\vert J\sm I\vert$. So let $I\sm J=\lbrace i_1,\dots,i_m\rbrace$ and $J\sm I=\lbrace j_1,\dots,j_m\rbrace$. Then $i_1,\dots,i_m,j_1,\dots,j_m$ are all distinct and none of them is contained in $I\cap J$. Now,
\[\prod_{i\in I}n_i-\prod_{j\in J}n_j=n_{i_1}\dotsm n_{i_m}\cdot \prod_{i\in I\cap J}n_i-n_{j_1}\dotsm n_{j_m}\cdot \prod_{j\in I\cap J}n_j
=\sum_{\l=1}^{m}n_{j_1}\dotsm n_{j_{\l-1}}\cdot (n_{i_\l}-n_{j_\l})\cdot n_{i_{\l+1}}\dotsm n_{i_m}\cdot \prod_{i\in I\cap J}n_i\]
Note that for each $\l=1,\dots,m$ we have $n_{i_\l}-n_{j_\l}\leq n_a-n_b$ and furthermore
\[0\leq n_{j_1}\dotsm n_{j_{\l-1}}\cdot n_{i_{\l+1}}\dotsm n_{i_m}\cdot \prod_{i\in I\cap J}n_i\leq  2e^6k^{7/4}\cdot \frac{N}{n^2}\]
by Lemma \ref{lemma-optimization4} applied to the set $\lbrace j_1,\dots,j_{\l-1},i_{\l+1},\dots,n_m\rbrace\cup (I\cap J)$, noting that $m-1+\vert I\cap J\vert=k-2$. Hence
\[\prod_{i\in I}n_i-\prod_{j\in J}n_j\leq \sum_{\l=1}^{m}(n_a-n_b)\cdot 2e^6k^{7/4}\cdot \frac{N}{n^2}=m\cdot (n_a-n_b)\cdot 2e^6k^{7/4}\cdot \frac{N}{n^2}\leq (n_a-n_b)\cdot 2e^6k^{7/4}(\log \tk)\cdot \frac{N}{n^2},\]
as $m\leq \log \tk$.
\end{proof}

Now let us consider $(n_j')_{j\in V(\tH)}$ given by $n_a'=n_a-1$ and $n_b'=n_b+1$ as well as $n_j'=n_j$ for all $j\in V(\tH)\sm \lbrace a,b\rbrace$. Then all $n_j'$ are non-negative integers and we have $\sum_{j\in V(\tH)}n_j'=\sum_{j\in V(\tH)}n_j=n$. To simplify notation, set
\[N_\phi'=N_\phi((n_j')_{j\in V(\tH)})\]
for all rotations and reflections $\phi$ of $\tH$. Recalling the assumptions on $(n_j)_{j\in V(\tH)}$, we have
\[\sum_{\phi}N_{\phi}+\sum_{j\in V(\tH)}\emb(H,n_j)=T((n_j)_{j\in V(\tH)})\geq T((n_j')_{j\in V(\tH)})=\sum_{\phi}N_{\phi}'+\sum_{j\in V(\tH)}\emb(H,n_j').\]
In other words,
\begin{equation}\label{eq-ineq-opt1}
\sum_{\phi}(N_{\phi}-N_{\phi}')+\sum_{j\in V(\tH)}(\emb(H,n_j)-\emb(H,n_j'))\geq 0.
\end{equation}
Recall that for each rotation and reflection $\phi$ of $\tH$,
\[N_\phi=N_\phi((n_j)_{j\in V(\tH)})=\prod_{j\in \phi(V(H))}n_j\]
and similarly for $n_\phi'$. Hence in the case $a,b\not\in \phi(V(H))$ we have $N_\phi=N_\phi'$. Now, let $\Phi_{a,b}$ denote the set of those rotations and reflections $\phi$ of $\tH$ with $a,b\in \phi(V(H))$. Furthermore, let $\Phi_a$ denote the set of those $\phi$ with $a\in \phi(V(H))$ but $b\not\in \phi(V(H))$, and similarly let $\Phi_b$ denote the set of those $\phi$ with $b\in \phi(V(H))$ but $a\not\in \phi(V(H))$. Then
\begin{equation}\label{eq-optimization-Ns}
\sum_{\phi}(N_{\phi}-N_{\phi}')=\sum_{\phi\in \Phi_{a,b}}(N_{\phi}-N_{\phi}')+\sum_{\phi\in \Phi_{a}}(N_{\phi}-N_{\phi}')+\sum_{\phi\in \Phi_{b}}(N_{\phi}-N_{\phi}').
\end{equation}
Note that the sets $\Phi_{a,b}$, $\Phi_a$ and $\Phi_b$ are disjoint. By Lemma \ref{lemma-number-rot-refl}, the number of rotations and reflections $\phi$ with $a\in \phi(V(H))$ equals the number of rotations and reflections $\phi$ with $b\in \phi(V(H))$. Therefore we have $\vert \Phi_{a,b}\vert+\vert \Phi_{a}\vert=\vert \Phi_{a,b}\vert+\vert \Phi_{b}\vert$, and hence $\vert \Phi_{a}\vert=\vert \Phi_{b}\vert$. Furthermore, by Lemma \ref{lemma-number-rot-refl} there are at most $2(\tk-k)\leq \frac{1}{2}\log \tk$ rotations and reflections $\phi$ with $b\not\in \phi(V(H))$. Thus, $\vert \Phi_{a}\vert\leq \frac{1}{2}\log \tk$ and all in all we obtain
\[\vert \Phi_{a}\vert=\vert \Phi_{b}\vert\leq \frac{1}{2}\log \tk.\]

Recall that we set $n_j'=n_j$ for all $j\in V(\tH)\sm \lbrace a,b\rbrace$. Hence $\emb(H,n_j)-\emb(H,n_j')=0$ for $j\in V(\tH)\sm \lbrace a,b\rbrace$ and we have
\begin{multline*}
\sum_{j\in V(\tH)}(\emb(H,n_j)-\emb(H,n_j'))=\emb(H,n_a)-\emb(H,n_a')-\emb(H,n_b')+\emb(H,n_b)\\
=\emb(H,n_a)-\emb(H,n_a-1)-\emb(H,n_b+1)+\emb(H,n_b)\leq (n_a-n_b)\cdot 2k^{-k+4}\cdot n_a^{k-2}.
\end{multline*}
where in the last step we used Corollary \ref{coro-optimization3}. Now recall that $n_a\leq k^{-4/5}\cdot n\leq \frac{1}{8}n$. Hence
\begin{multline}\label{eq-ineq-opt-emb}
\sum_{j\in V(\tH)}(\emb(H,n_j)-\emb(H,n_j'))\leq (n_a-n_b)\cdot 2k^{-k+4}\cdot n_a^{k-2}\leq (n_a-n_b)\cdot 2k^4k^{-k}\cdot \left(\frac{n}{8}\right)^{k-2}\\
=(n_a-n_b)\cdot 2k^4\cdot \left(\frac{n}{8k}\right)^{k}\cdot \frac{64}{n^2}=
(n_a-n_b)\cdot 128k^4\cdot \left(\frac{1}{4}\right)^{k}\cdot \left(\frac{n}{2k}\right)^{k}\cdot \frac{1}{n^2}\\
\leq (n_a-n_b)\cdot 128k^4\cdot \left(\frac{1}{4}\right)^{k}\cdot \left\lfloor \frac{n}{k}\right\rfloor^{k}\cdot \frac{1}{n^2}\leq (n_a-n_b)\cdot \left(\frac{1}{2}\right)^{k}\cdot \frac{N}{n^2},
\end{multline}
where in the last step we used $128k^4\leq 2^k$ and $N\geq (\tk^{1/4}-1)\cdot \E_k(n)\geq \E_k(n)\geq \left\lfloor \frac{n}{k}\right\rfloor^{k}$.

Plugging (\ref{eq-optimization-Ns}) and (\ref{eq-ineq-opt-emb}) into (\ref{eq-ineq-opt1}), we obtain
\[\sum_{\phi\in \Phi_{a,b}}(N_{\phi}-N_{\phi}')+\sum_{\phi\in \Phi_{a}}(N_{\phi}-N_{\phi}')+\sum_{\phi\in \Phi_{b}}(N_{\phi}-N_{\phi}')+(n_a-n_b)\cdot 2^{-k}\cdot \frac{N}{n^2}\geq 0.\]
Hence
\begin{equation}\label{eq-ineq-opt2}
\sum_{\phi\in \Phi_{a,b}}(N_{\phi}'-N_{\phi})\leq \sum_{\phi\in \Phi_{a}}(N_{\phi}-N_{\phi}')+\sum_{\phi\in \Phi_{b}}(N_{\phi}-N_{\phi}')+(n_a-n_b)\cdot 2^{-k}\cdot \frac{N}{n^2}.
\end{equation}

Note that for each $\phi\in \Phi_a$ we have
\[N_\phi-N_\phi'=\prod_{j\in \phi(V(H))}n_j-\prod_{j\in \phi(V(H))}n_j'=n_a\cdot \prod_{j\in \phi(V(H))\sm \lbrace a\rbrace}n_j-n_a'\cdot \prod_{j\in \phi(V(H))\sm \lbrace a\rbrace}n_j'.\]
However, as $b\not\in \phi(V(H))$, we have $n_j=n_j'$ for all $j\in \phi(V(H))\sm \lbrace a\rbrace$ and therefore, recalling $n_a-n_a'=1$,
\[N_\phi-N_\phi'=(n_a-n_a')\cdot \prod_{j\in \phi(V(H))\sm \lbrace a\rbrace}n_j=\prod_{j\in \phi(V(H))\sm \lbrace a\rbrace}n_j.\]
Similarly we can show that for each $\phi\in \Phi_b$ we have
\[N_\phi-N_\phi'=-\prod_{j\in \phi(V(H))\sm \lbrace b\rbrace}n_j.\]
Thus,
\[\sum_{\phi\in \Phi_{a}}(N_{\phi}-N_{\phi}')+\sum_{\phi\in \Phi_{b}}(N_{\phi}-N_{\phi}')=\sum_{\phi\in \Phi_{a}}\left(\prod_{i\in \phi(V(H))\sm \lbrace a\rbrace}n_i\right)-\sum_{\phi'\in \Phi_{b}}\left(\prod_{j\in \phi'(V(H))\sm \lbrace b\rbrace}n_j\right).\]
As $\vert \Phi_{a}\vert=\vert \Phi_{b}\vert\leq \frac{1}{2}\log \tk$, we can write this as the sum of at most $\frac{1}{2}\log \tk$ expressions of the form
\[\prod_{i\in \phi(V(H))\sm \lbrace a\rbrace}n_i-\prod_{j\in \phi'(V(H))\sm \lbrace a\rbrace}n_j.\]
By Lemma \ref{lemma-optimization5}, each of these expressions is at most $(n_a-n_b)\cdot 2e^6k^{7/4}(\log \tk)\cdot N/n^2$. Hence
\[\sum_{\phi\in \Phi_{a}}(N_{\phi}-N_{\phi}')+\sum_{\phi\in \Phi_{b}}(N_{\phi}-N_{\phi}')\leq \frac{1}{2}\log \tk\cdot (n_a-n_b)\cdot 2e^6k^{7/4}(\log \tk)\cdot \frac{N}{n^2}\leq (n_a-n_b)\cdot e^6k^{7/4}(\log \tk)^2\cdot \frac{N}{n^2}.\]
Plugging this into (\ref{eq-ineq-opt2}), we obtain
\begin{equation}\label{eq-ineq-opt3}
\sum_{\phi\in \Phi_{a,b}}(N_{\phi}'-N_{\phi})\leq (n_a-n_b)\cdot e^6k^{7/4}(\log \tk)^2\cdot \frac{N}{n^2}+(n_a-n_b)\cdot 2^{-k}\cdot \frac{N}{n^2}.
\end{equation}

It remains to find a lower bound for the sum on the left-hand side of (\ref{eq-ineq-opt3}). For each $\phi\in \Phi_{a,b}$ we have
\[N_\phi'-N_\phi=\prod_{j\in \phi(V(H))}n_j'-\prod_{j\in \phi(V(H))}n_j=n_a'\cdot n_b'\cdot \prod_{j\in \phi(V(H))\sm \lbrace a,b\rbrace}n_j'-n_a\cdot n_b\cdot \prod_{j\in \phi(V(H))\sm \lbrace a,b\rbrace}n_j.\]
Recall that $n_j=n_j'$ for all $j\in \phi(V(H))\sm \lbrace a,b\rbrace$ as well as $n_a'=n_a-1$ and $n_b'=n_b+1$. Thus,
\[N_\phi'-N_\phi=((n_a-1)\cdot (n_b+1)-n_a\cdot n_b)\cdot \prod_{j\in \phi(V(H))\sm \lbrace a,b\rbrace}n_j=(n_a-n_b-1)\cdot \prod_{j\in \phi(V(H))\sm \lbrace a,b\rbrace}n_j.\]
From our assumption $n_a-n_b\geq 2$ we obtain $n_a-n_b-1\geq \frac{1}{2}(n_a-n_b)$ and therefore
\[N_\phi'-N_\phi\geq \frac{1}{2}(n_a-n_b)\cdot \prod_{j\in \phi(V(H))\sm \lbrace a,b\rbrace}n_j=\frac{1}{2}(n_a-n_b)\cdot \frac{N_\phi}{n_a\cdot n_b}.\]
Hence
\begin{equation}\label{eq-opt-Phiab}
\sum_{\phi\in \Phi_{a,b}}(N_{\phi}'-N_{\phi})\geq \frac{1}{2}\cdot \frac{n_a-n_b}{n_a\cdot n_b}\cdot \sum_{\phi\in \Phi_{a,b}}N_{\phi}.
\end{equation}
On the other hand,
\[N=\sum_{\phi}N_\phi=\sum_{\phi\in \Phi_{a,b}}N_{\phi}+\sum_{\phi\not\in \Phi_{a,b}}N_{\phi}.\]
By Lemma \ref{lemma-number-rot-refl} there are at most $2(\tk-k)$ rotations and reflections $\phi$ with $a\not\in \phi(V(H))$ and at most $2(\tk-k)$ rotations and reflections $\phi$ with $b\not\in \phi(V(H))$. Thus, there are at most $4(\tk-k)\leq \log \tk$ rotations and reflections $\phi$ with $\phi\not\in \Phi_{a,b}$ and for each of them we have $N_\phi\leq \E_k(n)$. Thus, by (\ref{ineq-k-logtk}),
\[\sum_{\phi\not\in \Phi_{a,b}}N_{\phi}\leq (\log \tk)\cdot \E_k(n)\leq \frac{1}{2}(\tk^{1/4}-1)\cdot \E_k(n)\leq \frac{N}{2}\]
and consequently
\[\sum_{\phi\in \Phi_{a,b}}N_{\phi}=N-\sum_{\phi\not\in \Phi_{a,b}}N_{\phi}\geq \frac{N}{2}.\]
Thus, (\ref{eq-opt-Phiab}) yields
\[\sum_{\phi\in \Phi_{a,b}}(N_{\phi}'-N_{\phi})\geq \frac{1}{2}\cdot \frac{n_a-n_b}{n_a\cdot n_b}\cdot \frac{N}{2}=\frac{1}{4}\cdot (n_a-n_b)\cdot \frac{N}{n_a\cdot n_b}.\]
As $n_b$ is minimal among the $n_j$ with $j\in V(\tH)$ and $\sum_{j\in V(\tH)} n_j=n$, we have $n_b\leq n/\tk\leq n/k$. Furthermore recall that $n_a\leq k^{-4/5}n$. Hence
\[\sum_{\phi\in \Phi_{a,b}}(N_{\phi}'-N_{\phi})\geq \frac{1}{4}\cdot (n_a-n_b)\cdot \frac{N}{k^{-4/5}n \cdot k^{-1}n}=\frac{1}{4}\cdot (n_a-n_b)\cdot k^{9/5}\cdot \frac{N}{n^2}.\]
Plugging this into (\ref{eq-ineq-opt3}) yields
\[\frac{1}{4}\cdot (n_a-n_b)\cdot k^{9/5}\cdot \frac{N}{n^2}\leq (n_a-n_b)\cdot e^6k^{7/4}(\log \tk)^2\cdot \frac{N}{n^2}+(n_a-n_b)\cdot 2^{-k}\cdot \frac{N}{n^2}.\]
Recall that $n_a-n_b\geq 2>0$ and $N\geq (\tk^{1/4}-1)\cdot \E_k(n)>0$. Thus,
\[\frac{1}{4}\cdot k^{9/5}\leq e^6k^{7/4}(\log \tk)^2+2^{-k}\leq 10^3k^{7/4}(\log \tk)^2,\]
which means
\[k^{1/20}\leq 4\cdot 10^3(\log \tk)^2,\]
but this contradicts (\ref{ineq-k-logtk}). This proves Proposition \ref{propo-optimization}.

\section{Proof of Theorem  \ref{thm-tH-typcial-whp}}\label{sect-tH-typcial-whp}

The goal of this section is to prove Theorem \ref{thm-tH-typcial-whp}. In other words, we assume that $G$ is an abelian group of size $\tk$ and $0<p<1$ is such that
\begin{equation}\label{eq-bound-pstrich}
p'=\min(p,1-p)\geq 10^3(\log \tk)^{1/2}\cdot \tk^{-1/5},
\end{equation}
and we want to prove that for $q_0=p'/50$ and $\delta_0=p'/100$, the graph $\tH=\Cay(G,\Lambda)$ satisfies the conditions (i) to (iv) in Definition \ref{defi-typical} with probability $1-o(1)$ (where $\Lambda$ is chosen randomly according to Procedure \ref{procedure-lambda}). Recall that all $o(1)$-terms in this paper tend to zero as $\tk\to\infty$, independently of $p$.

Instead of condition (iv) in Definition \ref{defi-typical}, it will be easier to work with the following variant:
\begin{itemize}
\item[(iv')]For every subset $X\su V(\tH)$ of size $\vert X\vert\geq (1-\delta_0)\tk$ and every injective map $f:X\to V(\tH)$ with $a_{\tH}(v,w)=a_{\tH}(f(v),f(w))$ for all distinct $v,w\in X$, the following holds: Among the $\tk$ rotations and $\tk$ reflections of $\tH$, there exists some $\phi$ such that $f(v)=\phi(v)$ for at least $(1-2\delta_0)\tk$ vertices $v\in X$.
\end{itemize}
Note that the difference between (iv) and (iv') is the slightly weaker conclusion in (iv'): Whereas we demand $f(v)=\phi(v)$ for all vertices $v\in X$ in condition (iv), in condition (iv') we only demand $f(v)=\phi(v)$ for at least $(1-2\delta_0)\tk$ vertices $v\in X$.

\begin{claim}As long as $2\delta_0\leq q_0$, every Cayley graph $\tH=\Cay(G,\Lambda)$ satisfying (ii) and (iv'), also satisfies (iv).
\end{claim}
\begin{proof} Let $X\su V(\tH)$ be a subset of size $\vert X\vert\geq (1-\delta_0)\tk$ and $f:X\to V(\tH)$ an injective map with $a_{\tH}(v,w)=a_{\tH}(f(v),f(w))$ for all distinct $v,w\in X$. By (iv'), we can find a rotation or reflection $\phi$ of $\tH$ and a subset $U\su X$ of size $\vert U\vert \geq (1-2\delta_0)\tk\geq (1-q_0)\tk$ such that $f(v)=\phi(v)$ for all $v\in U$. Assume that there was some vertex $w\in X$ with $f(w)\neq \phi(w)$. Then $w\not\in U$ and therefore $\phi(w)\neq \phi(v)$ for all $v\in U$. Furthermore, for all $v\in U$ we have $f(v)=\phi(v)$ and consequently, by the injectivity of $f$ also $f(w)\neq \phi(v)$. For every $v\in U$, we now obtain
\[a_{\tH}(\phi(v),f(w))=a_{\tH}(f(v),f(w))=a_{\tH}(v,w)=a_{\tH}(\phi(v),\phi(w)).\]
So each of the vertices $\phi(v)$ for $v\in U$ is adjacent to either both of $f(w)$ and $\phi(w)$ or to neither of them. Hence there can be at most $\tk-2-\vert U\vert\leq q_0\tk-2$ vertices in $V(\tH)\sm \lbrace f(w),\phi(w)\rbrace$ that are adjacent to exactly one of the vertices $f(w)$ and $\phi(w)$. But this contradicts condition (ii). Hence we must have $f(w)= \phi(w)$ for all $w\in X$ and condition (iv) is satisfied.
\end{proof}

So it suffices to prove that for $q_0=p'/50$ and $\delta_0=p'/100$, the graph $\tH=\Cay(G,\Lambda)$ satisfies the conditions (i), (ii), (iii) and (iv') with probability $1-o(1)$. We will check this for each of these four conditions individually in the following four subsection.

Since we are only proving an asymptotic statement, we may assume that $\tk$ is very large. In particular, let us assume that $\log \tk\geq 2^{10}$ and $\tk^{2/5}\geq 9$. Then in particular,
\begin{equation}\label{eq-ineq-k25floor}
\lfloor \tk^{2/5}\rfloor\geq \tk^{2/5}-1\geq \frac{8}{9}\tk^{2/5}.
\end{equation}
Note that we also have the inequality
\begin{equation}\label{eq-p-ineq1}
p(1-p)\geq \frac{1}{2}\min(p,1-p)=\frac{1}{2}p'\geq \frac{10^3}{2}(\log \tk)^{1/2}\cdot \tk^{-1/5}
\end{equation}
and consequently
\begin{equation}\label{eq-p-ineq2}
p^2+(1-p)^2=1-2p(1-p)\leq 1-10^3(\log \tk)^{1/2}\cdot \tk^{-1/5}.
\end{equation}

\subsection{Condition (i)}

In order to check condition (i) note that it suffices to prove that with probability $1-o(1)$ we have $\deg(v)\geq q_0\tk$ for all vertices $v\in V(\tH)$. Indeed, if we replace $p$ by $1-p$, then the roles of edges and non-edges get interchanged (while the value of $q_0$ stays the same). So by applying the lower bound on the degrees for sampling probability $1-p$ in Procedure \ref{procedure-lambda}, we obtain that for the original problem with probability $1-o(1)$ we have $\deg(v)\leq (\tk-1)-q_0\tk\leq (1-q_0)\tk$ for all $v\in V(\tH)$, as desired.

Note that for each vertex $v\in V(\tH)$ we have $\deg v=\vert \Lambda\vert$. Hence the following lemma immediately implies that with probability $1-o(1)$ we have $\deg(v)\geq q_0\tk$ for all vertices $v\in V(\tH)$.

\begin{lemma}With probability $1-o(1)$ we have $\vert \Lambda\vert\geq q_0\tk$.
\end{lemma}
\begin{proof}Recall that in Procedure \ref{procedure-lambda}, we consider all subsets $\lbrace g,-g\rbrace\su G\sm \lbrace 0\rbrace$ independently. Let us choose a representative from each of these subsets. So let $g_1,\dots, g_m\in G\sm \lbrace 0\rbrace$ be such that all the sets $\lbrace g_i,-g_i\rbrace$ for $i=1,\dots,m$ are disjoint and
\[G\sm \lbrace 0\rbrace=\bigcup_{i=1}^{m} \lbrace g_i,-g_i\rbrace.\]
In particular $m\geq (\tk-1)/2\geq \tk/4$. We can assume without loss of generality that for some $0\leq \l\leq m$ we have $g_i\neq -g_i$ whenever $1\leq i\leq \l$ and $g_i= -g_i$ whenever $\l+1\leq i\leq m$. For $i=1,\dots ,m$, let $Z_i$ be the indicator random variable of the event $\lbrace g_i,-g_i\rbrace\su \Lambda$. These events are mutually independent and each of them has probability $p$. Remember that we either have  $\lbrace g_i,-g_i\rbrace\su \Lambda$ or  $\lbrace g_i,-g_i\rbrace\cap \Lambda=\emptyset$. Hence
\[\vert \Lambda\vert=\sum_{i=1}^{\l}2Z_i+\sum_{i=\l+1}^{m}Z_i\geq \sum_{i=1}^{m}Z_i\]
Thus, as
\[\frac{p}{10}\cdot m\geq \frac{p}{50}\cdot \tk\geq \frac{p'}{50}\cdot \tk=q_0\tk,\]
it suffices to prove that with probability $1-o(1)$ we have $Z_1+\dots+Z_m\geq pm/10$.

The Chernoff bound for lower tails of binomial random variables (see for example \cite[Theorem A.1.13]{alon-spencer}) gives
\[\P\left[Z_1+\dots+Z_m<\frac{pm}{10}\right]= \P\left[Z_1+\dots+Z_m<pm-\frac{9pm}{10}\right]< e^{-(9pm/10)^2/(2pm)}\leq e^{-pm/4}.\]
Note that
\[\frac{pm}{4}\geq \frac{1}{4}\cdot 10^3(\log \tk)^{1/2}\cdot \tk^{-2/5}\cdot \frac{\tk}{4}\geq \tk^{3/5}.\]
Thus,
\[\P\left[Z_1+\dots+Z_m<\frac{pm}{10}\right]\leq \exp(-pm/4)\leq \exp(-\tk^{3/5})=o(1),\]
as desired.
\end{proof}

\subsection{Condition (ii)}

We need to prove that with probability $1-o(1)$ we have that for any two distinct vertices $v,w\in V(\tH)$, there are at least $q_0\tk$ vertices in $V(\tH)\sm\lbrace v,w\rbrace$ that are adjacent to exactly one of the vertices $v$ and $w$. This means that for any two distinct elements $g,h\in G$ there are at least $q_0\tk$  elements $x\in G\sm \lbrace g,h\rbrace$ with $\vert \lbrace g-x, h-x\rbrace\cap \Lambda \vert =1$. We obtain the desired statement from the following lemma by summing over all choices for $g,h\in G$.

\begin{lemma}For any fixed distinct elements $g,h\in G$, with probability $1-o(\tk^{-2})$ there are at least $q_0\tk$  elements $x\in G\sm \lbrace g,h\rbrace$ with $\vert \lbrace g-x, h-x\rbrace\cap \Lambda \vert =1$.
\end{lemma}

\begin{proof}First, let us prove the following claim.

\begin{claim}\label{claim-solutions-xxgh}There are at most $\tk/2$ elements $x\in G$ with $x+x=g+h$.\end{claim}
\begin{proof}Let $X_0\su G$ be the subgroup formed by all elements $x\in G$ with $x+x=0$. Since $\vert X_0\vert$ divides $\vert G\vert=\tk$, we either have $X_0=G$ or $\vert X_0\vert\leq \tk/2$.

Let $X_1\su G$ consist of all those $x\in G$ with $x+x=g+h$. Then $X_1$ is either empty or a coset of $X_0$, so in particular $\vert X_1\vert\leq \vert X_0\vert$. If $\vert X_0\vert\leq \tk/2$ we immediately obtain $\vert X_1\vert\leq \tk/2$.

On the other hand, if $X_0=G$, then we have $x+x=0$ for all $x\in G$. Since $g\neq h$, this in particular implies $g+h\neq 0$. But then there cannot be any solutions for $x+x=g+h$, so $X_1$ must be empty. Thus, we have $\vert X_1\vert\leq \tk/2$ in any case.\end{proof}

Let $X\su G\sm \lbrace g,h\rbrace$ consist of all those $x\in G\sm \lbrace g,h\rbrace$ with $x+x\neq g+h$. By Claim \ref{claim-solutions-xxgh} we have
\[\vert X\vert\geq \tk-2-\frac{1}{2}\tk\geq \frac{3}{7}\tk.\]
Note that for each $x\in X$ we have $h-x\neq x-g$ and we clearly also have $h-x\neq g-x$. So for each $x\in X$ the two events $g-x\in \Lambda$ and $h-x\in \Lambda$ are independent and each of them happens with probability $p$. Thus, for each $x\in X$ the probability of $\vert \lbrace g-x, h-x\rbrace\cap \Lambda \vert =1$ is $2p(1-p)$.

Note that for each $x\in X$ there are at most three different $x'\in X\sm \lbrace x\rbrace$ with
\[g-x\in \lbrace g-x', x'-g, h-x', x'-h\rbrace.\]
Similarly, there are at most three different $x'\in X\sm \lbrace x\rbrace$ with
\[h-x\in \lbrace g-x', x'-g, h-x', x'-h\rbrace.\]
Hence there are at most six different $x'\in X\sm \lbrace x\rbrace$ with
\[\lbrace g-x, x-g, h-x, x-h\rbrace\cap \lbrace g-x', x'-g, h-x', x'-h\rbrace\neq \emptyset.\]
So by a simple greedy algorithm we can take a subset $Y\su X$ of size $\vert Y\vert\geq \vert X\vert/7\geq 3\tk/50$, such that for any distinct $x,x'\in Y$ we have 
\[\lbrace g-x, x-g, h-x, x-h\rbrace\cap \lbrace g-x', x'-g, h-x', x'-h\rbrace= \emptyset.\]
Then the events $\vert \lbrace g-x, h-x\rbrace\cap \Lambda \vert =1$ are independent for all $x\in Y$. Recall that each of these events happens with probability $2p(1-p)$. By the Chernoff bound for lower tails of binomial random variables (see for example \cite[Theorem A.1.13]{alon-spencer}) the probability that fewer than  $p(1-p)\vert Y\vert$ of the events $\vert \lbrace g-x, h-x\rbrace\cap \Lambda \vert =1$ for $x\in Y$ occur is at most
\begin{multline*}
\exp\left(-\frac{(p(1-p)\vert Y\vert)^2}{4p(1-p)\vert Y\vert}\right)= \exp\left(-\frac{1}{4}p(1-p)\vert Y\vert\right)\leq \exp\left(-\frac{1}{4}\cdot \frac{10^3}{2}(\log \tk)^{1/2}\tk^{-1/5}\cdot \frac{3\tk}{50}\right)\\
\leq \exp(-\tk^{4/5})=o(\tk^{-2}),
\end{multline*}
where we used (\ref{eq-p-ineq1}). Hence with probability $1-o(\tk^{-2})$ we have $\vert \lbrace g-x, h-x\rbrace\cap \Lambda \vert =1$ for at least $p(1-p)\vert Y\vert$ elements $x\in Y\su X\su G\sm \lbrace g,h\rbrace$. As
\[p(1-p)\vert Y\vert\geq \frac{p'}{2}\cdot \frac{3\tk}{50}\geq\frac{p'}{50}\cdot \tk=q_0\tk,\]
this finishes the proof of the lemma.
\end{proof}

\subsection{Condition (iii)}

In order to check conditions (iii) and (iv'), the following notation will be useful: For any $g\in G\sm\lbrace 0\rbrace$, set $\kap(g)=\lbrace g,-g\rbrace\su G\sm\lbrace 0\rbrace$. The relevance of this notion is that for $g,g'\in G\sm\lbrace 0\rbrace$ with $\kap(g)\neq \kap(g')$ the events $g\in \Lambda$ and $g'\in \Lambda$ are independent (and similarly for a collection $\lbrace g_i\rbrace_{i=1}^m\su G\sm\lbrace 0\rbrace$ such that $\kap(g_i)$ for $i=1,\dots, m$ are all distinct, the events $g_i\in \Lambda$ are mutually independent).

In order to check condition (iii), let us first make the following definition.

\begin{definition}\label{defi-pleasant}Let us call a pair $(X', Y')$ of disjoint subsets $X', Y'\su G$ with $\vert X'\vert=\vert Y'\vert=\lfloor \tk^{2/5}\rfloor$ \emph{pleasant} if there exists a collection of pairs $(x_1,y_1),\dots,(x_m,y_m)\in X'\times Y'$ of size $m\geq \frac{1}{8}\tk^{4/5}$ such that $\kap(x_i-y_i)$ for $i=1,\dots,m$ are all distinct.
\end{definition}

We will first show that it is very unlikely for the graph $\tH$ to be complete or empty between $X'$ and $Y'$ for some pleasant pair $(X', Y')$ . Afterwards, we will conclude that condition (iii) holds with probability $1-o(1)$.

\begin{lemma}\label{lemma-pleasant-smallprob}Let $(X',Y')$ be a pleasant pair of disjoint subsets $X', Y'\su G$. Then the probability that the graph $\tH=\Cay(G,\Lambda)$ is either complete or empty between $X'$ and $Y'$ is at most $\exp(-\tk^{3/5}).$
\end{lemma}
\begin{proof}We will first show that the probability that $\tH=\Cay(G,\Lambda)$ is complete between $X'$ and $Y'$ is at most
\[\exp\left(-\frac{10^3}{8}(\log \tk)^{1/2}\tk^{3/5}\right).\]
Let $(x_1,y_1),\dots,(x_m,y_m)\in X'\times Y'$ be a collection of pairs as in Definition \ref{defi-pleasant}. Then $m\geq \frac{1}{8}\tk^{4/5}$ and $\kap(x_i-y_i)$ for $i=1,\dots,m$ are all distinct. Note that in order for $\tH$ to be complete between $X'$ and $Y'$, we in particular need to have $x_i-y_i\in \Lambda$ for all $i=1,\dots,m$. For each $i=1,\dots,m$ the event $x_i-y_i\in \Lambda$ happens with probability $p$ (note that $x_i-y_i\neq 0$ as $X'$ and $Y'$ are disjoint). Furthermore, as the $\kap(x_i-y_i)$ are all distinct, these events are mutually independent. Hence the probability that $x_i-y_i\in \Lambda$ happens for all $i=1,\dots,m$ equals $p^m$. Thus, the probability that $\tH$ is complete between $X'$ and $Y'$ is at most
\[p^{m}\leq p^{\tk^{4/5}/8}\leq (1-p')^{\tk^{4/5}/8}\leq \exp\left(-\frac{1}{8}\tk^{4/5}\cdot p'\right)\leq \exp\left(-\frac{10^3}{8}(\log \tk)^{1/2}\tk^{3/5}\right).\]
Here we used (\ref{eq-bound-pstrich}) and the fact that $1-p\geq p'$ implies $p\leq 1-p'$ .

Analogously we can prove that the probability that $\tH=\Cay(G,\Lambda)$ is empty between $X'$ and $Y'$ is also at most
\[\exp\left(-\frac{10^3}{8}(\log \tk)^{1/2}\tk^{3/5}\right).\]

Hence the probability that $\tH$ is either complete or empty between $X'$ and $Y'$ is at most
\[2\exp\left(-\frac{10^3}{8}(\log \tk)^{1/2}\tk^{3/5}\right)\leq \exp\left(-\tk^{3/5}\right),\]
as desired.\end{proof}

\begin{corollary}\label{coro-prop3}With probability $o(1)$ the graph $\tH=\Cay(G,\Lambda)$ has the property that there exists a pleasant pair $(X', Y')$ of disjoint subsets $X', Y'\su G$ such that $\tH$ is either complete or empty between $X'$ and $Y'$.
\end{corollary}
\begin{proof}Recall that for each pleasant pair $(X', Y')$ we have $\vert X'\vert=\vert Y'\vert=\lfloor \tk^{2/5}\rfloor$. Thus, the number of pleasant pairs $(X', Y')$ is at most
\[\binom{\tk}{\lfloor \tk^{2/5}\rfloor}\cdot \binom{\tk}{\lfloor \tk^{2/5}\rfloor}\leq \tk^{\lfloor \tk^{2/5}\rfloor}\cdot \tk^{\lfloor \tk^{2/5}\rfloor}\leq \tk^{2\tk^{2/5}}=\exp(2(\log \tk)\tk^{2/5}).\]
Using Lemma \ref{lemma-pleasant-smallprob}, by a union bound we obtain that the probability that the graph $\tH=\Cay(G,\Lambda)$ is either complete or empty between $X'$ and $Y'$ for some pleasant pair $(X', Y')$ is at most
\[\exp(2(\log \tk)\tk^{2/5})\cdot \exp(-\tk^{3/5})=o(1).\]
This proves the corollary.\end{proof}

\begin{lemma}\label{lemma2-prop3}For any disjoint subsets $X,Y\su G$ with sizes $\vert X\vert\geq 2\tk^{4/5}$ and $\vert Y\vert\geq 2\tk^{4/5}$, we can find subsets $X'\su X$ and $Y'\su Y$ with $\vert X'\vert=\vert Y'\vert=\lfloor \tk^{2/5}\rfloor$ such that $(X', Y')$ is a pleasant pair.
\end{lemma}
\begin{proof}Let us choose $X'$ uniformly at random among all subsets $X'\su X$ of size $\vert X'\vert=\lfloor \tk^{2/5}\rfloor$, and let us also choose $Y'$ uniformly at random among all subsets $Y'\su Y$ of size $\vert Y'\vert=\lfloor \tk^{2/5}\rfloor$ (and independently of $X'$).

Let $Z$ be the number of quadruples $(x,y,x',y')\in X'\times Y'\times X'\times Y'$ with $x-y=x'-y'$ and $x\neq x'$. Note that for each such quadruple the four entries $x,y,x',y'\in G$ must be distinct (as $X'$ and $Y'$ are disjoint). The total number of quadruples $(x,y,x',y')\in X\times Y\times X\times Y$ with distinct entries and $x-y=x'-y'$ is at most
\[\vert X\vert  \cdot \vert Y\vert\cdot \vert X\vert=\vert X\vert^2\cdot \vert Y\vert.\]
For each quadruple $(x,y,x',y')\in X\times Y\times X\times Y$ with distinct entries, the probability for $x,x'\in X'$ and $y,y'\in Y'$ is
\[\frac{\lfloor \tk^{2/5}\rfloor}{\vert X\vert}\cdot \frac{\lfloor \tk^{2/5}\rfloor-1}{\vert X\vert-1}\cdot \frac{\lfloor \tk^{2/5}\rfloor}{\vert Y\vert}\cdot \frac{\lfloor \tk^{2/5}\rfloor-1}{\vert Y\vert-1}\leq \frac{\lfloor \tk^{2/5}\rfloor}{\vert X\vert}\cdot \frac{\lfloor \tk^{2/5}\rfloor}{\vert X\vert}\cdot \frac{\lfloor \tk^{2/5}\rfloor}{\vert Y\vert}\cdot \frac{\lfloor \tk^{2/5}\rfloor}{\vert Y\vert}\leq \frac{\tk^{8/5}}{\vert X\vert^2\cdot \vert Y\vert^2}.\]
Hence the expected value of $Z$ is at most
\[\frac{\tk^{8/5}}{\vert X\vert^2\cdot \vert Y\vert^2}\cdot \vert X\vert^2\cdot \vert Y\vert=\frac{\tk^{8/5}}{\vert Y\vert}\leq \frac{\tk^{8/5}}{2\tk^{4/5}}=\frac{1}{2}\tk^{4/5}.\]
So we can choose subsets $X'\su X$ and $Y'\su Y$ of sizes $\vert X'\vert=\vert Y'\vert=\lfloor \tk^{2/5}\rfloor$ such that there are at most $\frac{1}{2}\tk^{4/5}$ quadruples $(x,y,x',y')\in X'\times Y'\times X'\times Y'$ with $x-y=x'-y'$ and $x\neq x'$. Since $X$ and $Y$ are disjoint, the sets $X'$ and $Y'$ are disjoint as well.

Let us make a list of all pairs $(x,y)\in X'\times Y'$. This list has length
\[\lfloor \tk^{2/5}\rfloor^2\geq \left(\frac{8}{9}\tk^{2/5}\right)^2\geq \frac{3}{4}\tk^{4/5},\]
where we used (\ref{eq-ineq-k25floor}).

Now, for every quadruple $(x,y,x',y')\in X'\times Y'\times X'\times Y'$ with $x-y=x'-y'$ and $x\neq x'$, let us delete the pair $(x',y')\in X'\times Y'$ from the list. Afterwards, the list still has length at least $\frac{3}{4}\tk^{4/5}-\frac{1}{2}\tk^{4/5}=\frac{1}{4}\tk^{4/5}$. Furthermore, we cannot find two distinct pairs $(x,y)$ and $(x',y')$ on the list with $x-y=x'-y'$ anymore (note that $x-y=x'-y'$ and $(x,y)\neq (x',y')$ automatically imply $x\neq x'$). So the differences $x-y$ are distinct for all the pairs $(x,y)$ on the list.

In order to show that $(X',Y')$ is a pleasant pair, we need a collection of pairs $(x_1,y_1),\dots,(x_m,y_m)\in X'\times Y'$ of size $m\geq \frac{1}{8}\tk^{4/5}$ such that $\kap(x_i-y_i)$ for $i=1,\dots,m$ are all distinct. We can find such a collection by greedily choosing pairs from our list: For every pair $(x,y)$ on the list, there is no other pair $(x',y')$ on the list with $x-y=x'-y'$. Furthermore there is at most one pair $(x',y')$ with $x'-y'=y-x$ (as the differences $x'-y'$ for all the pairs are all distinct). Thus, for every pair $(x,y)$ on the list, there is at most one other pair $(x',y')$ on the list with $\lbrace x'-y', y'-x'\rbrace=\lbrace x-y, y-x\rbrace$, which means $\kap(x-y)=\kap(y'-x')$. So by greedily choosing pairs $(x_i,y_i)\in X'\times Y'$ from our list we can find a collection $(x_1,y_1),\dots,(x_m,y_m)\in X'\times Y'$ of size $m\geq \frac{1}{2}\cdot \frac{1}{4}\tk^{4/5}=\frac{1}{8}\tk^{4/5}$ with all $\kap(x_i-y_i)$ being distinct. This shows that $(X',Y')$ is a pleasant pair.
\end{proof}

Now it is easy to see that $\tH$ satisfies condition (iii) with probability $1-o(1)$. Indeed, by Corollary \ref{coro-prop3}, with probability $1-o(1)$ the graph $\tH$ is not complete or empty between any pleasant pair $(X',Y')$ of disjoint subsets of $G$. Suppose $\tH$ satisfies this, but it does not satisfy condition (iii). Then there are disjoint subsets $X, Y\su V(\tH)=G$ with sizes $\vert X\vert\geq 2\tk^{4/5}$ and $\vert Y\vert\geq 2\tk^{4/5}$ such that between the sets $X$ and $Y$ the graph $\tH$ is complete or empty. But by Lemma \ref{lemma2-prop3} there exist subsets $X'\su X$ and $Y'\su Y$ such that $(X', Y')$ is a pleasant pair, and then $\tH$ is in also complete or empty between $X'$ and $Y'$. This is a contradiction. Hence $\tH$ must satisfy condition (iii) with probability $1-o(1)$.

\subsection{Condition (iv')}

Recall that for any $g\in G\sm\lbrace 0\rbrace$, we defined $\kap(g)=\lbrace g,-g\rbrace\su G\sm\lbrace 0\rbrace$.

Our approach for checking condition (iv') is similar to the way we checked condition (iii) in the previous subsection. However, the technical details are slightly more complicated here. As in the previous subsection, we first start with a definition which the argument will then build on in a similar way.

\begin{definition}\label{defi-nice}For a subset $Y\su G$ of size $\vert Y\vert=\lfloor \tk^{2/5}\rfloor$, let us call an injective map $f:Y\to G$ \emph{nice} if there exists a collection of pairs $(x_1,y_1),\dots,(x_m,y_m)\in Y\times Y$  of size $m\geq \frac{1}{200}(\log \tk)^{1/2}\tk^{3/5}$ such that $x_i\neq y_i$ for every $i=1,\dots,m$ and such that the $2m$ sets $\kap(x_i-y_i)$ and $\kap(f(x_i)-f(y_i))$ for $i=1,\dots,m$ are all distinct.
\end{definition}

We will first show that it is very unlikely for the graph $\tH$ to have the property that there exists a nice map $f:Y\to G$ with $a_{\tH}(v,w)=a_{\tH}(f(v),f(w))$ for all distinct $v,w\in Y$. Afterwards, we will conclude that condition (iv') holds with probability $1-o(1)$.

\begin{lemma}\label{lemma-nice-smallprob}Let $f:Y\to G$ be a nice map. Then the probability that the graph $\tH=\Cay(G,\Lambda)$ satisfies $a_{\tH}(v,w)=a_{\tH}(f(v),f(w))$ for all distinct $v,w\in Y$ is at most $\exp(-5(\log \tk)\cdot \tk^{2/5})$.
\end{lemma}
\begin{proof}Let $(x_1,y_1),\dots,(x_m,y_m)\in Y\times Y$ be a collection of pairs as in Definition \ref{defi-nice}. Note that $x_i\neq y_i$ for $i=1,\dots, m$ and hence the property considered in the lemma in particular implies $a_{\tH}(x_i,y_i)=a_{\tH}(f(x_i),f(y_i))$ for $i=1,\dots, m$. In other words, for every $i=1,\dots, m$ we need $x_i-y_i\in \Lambda$ if and only if $f(x_i)-f(y_i)\in \Lambda$. Note that each of the events $x_i-y_i\in \Lambda$ and also each of the events $f(x_i)-f(y_i)\in \Lambda$ has probability $p$ . Furthermore, the $2m$ events $x_i-y_i\in \Lambda$ and $f(x_i)-f(y_i)\in \Lambda$ for $i=1,\dots,m$ are mutually independent, since $\kap(x_i-y_i)$ and $\kap(f(x_i)-f(y_i))$ for $i=1,\dots,m$ are all distinct. Hence for each $i=1,\dots, m$ the probability that both or neither of $x_i-y_i\in \Lambda$ and $f(x_i)-f(y_i)\in \Lambda$ happen is $p^2+(1-p)^2$. Now, the probability that this is the case for all $i=1,\dots, m$ is
\begin{multline*}
(p^2+(1-p)^2)^m\leq \left(1-10^3(\log \tk)^{1/2}\cdot \tk^{-1/5}\right)^m\leq \exp(-10^3(\log \tk)^{1/2}\cdot \tk^{-1/5}\cdot m)\\
\leq \exp(-10^3(\log \tk)^{1/2}\cdot \tk^{-1/5}\cdot \frac{1}{200}(\log \tk)^{1/2}\tk^{3/5})=\exp(-5(\log \tk)\cdot \tk^{2/5}).
\end{multline*}
Here, we used (\ref{eq-p-ineq2}) for the first inequality. Thus, the probability that $\tH$ satisfies $a_{\tH}(v,w)=a_{\tH}(f(v),f(w))$ for all distinct $v,w\in Y$ is at most $\exp(-5(\log \tk)\cdot \tk^{2/5})$.
\end{proof}

\begin{corollary}\label{coro-prop4}The probability that there is a nice map $f:Y\to G$ such that the graph $\tH=\Cay(G,\Lambda)$ satisfies $a_{\tH}(v,w)=a_{\tH}(f(v),f(w))$ for all distinct $v,w\in Y$ is at most $o(1)$.
\end{corollary}
\begin{proof}Recall that for each nice map  $f:Y\to G$ we have $\vert Y\vert=\lfloor \tk^{2/5}\rfloor$. Hence the number of nice maps $f:Y\to G$ is at most
\[\binom{\tk}{ \lfloor \tk^{2/5}\rfloor}\tk^{\lfloor \tk^{2/5}\rfloor}\leq \tk^{\lfloor \tk^{2/5}\rfloor}\cdot \tk^{\lfloor \tk^{2/5}\rfloor}=\exp(2(\log \tk)\cdot\lfloor \tk^{2/5}\rfloor)\leq \exp(2(\log \tk)\cdot \tk^{2/5}).\]
Using Lemma \ref{lemma-nice-smallprob}, by a union bound we obtain that with probability at most
\[\exp(2(\log \tk)\cdot \tk^{2/5})\cdot \exp(-5(\log \tk)\cdot \tk^{2/5})=\exp(-3(\log \tk)\cdot \tk^{2/5})=o(1).\]
there is a nice map $f:Y\to G$ such that the graph $\tH=\Cay(G,\Lambda)$ satisfies $a_{\tH}(v,w)=a_{\tH}(f(v),f(w))$ for all distinct $v,w\in Y$.\end{proof}

The main step for checking condition (iv') will be to prove the following lemma.

\begin{lemma}\label{lemma-prop4}Let $X\su V(\tH)$ be a subset of size $\vert X\vert\geq (1-\delta_0)\tk$ and let $f:X\to G$ be an injective map. Assume that there is no element $g\in G$ such that $f(x)=x+g$ for at least $(1-2\delta_0)\tk$ elements $x\in X$, and also assume that there is no element $g\in G$ such that $f(x)=-x+g$ for at least $(1-2\delta_0)\tk$ elements $x\in X$. 
Then we can find a subset $Y\su X$ of size $\vert Y\vert=\lfloor \tk^{2/5}\rfloor$ such that $f\vert_Y$ is a nice map.
\end{lemma}

Before starting the proof of Lemma \ref{lemma-prop4}, let us prove another lemma.

\begin{lemma}\label{lemma-many-kappa-pairs}Suppose that $X\su V(H)$ and $f:X\to G$ satisfy the assumptions of Lemma \ref{lemma-prop4}. Then there are at least $\frac{1}{12}\delta_0\tk^2$ pairs $(x,y)\in X\times X$ satisfying $x\neq y$ and $\kap(x-y)\neq \kap(f(x)-f(y))$.
\end{lemma}
\begin{proof}Consider the map $h:X\to G$ defined by $h(x)=f(x)-x$ for all $x\in X$. Then, for each $g\in G$, the preimage $h^{-1}(g)$ consist precisely of those $x\in X$ with $f(x)=x+g$. Hence, by the assumptions in Lemma \ref{lemma-prop4} we have $\vert h^{-1}(g)\vert<(1-2\delta_0)\tk$ for every $g\in G$.

Similarly, let us consider the map $h':X\to G$ defined by $h'(x)=f(x)+x$ for all $x\in X$. Then, for each $g\in G$, the preimage $h'^{-1}(g)$ consist precisely of those $x\in X$ with $f(x)=-x+g$. Hence, by the assumptions in Lemma \ref{lemma-prop4} we also have $\vert h'^{-1}(g)\vert<(1-2\delta_0)\tk$ for every $g\in G$.

\begin{claim}\label{claim-hhprime}If we have $\vert h^{-1}(g)\cap h'^{-1}(g')\vert>\tk/2$ for some (not necessarily distinct) elements $g,g'\in G$, then $h=h'$.
\end{claim}
\begin{proof}
Assume that $g,g'\in G$ satisfy $\vert h^{-1}(g)\cap h'^{-1}(g')\vert>\tk/2$, and set $U=h^{-1}(g)\cap h'^{-1}(g')$. Then for every $x\in U$ we have $f(x)-x=h(x)=g$ and $f(x)+x=h'(x)=g'$. Hence
\[x+x=(f(x)+x)-(f(x)-x)=g'-g\]
for every $x\in U$. Now, set
\[U'=\lbrace x\in G\mid x+x=g'-g\rbrace,\]
then $U\su U'$ and consequently $\vert U'\vert\geq \vert U\vert=\vert h^{-1}(g)\cap h'^{-1}(g')\vert>\tk/2$. On the other hand, let $U_0\su G$ be the subgroup consisting of all those $y\in G$ with $y+y=0$. Then $U'$ is a coset of $U_0$ and in particular $\vert U'\vert>\tk/2$ implies $\vert U_0\vert>\tk/2$. But as $U_0\su G$ is a subgroup, this means that we must have $U_0=G$. Thus, $y+y=0$ for all $y\in G$ and in particular for all $x\in X$ we have $h(x)=f(x)-x=f(x)+x=h'(x)$.
\end{proof}

We claim that for distinct $x,y\in X$ with $h(x)\neq h(y)$ and $h'(x)\neq h'(y)$ we have $\kap(x-y)\neq \kap(f(x)-f(y))$. Indeed, note that $\kap(x-y)= \kap(f(x)-f(y))$ would mean $\lbrace x-y,y-x\rbrace=\lbrace f(x)-f(y), f(y)-f(x)\rbrace$. Hence $f(x)-f(y)=x-y$ or $f(x)-f(y)=y-x$. But then $h(x)=f(x)-x=f(y)-y=h(y)$ or $h'(x)=f(x)+x=f(y)+y=h'(y)$, a contradiction. So for any distinct $x,y\in X$ with $h(x)\neq h(y)$ and $h'(x)\neq h'(y)$ we indeed have $\kap(x-y)\neq \kap(f(x)-f(y))$.

Now, we want to construct a suitably chosen set $A_1\su X$ which is the union of some of the preimages $h^{-1}(g)$ and satisfies $\vert A_1\vert\geq \frac{1}{3}\tk$. If there exists a $g\in G$ with $\vert h^{-1}(g)\vert\geq \frac{1}{3}\tk$, then let us take $A_1=h^{-1}(g)$ (if there are multiple choices for $g\in G$ with $\vert h^{-1}(g)\vert\geq \frac{1}{3}\tk$ just take any of them). Otherwise, we have $\vert h^{-1}(g)\vert<\frac{1}{3}\tk$ for every $g\in G$. But then starting from $A_1=\emptyset$ and successively adding the preimages $h^{-1}(g)$ to $A_1$, one at a time, we will reach a point when $\frac{1}{3}\tk\leq \vert A_1\vert\leq \frac{2}{3}\tk$. Then let us fix $A_1$ at that point. So in either case the following claim holds by construction of $A_1$.

\begin{claim}\label{claim-A1}The set $A_1\su X$ satisfies $\vert A_1\vert\geq \frac{1}{3}\tk$ and it is the union of some of the preimages $h^{-1}(g)$. Furthermore, if $\vert A_1\vert> \frac{2}{3}\tk$, then $A_1$ just consists of a single preimage $h^{-1}(g)$ for some $g\in G$.
\end{claim}

Now set $A_2=X\sm A_1$. Then $A_2$ is the union of the remaining preimages $h^{-1}(g)$, and as $A_1\cap A_2=\emptyset$ this implies $h(x)\neq h(y)$ for all $(x,y)\in A_1\times A_2$.

\begin{claim}\label{claim-A2}We always have $\vert A_2\vert\geq\delta_0\tk$ and furthermore  $\vert A_2\vert<\frac{1}{4}\tk$ is only possible if the set $A_1$ just consists of a single preimage $h^{-1}(g)$ for some $g\in G$.
\end{claim}
\begin{proof}
Recall that $\delta_0=p'/100<\frac{1}{12}$. Suppose that $\vert A_2\vert<\frac{1}{4}\tk$. Then $\vert A_1\vert=\vert X\vert- \vert A_2\vert>(1-\delta_0)\tk-\frac{1}{4}\tk>\frac{2}{3}\tk$. By the last part of Claim \ref{claim-A1}, this is indeed only possible if $A_1$ just consists of a single preimage $h^{-1}(g)$ for some $g\in G$. But then in particular $\vert A_1\vert=\vert h^{-1}(g)\vert<(1-2\delta_0)\tk$, and therefore $\vert A_2\vert\geq (1-\delta_0)\tk-\vert A_1\vert\geq \delta_0\tk$.
\end{proof}

From the first part of Claim \ref{claim-A1} and the first part of Claim \ref{claim-A2}, we in particular obtain
\begin{equation}\label{eq-A1-A2}
\vert A_1\vert\cdot \vert A_2\vert\geq \frac{1}{3}\tk\cdot \delta_0\tk=\frac{1}{3}\delta_0\tk^2.
\end{equation}

It suffices to show that there are at least $\frac{1}{12}\delta_0\tk^2$ distinct pairs $(x,y)\in A_1\times A_2$ with $h'(x)\neq h'(y)$. Indeed, recall that $h(x)\neq h(y)$ for all $(x,y)\in A_1\times A_2$ and that any distinct $x,y\in X$ with $h(x)\neq h(y)$ and $h'(x)\neq h'(y)$ satisfy the desired property $\kap(x-y)\neq \kap(f(x)-f(y))$.

If $h=h'$, then for all $(x,y)\in A_1\times A_2$ we have $h'(x)\neq h'(y)$ (since  $h(x)\neq h(y)$). So in this case the number of pairs $(x,y)\in A_1\times A_2$ with $h'(x)\neq h'(y)$ is $\vert A_1\vert\cdot \vert A_2\vert\geq\frac{1}{3}\delta_0\tk^2$ by (\ref{eq-A1-A2}) and we are done.

So from now on we can assume that $h\neq h'$. Then Claim \ref{claim-hhprime} implies $\vert h^{-1}(g)\cap h'^{-1}(g')\vert\leq \tk/2$ for all $g,g'\in G$.

First suppose that we have $\vert h'^{-1}(g')\cap A_2\vert\leq \frac{1}{2}\vert A_2\vert$ for all $g'\in G$. Then in particular, for every $x\in A_1$ we have $\vert h'^{-1}(h'(x))\cap A_2\vert\leq \frac{1}{2}\vert A_2\vert$, which means that there are at most $\frac{1}{2}\vert A_2\vert$ elements $y\in A_2$ with $h'(y)=h'(x)$. Thus, for each $x\in A_1$ there are at least $\frac{1}{2}\vert A_2\vert$ elements $y\in A_2$ with $h'(x)\neq h'(y)$. So the number of pairs $(x,y)\in A_1\times A_2$ with $h'(x)\neq h'(y)$ is at least $\vert A_1\vert\cdot \frac{1}{2}\vert A_2\vert\geq\frac{1}{6}\delta_0\tk^2$ by (\ref{eq-A1-A2}) and we are done.

So we may assume that there exists some $g'\in G$ with $\vert h'^{-1}(g')\cap A_2\vert> \frac{1}{2}\vert A_2\vert$. Clearly, such a $g'$ is unique. Set $B=h'^{-1}(g')$, then $\vert A_2\cap B\vert> \frac{1}{2}\vert A_2\vert\geq \frac{1}{2}\delta_0\tk$ by Claim \ref{claim-A2}.

Note that whenever $(x,y)\in A_1\times A_2$ satisfy $x\not\in B$ and $y\in B$, then $h'(x)\neq h'(y)$. Hence, the number of pairs $(x,y)\in A_1\times A_2$ with $h'(x)\neq h'(y)$ is at least $\vert A_1\sm B\vert\cdot \vert A_2\cap B\vert$. If $\vert A_1\sm B\vert\geq \frac{1}{6}\tk$, this is at least $\frac{1}{6}\tk\cdot \frac{1}{2}\delta_0\tk=\frac{1}{12}\delta_0\tk^2$ and we are done. Hence we may assume that $\vert A_1\sm B\vert< \frac{1}{6}\tk$.

We claim that then we must have $\vert A_2\vert\geq \frac{1}{4}\tk$. Indeed, let us assume that   $\vert A_2\vert< \frac{1}{4}\tk$. Then by Claim \ref{claim-A2} we have $A_1=h^{-1}(g)$ for some $g\in G$ and therefore $\vert A_1\cap B\vert=\vert h^{-1}(g)\cap h'^{-1}(g')\vert\leq \tk/2$ by Claim \ref{claim-hhprime}. Furthermore $\vert A_1\vert=\vert X\vert-\vert A_2\vert\geq (1-\delta_0)\tk-\frac{1}{4}\tk\geq \frac{2}{3}\tk$. Hence $\vert A_1\sm B\vert=\vert A_1\vert-\vert A_1\cap B\vert\geq \frac{2}{3}\tk-\frac{1}{2}\tk=\frac{1}{6}\tk$, but this contradicts the assumption we made at the end of the previous paragraph.

Thus, we indeed have $\vert A_2\vert\geq \frac{1}{4}\tk$ and consequently $\vert  A_2\cap B\vert> \frac{1}{2}\vert A_2\vert\geq \frac{1}{8}\tk$. From $\vert A_1\sm B\vert< \frac{1}{6}\tk$ and $\vert A_1\vert\geq \frac{1}{3}\tk$ (by Claim \ref{claim-A1}), we also obtain $\vert A_1\cap B\vert> \frac{1}{6}\tk\geq \frac{1}{8}\tk$.

We already saw that whenever $(x,y)\in A_1\times A_2$ satisfy $x\not\in B$ and $y\in B$, then $h'(x)\neq h'(y)$. Similarly, whenever $(x,y)\in A_1\times A_2$ satisfy $x\in B$ and $y\not\in B$, we also have $h'(x)\neq h'(y)$. Hence, the number of pairs $(x,y)\in A_1\times A_2$ with $h'(x)\neq h'(y)$ is at least
\[\vert A_1\sm B\vert\cdot \vert A_2\cap B\vert+\vert A_2\sm B\vert\cdot \vert A_1\cap B\vert\geq \vert A_1\sm B\vert\cdot \frac{1}{8}\tk+\vert A_2\sm B\vert\cdot \frac{1}{8}\tk=\vert X\sm B\vert\cdot \frac{1}{8}\tk.\]
Recall that $\vert B\vert=\vert h'^{-1}(g')\vert\leq (1-2\delta_0)\tk$. Thus, $\vert X\sm B\vert\geq \delta_0\tk$ and the number of pairs $(x,y)\in A_1\times A_2$ with $h'(x)\neq h'(y)$ is at least $\delta_0\tk\cdot \frac{1}{8}\tk=\frac{1}{8}\delta_0\tk^2$. This finishes the proof of Lemma \ref{lemma-many-kappa-pairs}.\end{proof}

Now, we are ready for the proof of Lemma \ref{lemma-prop4}.

\begin{proof}[Proof of  Lemma \ref{lemma-prop4}]Recall that $X\su V(\tH)$ is a subset of size $\vert X\vert\geq (1-\delta_0)\tk$ and $f:X\to G$ is an injective map. Using the further assumptions, Lemma \ref{lemma-many-kappa-pairs} implies that there are at least $\frac{1}{12}\delta_0\tk^2$ pairs $(x,y)\in X\times X$ satisfying $x\neq y$ and $\kap(x-y)\neq \kap(f(x)-f(y))$. Let us call these pairs \emph{tame}.

Let us choose $Y\su X$ uniformly at random among all subsets $Y\su X$ of size $\vert Y\vert=\lfloor \tk^{2/5}\rfloor$.

Let $Z_0$ be the number of tame pairs $(x,y)\in Y\times Y$. For each tame pair $(x,y)\in X\times X$ the probability for $x,y\in Y$ is, using (\ref{eq-ineq-k25floor}),
\[\frac{\lfloor \tk^{2/5}\rfloor}{\vert X\vert}\cdot \frac{\lfloor \tk^{2/5}\rfloor-1}{\vert X\vert-1}\geq \frac{\frac{8}{9}\tk^{2/5}}{\tk}\cdot \frac{\frac{1}{2}\cdot \frac{8}{9}\tk^{2/5}}{\tk}=\frac{32}{81}\tk^{-6/5}\geq \frac{1}{3}\tk^{-6/5}.\]
Since there are at least $\frac{1}{12}\delta_0\tk^2$ tame pairs in $X\times X$, the expectation for the number $Z_0$ of tame pairs in $Y\times Y$ is at least
\[\frac{1}{12}\delta_0\tk^2\cdot \frac{1}{3}\tk^{-6/5}=\frac{1}{36}\delta_0\tk^{4/5}=\frac{1}{36}\cdot \frac{p'}{100}\cdot \tk^{4/5}\geq \frac{1}{4}\cdot (\log \tk)^{1/2}\cdot \tk^{3/5}.\]
Here we used (\ref{eq-bound-pstrich}).

Let $Z_1$ be the number of quadruples $(x,y,x',y')\in Y^4$ with distinct entries $x,y,x',y'$ and $x-y=x'-y'$. The total number of quadruples $(x,y,x',y')\in X^4$ with distinct entries and $x-y=x'-y'$ is at most $\vert X\vert^3$. For each quadruple $(x,y,x',y')\in X^4$ with distinct entries, the probability for $x,y,x',y'\in Y$ is
\[\frac{\lfloor \tk^{2/5}\rfloor}{\vert X\vert}\cdot \frac{\lfloor \tk^{2/5}\rfloor-1}{\vert X\vert-1}\cdot \frac{\lfloor \tk^{2/5}\rfloor-2}{\vert X\vert-2}\cdot \frac{\lfloor \tk^{2/5}\rfloor-3}{\vert X\vert-3}\leq \left(\frac{\lfloor \tk^{2/5}\rfloor}{\vert X\vert}\right)^4\leq \frac{\tk^{8/5}}{\vert X\vert^4}.\]
Hence the expected value of $Z_1$ is at most
\[\frac{\tk^{8/5}}{\vert X\vert^4}\cdot \vert X\vert^3=\frac{\tk^{8/5}}{\vert X\vert}\leq \frac{\tk^{8/5}}{(1-\delta_0)\tk}\leq 2\tk^{3/5}.\]

Similarly, let $Z_2$ be the number of quadruples $(x,y,x',y')\in Y^4$ with distinct entries $x,y,x',y'$ and with $f(x)-f(y)=f(x')-f(y')$. By the injectivity of $f$, the total number of quadruples $(x,y,x',y')\in X^4$ with distinct entries and $f(x)-f(y)=f(x')-f(y')$ is at most $\vert X\vert^3$. We already saw that for each quadruple $(x,y,x',y')\in X^4$ with distinct entries, the probability for $x,y,x',y'\in Y$ is at most $\tk^{8/5}/\vert X\vert^4$. Hence the expected value of $Z_2$ is also at most
\[\frac{\tk^{8/5}}{\vert X\vert^4}\cdot \vert X\vert^3\leq 2\tk^{3/5}.\]

All in all, we find that
\[\mathbb{E}[Z_0-Z_1-Z_2]\geq \frac{1}{4}\cdot (\log \tk)^{1/2}\cdot \tk^{3/5}-2\tk^{3/5}-2\tk^{3/5}\geq \frac{1}{8}\cdot (\log \tk)^{1/2}\cdot \tk^{3/5},\]
recalling that we assumed $\log \tk\geq 2^{10}$. So we can choose a subset $Y\su X$ of size $\vert Y\vert=\lfloor \tk^{2/5}\rfloor$ such that $Z_0-Z_1-Z_2\geq \frac{1}{8}\cdot (\log \tk)^{1/2}\cdot \tk^{3/5}$.

Let us make a list of all the $Z_0$ tame pairs $(x,y)\in Y\times Y$.

Now, for every quadruple $(x,y,x',y')\in Y^4$ with distinct entries $x,y,x',y'$ and $x-y=x'-y'$, let us delete the pair $(x',y')$ from the list (if it appears on the list). Afterwards, any two pairs $(x,y)$ and $(x',y')$ on the list with $x-y=x'-y'$ must satisfy $x=y'$ or $y=x'$ (note that if $x=x'$ or $y=y'$, then already $(x,y)=(x',y')$). Hence for each pair $(x,y)$ on the list, there can be at most two other pairs $(x',y')$ on the list satisfying $x-y=x'-y'$. Therefore, each difference $x-y$ occurs for at most three pairs $(x,y)$ on the list.

Furthermore, for every quadruple $(x,y,x',y')\in Y^4$ with distinct entries $x,y,x',y'$ and $f(x)-f(y)=f(x')-f(y')$, let us delete the pair $(x',y')$ from the list (if it appears on the list). Using that $f$ is injective, we can see similarly to the previous argument that afterwards  each difference $f(x)-f(y)$ occurs for at most three pairs $(x,y)$ on the list.

After the deletion process, the list still has length at least $Z_0-Z_1-Z_2\geq \frac{1}{8}\cdot (\log \tk)^{1/2}\cdot \tk^{3/5}$.

In order to establish that $f\vert_Y$ is a nice map, we need a collection of pairs $(x_1,y_1),\dots,(x_m,y_m)\in Y\times Y$  of size $m\geq \frac{1}{200}(\log \tk)^{1/2}\tk^{3/5}$ such that $x_i\neq y_i$ for every $i=1,\dots,m$ and such that the $2m$ sets $\kap(x_i-y_i)$ and $\kap(f(x_i)-f(y_i))$ for $i=1,\dots,m$ are all distinct. We can find such a collection by greedily choosing pairs from our list: For every pair $(x,y)$ on the list, among the other pairs of the list there are at most three pairs $(x',y')$ with $x'-y'=x-y$, at most three pairs with $x'-y'=y-x$, at most three pairs with $x'-y'=f(x)-f(y)$ and at most three pairs with $x'-y'=f(y)-f(x)$. Thus, the list contains at most $12$ pairs $(x',y')$ with $x'-y'\in \lbrace x-y,y-x,f(x)-f(y),f(y)-f(x)\rbrace$ and similarly at most $12$ pairs $(x',y')$ with $f(x')-f(y')\in \lbrace x-y,y-x,f(x)-f(y),f(y)-f(x)\rbrace$. Hence there are at most 24 other pairs $(x',y')$ on the list with $\kap(x'-y')=\kap(x-y)$ or $\kap(x'-y')=\kap(f(x)-f(y))$ or $\kap(f(x')-f(y'))=\kap(x-y)$ or $\kap(f(x')-f(y'))=\kap(f(x)-f(y))$. So by greedily choosing pairs $(x_i,y_i)\in X'\times Y'$ from our list we can find a collection $(x_1,y_1),\dots,(x_m,y_m)\in X'\times Y'$ of size $m\geq \frac{1}{25}\cdot \frac{1}{8}\cdot (\log \tk)^{1/2}\cdot \tk^{3/5}=\frac{1}{200}\cdot (\log \tk)^{1/2}\cdot \tk^{3/5}$ such that $\kap(x_i-y_i)\neq \kap(x_j-y_j)$ and $\kap(x_i-y_i)\neq \kap(f(x_j)-f(y_j))$ and $\kap(f(x_i)-f(y_i))\neq \kap(f(x_j)-f(y_j))$ for all $i\neq j$. Since each pair $(x_i,y_i)$ is tame, we also have $x_i\neq y_i$ and $\kap(x_i-y_i)\neq \kap(f(x_i)-f(y_i))$ for $i=1,\dots,m$. Hence the $2m$ sets $\kap(x_i-y_i)$ and $\kap(f(x_i)-f(y_i))$ for $i=1,\dots,m$ are all distinct and $f\vert_Y$ is a nice map.\end{proof}

Now it is easy to see that $\tH$ satisfies condition (iv') with probability $1-o(1)$. Indeed, by Corollary \ref{coro-prop3}, with probability $1-o(1)$ there does not exist a nice map $f:Y\to G$ such that the graph $\tH=\Cay(G,\Lambda)$ satisfies $a_{\tH}(v,w)=a_{\tH}(f(v),f(w))$ for all distinct $v,w\in Y$. We claim that in this case $\tH$ automatically has condition (iv'). Suppose that there is injective map $f:X\to V(\tH)$ with $\vert X\vert\geq (1-2\delta_0)\tk$ and $a_{\tH}(v,w)=a_{\tH}(f(v),f(w))$ for all distinct $v,w\in X$, which contradicts condition (iv'). Then $f$ interpreted as a map $X\to G$ satisfies the assumptions of Lemma \ref{lemma-prop4}. But then, by Lemma \ref{lemma-prop4}, we can find a subset $Y\su X$ such that $f\vert_Y$ is a nice map. This is a contradiction. Hence $\tH$ satisfies condition (iv') with probability $1-o(1)$.


\textit{Acknowledgements.} We would like to thank the anonymous referee for their careful reading of this paper and many helpful comments.

\section*{Appendix}

\subsection*{Proof of Lemma \ref{lemma-functions-E}}

Here, we give proof of Lemma \ref{lemma-functions-E}. Note that the proof of part (iii) is essentially identical with the proof of Proposition 4.3(ii) in \cite{fox-huang-lee}, we just need a slightly better estimate here.

For (i), we need to show that $\tilde{m}_1\dotsm \tilde{m}_\l\leq \E_\l(m)$ for any non-negative integers $\tilde{m}_1,\dots,\tilde{m}_\l$ with $\tilde{m}_1+\dots+\tilde{m}_\l\leq m$. We may assume that $\tilde{m}_1,\dots,\tilde{m}_\l$ are chosen such that the product $\tilde{m}_1\dotsm \tilde{m}_\l$ is maximized under the constraint $\tilde{m}_1+\dots+\tilde{m}_\l\leq m$. Then clearly $\tilde{m}_1+\dots +\tilde{m}_\l=m$. We may also assume without loss of generality that $\tilde{m}_1\geq \dots\geq \tilde{m}_\l$. Suppose we had $\tilde{m}_1-\tilde{m}_\l\geq 2$. Then replacing $\tilde{m}_1$ with $\tilde{m}_1-1$ and replacing $\tilde{m}_\l$ with $\tilde{m}_\l+1$ would strictly increase the product $\tilde{m}_1\dotsm \tilde{m}_\l$, a contradiction to our choice of $\tilde{m}_1,\dots,\tilde{m}_\l$. Hence we must have $\tilde{m}_1-\tilde{m}_\l\leq 1$. But this implies that $\lceil \frac{m}{\l}\rceil\geq \tilde{m}_1\geq \dots\geq \tilde{m}_\l\geq \lfloor \frac{m}{\l}\rfloor$ and $\tilde{m}_1,\dots,\tilde{m}_\l$ agree with the non-negative integers $m_1,\dots,m_{\l}$ in Definition \ref{defi-functions-E}. Hence $\tilde{m}_1\dotsm \tilde{m}_\l= \E_\l(m)$. This shows that $\tilde{m}_1\dotsm \tilde{m}_\l\leq \E_\l(m)$ for any choice of non-negative integers $\tilde{m}_1,\dots,\tilde{m}_\l$ with $\tilde{m}_1+\dots+\tilde{m}_\l\leq m$.

For (ii) and (iii), let $m_1,\dots,m_{\l}$ be the non-negative integers with $\lceil \frac{m}{\l}\rceil\geq m_1\geq \dots\geq m_\l\geq \lfloor \frac{m}{\l}\rfloor$ and $m=m_1+\dots+m_{\l}$. Then $\E_\l(m)=m_1\dotsm m_\l$. Similarly, let $m_1',\dots,m_{\l'}'$ be the non-negative integers with $\lceil \frac{m'}{\l'}\rceil\geq m_1'\geq \dots\geq m_{\l'}'\geq \lfloor \frac{m'}{\l'}\rfloor$ and $m'=m_1'+\dots+m_{\l'}'$, then $\E_{\l'}(m')=m_1'\dotsm m_{\l'}'$.

For (ii), note that we have $m+m'=m_1+\dots+m_{\l}+m_1'+\dots+m_{\l'}$. Therefore, by part (i),
\[\E_\l(m)\cdot\ \E_{\l'}(m')=m_1\dotsm m_\l\cdot m_1'\dotsm m_{\l'}'\leq \E_{\l+\l'}(m+m').\]

For (iii), recall that we are assuming $m\geq \l$ and $\l'\leq \l$ as well as $m'\leq (1-\mu)m$ for some $\mu\geq 0$. From $m\geq \l$ we obtain $\lfloor \frac{m}{\l}\rfloor\geq 1$, hence $m_i\geq 1$ for $i=1,\dots,\l$. Furthermore we have $m_i\geq \lfloor \frac{m}{\l}\rfloor\geq \frac{m}{2\l}$ and $m_i\leq \lceil \frac{m}{\l}\rceil\leq \frac{2m}{\l}$ for each $i=1,\dots,\l$.

Note that $m_{1}+\dots+m_{\l'}$ is the sum of the $\l'$ largest of the $\l$ summands in $m_1+\dots+m_\l=m$, hence
\[m_1+\dots+m_{\l'}\geq \frac{\l'}{\l}\cdot (m_1+\dots+m_\l)=\frac{\l'}{\l}\cdot m.\]
Therefore
\[\sum_{i=1}^{\l'}(m_i-m_i')=\sum_{i=1}^{\l'} m_i-\sum_{i=1}^{\l'}m_i'\geq \frac{\l'}{\l}\cdot m-m'\geq \left(1-\frac{\l-\l'}{\l}\right)\cdot m-(1-\mu)m=\left(\mu-\frac{\l-\l'}{\l}\right)\cdot m.\]

We first claim that we either have $m_i\geq m_i'$ for all $i=1,\dots,\l'$ or $m_i\leq m_i'$ for all $i=1,\dots,\l'$. Suppose for contradiction that there is both an index $i\in \lbrace 1,\dots,\l'\rbrace$ with $m_i< m_i'$ and an index $j\in \lbrace 1,\dots,\l'\rbrace$ with $m_j>m_j'$. Clearly, $i\neq j$. If $i<j$, then $m_i'>m_i\geq m_j>m_j'$ and therefore $m_i'-m_j'\geq 2$, which is a contradiction to $\lceil \frac{m'}{\l'}\rceil\geq m_1'\geq \dots\geq m_{\l'}'\geq \lfloor \frac{m'}{\l'}\rfloor$. If $i>j$, then $m_j>m_j'\geq m_i'>m_i$ and $m_j-m_i\geq 2$ yields a similar contradiction. So we indeed either have have $m_i\geq m_i'$ for all $i=1,\dots,\l'$ or $m_i\leq m_i'$ for all $i=1,\dots,\l'$.

If $m_i\geq m_i'$ for $i=1,\dots,\l'$, then we have
\begin{multline*}
\E_{\l'}(m')=\prod_{i=1}^{\l'}m_i'=\prod_{i=1}^{\l'}\left(1-\frac{m_i-m_i'}{m_i}\right)m_i\leq \prod_{i=1}^{\l'}\left(1-\frac{m_i-m_i'}{2m/\l}\right)m_i\leq \prod_{i=1}^{\l'} e^{-(m_i-m_i')\cdot \l/(2m)}m_i\\
=\left(\prod_{i=1}^{\l'}m_i\right)\cdot \exp\left(-\sum_{i=1}^{\l'}(m_i-m_i')\cdot \frac{\l}{2m}\right)\leq \left(\prod_{i=1}^{\l'}m_i\right)\cdot \exp\left(-\left(\mu-\frac{\l-\l'}{\l}\right)\cdot \frac{\l}{2}\right)\\
=e^{(\l-\l')-\mu\l/2}\cdot \prod_{i=1}^{\l'}m_i\leq e^{(\l-\l')-\mu\l/2}\cdot \left(\frac{2\l}{m}\right)^{\l-\l'}\cdot \prod_{i=1}^{\l}m_i\leq e^{2(\l-\l')-\mu\l/2}\cdot \left(\frac{\l}{m}\right)^{\l-\l'}\cdot \E_\l(m).
\end{multline*}

If $m_i\leq m_i'$ for $i=1,\dots,\l'$, then we have
\begin{multline*}
\E_{\l'}(m')=\prod_{i=1}^{\l'}m_i'=\prod_{i=1}^{\l'}\left(1+\frac{m_i'-m_i}{m_i}\right)m_i\leq \prod_{i=1}^{\l'}\left(1+\frac{m_i'-m_i}{m/(2\l)}\right)m_i\leq \prod_{i=1}^{\l'} e^{-(m_i-m_i')\cdot 2\l/m}m_i\\
=\left(\prod_{i=1}^{\l'}m_i\right)\cdot \exp\left(-\sum_{i=1}^{\l'}(m_i-m_i')\cdot \frac{2\l}{m}\right)\leq \left(\prod_{i=1}^{\l'}m_i\right)\cdot \exp\left(-\left(\mu-\frac{\l-\l'}{\l}\right)\cdot 2\l\right)\\
=e^{2(\l-\l')-2\mu\l}\cdot \prod_{i=1}^{\l'}m_i\leq e^{2(\l-\l')-2\mu\l}\cdot \left(\frac{2\l}{m}\right)^{\l-\l'}\cdot \prod_{i=1}^{\l}m_i\leq e^{3(\l-\l')-2\mu\l}\cdot \left(\frac{\l}{m}\right)^{\l-\l'}\cdot \E_\l(m).
\end{multline*}
This finishes the proof of Lemma \ref{lemma-functions-E}.

\subsection*{Proof of Lemma \ref{lemma-epsilon-parameters}}

As $q\leq 10^{-20}$ by (\ref{ineq-q-lower-bound}), inequality (\ref{ineq-eps1-q}) follows straight from the definition of the $\eps_i$. Similarly (\ref{ineq-eps2-100}), (\ref{ineq-eps4-eps3}), (\ref{ineq-eps5-eps4}) and (\ref{ineq-eps5-q}) follow immediately from the definitions. Furthermore, note that (\ref{ineq-eps2-delta}) follows straight from (\ref{ineq-delta-q}). Let us now check the remaining inequalities (\ref{ineq-eps2-eps1}), (\ref{ineq-eps3-eps2}), (\ref{ineq-eps3-k}) and (\ref{ineq-eps5-logk-k}).

For (\ref{ineq-eps2-eps1}) and (\ref{ineq-eps3-eps2}) recall that we assumed $q\leq 10^{-20}$, see (\ref{ineq-q-lower-bound}). Hence
\[10^{-3}\cdot q^{-1}\geq q^{-1/2}\geq \log (q^{-1/2})=\frac{1}{2}\log (1/q)\]
and therefore
\[\frac{2\cdot 10^{-3}}{\log(1/q)}\geq q.\]

Now, for (\ref{ineq-eps2-eps1}), note that 
\[\eps_2=10^{-2}\frac{q}{\log(1/q)}>\frac{2\cdot 10^{-3}}{\log(1/q)}\cdot q\geq q^2\]
and consequently $\log(1/\eps_2)<\log (1/q^2)=2\log(1/q)$. Thus, as $\log 2>\frac{1}{2}$,
\[\log(1/\eps_2)\cdot \eps_2<2\log(1/q)\cdot \eps_2=\frac{q}{50}<\frac{\log 2}{8}\cdot \frac{q}{3}=\frac{\log 2}{8}\eps_1.\]

Similarly, for (\ref{ineq-eps3-eps2}), note that
\[\eps_3=10^{-5}\frac{q}{(\log(1/q))^2}>\left(\frac{2\cdot 10^{-3}}{\log(1/q)}\right)^2\cdot q\geq q^3\]
and consequently $\log(1/\eps_2)<\log (1/q^3)=3\log(1/q)$
\[\log(1/\eps_3)\cdot \eps_3<3\log(1/q)\cdot \eps_3=3\cdot 10^{-5}\frac{q}{\log(1/q)}<10^{-4}\frac{q}{\log(1/q)}=\frac{1}{100}\eps_2.\]

For (\ref{ineq-eps3-k}), first note that from (\ref{ineq-q-lower-bound}) we obtain $q^3\geq 10^{12}(\log k)^3k^{-3/5}$. On the other hand, (\ref{ineq-k-logtk}) implies $k^{2/5}>k^{2/20}>10^6(\log k)^4$. Thus,
\[q^3>10^{18}\frac{(\log k)^7}{k}>10^{16}\frac{(\log k)^6}{k}.\]
and therefore (as $q>1/k$ by (\ref{ineq-q-lower-bound}))
\[\frac{q^3}{(\log(1/q))^{4}}>\frac{q^3}{(\log k)^{4}}>10^{16}\frac{(\log k)^2}{k}.\]
Now we indeed obtain
\[\eps_3^2=10^{-10}\frac{q^2}{(\log(1/q))^4}>\frac{10^6}{q}\cdot \frac{(\log k)^2}{k},\]
as desired.

For (\ref{ineq-eps5-logk-k}), note that (\ref{ineq-q-lower-bound}) implies
\[q^5\geq 10^{20}\frac{(\log k)^6}{k}.\]
and therefore (as $q>k^{-1/5}$ by (\ref{ineq-q-lower-bound}))
\[\frac{q^5}{(\log(1/q))^{4}}>\frac{q^5}{(\frac{1}{5}\log k)^{4}}\geq 5^4\cdot 10^{20}\cdot \frac{(\log k)^2}{k}>10^{19}\cdot 40\cdot \frac{(\log k)^2}{k}.\]
Hence indeed
\[\eps_5=10^{-19}\frac{q^4}{(\log(1/q))^4}>\frac{40}{q}\cdot \frac{(\log k)^2}{k}>\frac{10^3}{q}\cdot \frac{\log k}{k},\]
where in the last step we used $k\geq \tk/2\geq 10^{199}$.
\end{document}